\documentclass[11pt,oneside]{amsart}
\usepackage[utf8]{inputenc}
\usepackage[english,francais]{babel}
\usepackage[centertags]{amsmath}
\usepackage{amsmath}
\usepackage{amssymb}
\usepackage{mathrsfs}
\newcommand \A[1]{{\bf (#1)}}
\def\leftB{[\![}
\def\rightB{]\!]}

\usepackage{enumitem}
\usepackage{graphicx}
\usepackage{wasysym}
\usepackage{dsfont}
\newcommand{\Tr}{{{\rm Tr}}}

\newcommand{\bb}{\mathbb}

\newcommand\R{\mathbb{R}}
\newcommand\Z{\mathbb{Z}}
\newcommand\N{\mathbb{N}}

\newcommand\F{\mathcal{F}}

\DeclareMathSymbol{\leqslant}{\mathrel}{AMSa}{"36}
\DeclareMathSymbol{\geqslant}{\mathrel}{AMSa}{"3E}
\def \P{{\bb P}}

\def \R{{\bb R}}

\def \E{{\bb E}}

\def\leftB{[\![}
\def\rightB{]\!]}

\def\G{\mathcal{G}}
\def\F{\mathcal{F}}

\def \ra{\rightarrow}

\newcounter{definition}[section]
\newcommand{\mysection}{\setcounter{equation}{0} \section}

\newtheorem{lemme}{Lemma}
\newtheorem{theo}{Theorem}
\newtheorem{remark}{Remark}
\newtheorem{corol}{Corollary}

\DeclareMathOperator*{\argmax}{arg\,max}
\DeclareMathOperator*{\argmin}{arg\,min}

\usepackage{geometry}
\usepackage{color}

\begin{document}

\author{I. HONORÉ} 
\address{Universit\'e d'Evry Val d'Essonne}
\email{igor.honore@univ-evry.fr} 
\curraddr{Laboratoire de Mathématiques et Modélisation d'Évry (LaMME), 23 Boulevard de France, 91037, Evry, France. }

\title[Sharp non-asymptotic Concentration for Ergodic Approximations]{Sharp non-asymptotic Concentration Inequalities for the Approximation of the Invariant Measure of a Diffusion}

\selectlanguage{english}
\begin{abstract}
For an ergodic Brownian diffusion with invariant measure $\nu$, we consider a sequence of empirical distributions $(\nu_n)_{n \geq 1}$ associated with an approximation scheme with decreasing time step $(\gamma_n)_{n \geq 1}$ along an adapted regular enough class of test functions $f$ such that $f-\nu(f)$ is a coboundary of the infinitesimal generator $\mathcal{A}$. 
Denote by $\sigma$ the diffusion coefficient and $\varphi$ the solution of the Poisson equation $\mathcal{A}\varphi= f- \nu(f)$.
When the square norm $|\sigma^*  \nabla \varphi|^2$ lies in the same coboundary class as $f$, we establish sharp non-asymptotic concentration bounds for suitable normalizations of $\nu_n(f)-\nu(f)$. 
Our bounds are optimal in the sense that they match the asymptotic limit obtained by Lamberton and Pagès in \cite{lamb:page:02}, 
for a certain large deviation regime. In particular, this allows us to derive sharp non-asymptotic confidence intervals.
Eventually, we are able to handle, up to an additional constraint on the time steps, Lipschitz sources $f$ in an appropriate non-degenerate setting.
\end{abstract}


\date{\today}

\maketitle
\mysection{Introduction}
\setcounter{equation}{0}
\subsection{Statement of the problem}

Consider the stochastic differential equation
\begin{equation}
\label{eq_diff}
dY_t = b(Y_t) dt + \sigma(Y_t) dW_t,
\end{equation}
where $(W_t)_{t\ge 0}$ stands for a Wiener process of dimension $r\in \N$ on a given filtered probability space $(\Omega,\G,(\G_t)_{t\ge 0},\P) $, $b : \R^d \rightarrow \R^d,$ and $ \sigma : \R^d \rightarrow \R^d \otimes \R^r $ are Lipschitz continuous functions and satisfy a Lyapunov condition  (see further Assumption $\mathbf{\mathcal{L}_V}$) which provides the existence of an invariant measure $\nu$. 
Throughout the article, uniqueness of the invariant measure $\nu$ is assumed.
The purpose of this work is to estimate the invariant measure of the diffusion equation \eqref{eq_diff}.

In order to make a clear parallel with the objects we will introduce for the approximation of $\nu$, let us first recall some basic facts on $(Y_t)_{t\ge 0} $ and $\nu$.

Introduce for a bounded continuous function $f$ and $t \in \R_+$ the average occupation measure:
\begin{equation}\label{nu_t}
\nu_t(f) := \frac 1t \int_0^t f(Y_s)ds.
\end{equation}
Foremost, bear in mind the \textit{usual} ergodic theorem which holds under appropriate Lyapunov conditions (see e.g. \cite{khasminskii2011stochastic}):
\begin{equation}\label{Bhat_ergo}
\nu_t(f) \overset{{a.s.}}{\underset{t \to + \infty}{\longrightarrow}} \nu(f):=\int f d\nu.
\end{equation}
Under suitable stability and regularity conditions, Bhattacharya \cite{bhat:82} then established a corresponding Central Limit Theorem (CLT). Namely,  for all smooth enough function $f$, 
\begin{equation}\label{Bhat}
\sqrt t \big(\nu_t(f)-\nu(f)\big) \overset{{\mathcal L}}{\underset{t \to + \infty}{\longrightarrow}} {\mathcal N} \left (0,\int_{\R^d} |\sigma^*\nabla \varphi(x)|^2\nu(dx) \right),
\end{equation}
where $\varphi$ is the solution of the Poisson equation $\mathcal{A} \varphi= f-\nu(f)$ and $\mathcal{A}$ stands for the infinitesimal operator of the diffusion \eqref{eq_diff} (see \eqref{A_phi} below for more details). In the following, we say that $f$ is \textbf{coboundary} when there is a smooth solution $\varphi$ to the Poisson equation $\mathcal{A} \varphi= f-\nu(f)$. 

Identity \eqref{Bhat} is a Central Limit Theorem (CLT) whose asymptotic variance is the integral of the well known \textbf{carré du champ} (called energy but we will say, from now on, by abuse of terminology, \textit{carré du champ}),
for more precision, see \cite{bakr:gent:ledo:14} and \cite{led:99}. The \textit{carré du champ} is actually a bilinear operator defined for any smooth functions $\varphi,  \psi$ by  $ \mathbf{\Gamma} (\varphi,  \psi)  := \mathcal{A} (\varphi \cdot \psi) -   \mathcal{A} \varphi \cdot  \psi -    \varphi \cdot  \mathcal {A} \psi $, and so:
\begin{equation*}
\nu \left (\mathbf{\Gamma} \left (\varphi,\varphi \right) \right )= \int_{\R^d} \mathbf{\Gamma} (\varphi, \varphi ) \nu(dx) = -2  \int_{\R^d} \mathcal A \varphi \cdot  \varphi \ \nu(dx) =
\int_{\R^d} |\sigma^*  \nabla \varphi |^2 \nu(dx).
\end{equation*}  
Indeed, observe that $\nu \left (\mathcal A \left(\varphi  \cdot \psi \right) \right)=0$. This is a consequence of the fact that $\nu$ solves in the distributional sense the Fokker-Planck equation $\mathcal A^* \nu=0$.
Also, this observation yields that, in order to bypass solving a Poisson equation, 
a common trick consists in dealing with smooth functions of the form $\mathcal A \varphi$.

From a practical point of view, several questions appear: how to approach the process $(Y_t)_{t \geq 0}$, the integral $\nu_t$, and the deviation from the asymptotic measure appearing in \eqref{Bhat}? 
Here, the first question is addressed by considering a suitable discretization scheme with decreasing time steps, $(\gamma_k)_{k\geq 1}$. 
The integral $\nu_t$ can then be approximated by the associated empirical measure, whose deviations will be controlled in our main results.
In particular, we take advantage of the discrete analogue to \eqref{Bhat} established by Lamberton and Pag\`es in \cite{lamb:page:02} for the current approximation scheme, to derive sharp non-asymptotic bounds for the empirical measure.

We propose an approximation algorithm based on an Euler like discretization with decreasing time step, first introduced by Lamberton and Pagès in \cite{lamb:page:02} who derived related asymptotic limit theorem in the spirit of \eqref{Bhat}, and exploited as well in \cite{hon:men:pag:16} where some corresponding non-asymptotic bounds are obtained.

For the decreasing step sequence $(\gamma_k)_{k \geq 1} $ and $n\ge 0$, the scheme deriving from \eqref{eq_diff} is defined by:
\begin{equation}\label{scheme}\tag{S}
\left\{
\begin{array}{ll}
&X_{n+1} = X_n + \gamma_{n+1} b(X_n) + \sqrt{\gamma_{n+1}} \sigma(X_n) U_{n+1}, \\
&X_0 \in L^2(\Omega,\F_0,\P),
\end{array}
\right.
\end{equation}
where $(U_n)_{n \geq 1}$ is an i.i.d. sequence of random variables on $\R^r$, independent of $X_0$, and whose moments match with the Gaussian ones up to order $3$. In particular, more general innovations than the Brownian increments can be used.

 Intuitively, the decreasing steps in \eqref{scheme} allow to be more and more precise when time grows.

The empirical (random) occupation measure of the scheme is defined for all $A\in {\mathcal B}(\R^d) $ (where ${\mathcal B}(\R^d)$ denotes the Borel $\sigma $-field on $\R^d$) by:
\begin{equation}\label{measure}
\nu_n(A) :=\nu_n(\omega,A) :=\frac{\sum_{k=1}^n \gamma_k \delta_{X_{k-1}(\omega)}(A)}{\sum_{k=1}^n \gamma_k}.
\end{equation}
We are interested in the long time approximation, so we need to consider steps $ (\gamma_k)_{k \geq 1}$ such that $\Gamma_n:=\sum_{k=1}^n \gamma_k \underset{n}{\rightarrow} + \infty $.

Under suitable Lyapunov like assumptions, Lamberton and Pagès in~\cite{lamb:page:02} first proved the following ergodic result: for any $\nu -a.s.$ continuous function $f$ with polynomial growth, $\nu_n(f) \overset{a.s.}{\underset{n}{\longrightarrow}} \nu(f)=\int_{\R^d}f(x) \nu (dx)$, which is the discrete analogue of \eqref{Bhat_ergo}.

The main benefit of decreasing steps instead of constant ones is thus that the empirical measure directly converges towards the invariant one. 
Otherwise, taking $\gamma_k=h>0$ in~\eqref{scheme}, the previous ergodic theorem must be changed into: $\nu_n(f) \overset{a.s.}{\underset{n}{\longrightarrow}} \nu^h(f)=\int_{\R^d}f(x) \nu^h (dx)$,
where $\nu^h $ is the invariant measure of the scheme. 
So, an extra study must be carried out, namely the difference $\nu-\nu^h $ should be estimated. For more details about this approach we refer to \cite {tala:tuba:90}, \cite{tala:02} and the work of Malrieu and Talay \cite{MalrieuTalay}. This work first addressed the issue of deriving non-asymptotic controls for the deviations of empirical measure of type \eqref{measure} when $\gamma_k=\gamma>0 $ (constant step). 
The backbone of their approach consisted in establishing a Log Sobolev inequality, which implies Gaussian concentration, for the Euler scheme.
In whole generality, functional inequalities (such as the Log Sobolev one) are a powerful tools to get simple controls on the invariant distribution associated with the diffusion process~\eqref{eq_diff}, see e.g. Ledoux~\cite{led:99} or Bakry \textit{et al.}~\cite{bakr:gent:ledo:14}.
 Withal Log Sobolev, and Poincaré inequalities turn out to be quite  \textit{rigid} in the framework of discretization schemes like~\eqref{scheme} with or without decreasing steps.

For the CLT associated with stationary Markov chains, we refer to Gordin's Theorem (see \cite{gor:lif:78}). Note as well that the variance of the limit Gaussian law is also the  \textit{carré du champ} for discrete Poisson equation associated with the generator of the chain.

 Let us mention as well some related works.  In \cite{blow:boll:06}, Blower and Bolley establish Gaussian concentration properties for deviations of functional of the path in the case of metric space valued homogeneous Markov chains.
Non-asymptotic deviation bounds for the Wasserstein distance between the marginal distributions  and the stationary law, in the homogeneous case can be found in \cite{bois:11} (see also  Boissard and Le Gouic in \cite{bois:lego:14} for controls on the expectations of this Wasserstein distance).
The key point of these works is to demonstrate contraction properties of the transition kernel of the homogeneous Markov chain for a Wasserstein metric, which requires some continuity in this metric for the transition law involved, see e.g.~\cite{blow:boll:06}.

In the current work, we aim to establish an optimal non-asymptotic concentration inequality for $\nu_n(f)- \nu(f )$. 
%
When $| \sigma^* \nabla \varphi |^2 - \nu(| \sigma^* \nabla \varphi |^2 )$ is a coboundary, we manage to improve the estimates in \cite{hon:men:pag:16}. 
Insofar, we better the variance in the upper-bound obtained in Theorem 3 therein: when the time step is such that $\gamma_n \asymp n^{- \theta}$, for $\theta>\frac 13$, for all $n \in \N$, 
\textcolor{black}{for a smooth enough function $\varphi$ s.t. $\mathcal A \varphi = f-\nu(f)$}, under 	
suitable assumptions (further called Assumptions \A{A}),  and if $\|\sigma\|^2$ is coboundary then there exist explicit non-negative sequences 
$(\widetilde c_n)_{n\ge 1}$ and $(\widetilde C_n)_{n\ge 1}$, respectively increasing and decreasing for $n$ large enough, with $\lim_n \widetilde C_n = \lim_n \widetilde c_n =1$ 
s.t. for all $n\ge 1$ and $0<a=o(\sqrt{\Gamma_n})$:
\begin{equation*}
\P\big[ \sqrt{\Gamma_n} |\nu_n(f )-\nu(f)| \geq a \big]
=
 \P\big[ |\sqrt{\Gamma_n}\nu_n( \mathcal{A} \varphi )| \geq a \big]
  \leq 2 \widetilde C_n \exp \big( - \frac{\widetilde c_n a^2 }{2\nu(\|\sigma\|^2)\|\nabla \varphi\|_\infty^2}  
\big). 
\end{equation*}
In fact, we get below the optimal variance bound, namely the \textit{carré du champ} $ \nu(|\sigma^* \nabla \varphi |^2)$, instead of  the expression $\nu(\|\sigma\|^2)\|\nabla \varphi\|_\infty^2$ as in the previous inequality.
Up to the same previously indicated deviation threshold, $a=o(\sqrt{\Gamma_n}) $, we derive the optimal Gaussian concentration. Consequently, we are able to derive directly some sharp non-asymptotic confidence intervals.

To establish our non-asymptotic results, we use martingale increment techniques which turn out to be very robust in a rather large range of application fields.
Let us for instance mention 
the work of Frikha and Menozzi~\cite{frik:meno:12} which establishes non-asymptotic bounds for the regular Monte Carlo error associated with the Euler discretization of a diffusion until a finite time interval $[0,T] $ and for a class of stochastic algorithms of Robbins-Monro type. 
Still with martingale approach, 
Dedecker and Gou\"ezel~\cite{dede:goue:15} have 
obtained non-asymptotic deviation bounds for separately bounded functionals of geometrically ergodic Markov chains on a general state space. 
Eventually, we can refer again to the work  \cite{hon:men:pag:16} in the current setting. 

The  paper is organized as follows.
In Section \ref{Assumption_existing_result}, we state our notations and assumptions as well as some known and useful results related to our approximation scheme.
Section \ref{MR} is devoted to our main concentration results \textcolor{black}{(for a certain deviation regime that we will call \textit{Gaussian deviations})}, we also state therein several technical 
lemmas whose proofs are postponed to Section~\ref{SEC_TEC}. 
Importantly, we also provide a user's guide to the proof which emphasizes the key steps in our approach.
Section~\ref{SEC_TEC} is the technical core of the paper.
 We then discuss in Section \ref{sec_Reg} some regularity issues for the considered test functions. Namely, we recall some assumptions introduced in \cite{hon:men:pag:16} which yield appropriate regularity concerning the solution of Poisson equation. We also extend there our main results to test functions $f$ that are Lipschitz continuous, up to some constraints on the step sequence. 

\textcolor{black}{ We proceed in Section \ref{sec_optim} to the explicit optimization of the constants appearing in the concentration bound deriving from our approach. Intrinsically, this procedure conducts to two deviation regimes, the \textit{Gaussian} one up to $a=o(\sqrt{\Gamma_n}) $ and the \textit{super Gaussian} one for $a\gg\sqrt{\Gamma_n} $ which deteriorates the concentration rate. Even though awkward at first sight (see e.g. Remark \ref{LA_RQ_SUR_LES_REGIMES} below), this refinement 
turns out to be useful for some  numerical purposes, as it emphasized in  Section~\ref{SEC_NUM}.
}
We conclude there with some numerical results associated with 
a degenerate diffusion.
Some additional technical details needed in Section~\ref{sec_optim}
 are gathered in Appendix \ref{asymptotic_analysis}.

\section{Assumptions and Existing Results}\label{Assumption_existing_result}
\setcounter{equation}{0}
\subsection{General notations}
\label{notations}
For all step sequence $(\gamma_n)_{n\ge 1} $, we denote: 
\begin{equation*}
\forall \ell \in \R,\ \Gamma_n^{(\ell)} :=\sum_{k=1}^n \gamma_k^\ell, \ 
\Gamma_n := \sum_{k=1}^n \gamma_k = \Gamma_n^{(1)}.
\end{equation*}
Practically, the time step sequence is assumed to have the form: $\gamma_n \asymp \frac{1}{n^{\theta}}$ with $\theta \in (0, 1]$, where for two sequences $(u_n)_{n\in \N}, \ (v_n)_{n\in \N}$ the notation $u_n \asymp v_n$ means that $\exists n_0 \in \N, \ \exists C\geq 1$ s.t. $\forall n \geq n_0, \ C^{-1} v_n \leq u_n \leq C v_n$.
\\

We will denote by $C$ a non negative constant,  and by $(e_n)_{n \geq 1}, (\mathscr R_n)_{n \geq 1} $ deterministic generic sequences s.t. $e_n \underset{n}{\rightarrow} 0$ and $\mathscr R_n \underset{n}{\rightarrow} 1$,
that may change from line to line. The constant $C$ depends, uniformly in time, as well as the sequences $(e_n)_{n \geq 1}, (\mathscr R_n)_{n \geq 1} $,
  on known parameters appearing in the assumptions introduced in Section \ref{Hypotheses} (called \A{A} throughout the document). Other possible dependencies will be explicitly specified.
\\

In the following, for any smooth enough function $f$, for $k \in \N$ we will denote $D^kf$ the tensor of the $k^{\rm th}$ derivatives of $f$. Namely $D^k f= (\partial_{{i_1}} \hdots \partial_{{i_k}} f)_{1 \leq i_1, \hdots ,  i_k \leq d}$.
However, for a multi-index $\alpha \in \N_0^d:=(\N\cup \{0\})^d$, we set $D^\alpha f= \partial_{x_1}^{\alpha_1} \hdots \partial_{x_d}^{\alpha_d}f : \R^d \to \R$.
\\

For a $\beta$-Hölder continuous function $f:\R^d\rightarrow \R$
, we introduce the notation 
$$[ f]_\beta:=\sup_{x\neq x'}
\frac{| f(x)- f(x')|}{|x-x'|^\beta}< +\infty,$$
for its H\"older modulus of continuity. Here, $|x-x' |$ stands for the Euclidean norm of $x-x'\in \R^d $.

We denote, for $(p,m) \in \N^2$, by $\mathcal C^p(\R^d, \R^m)$ the space of $p$-times continuously differentiable functions from $\R^d $ to $\R^m$.
Besides, for $f\in {\mathcal C}^{p}(\R^d,\R^m)$, $p\in \N$, 
we define for $\beta\in (0,1] $ the Hölder modulus: 
\begin{equation*}
[ f^{(p)}]_\beta:=\sup_{x\neq x', 
|\alpha|=p}\frac{|D^\alpha f(x)-D^\alpha f(x')|}{|x-x'|^\beta}\le +\infty,
\end{equation*}
where $\alpha$ (viewed as an element of $\N^d$) is a multi-index of length $p$, i.e. $|\alpha|:=\sum_{i=1}^d \alpha_i=p $. 
\textcolor{black}{Hence, in the above definition, the $|\cdot|$ in the numerator is the usual absolute value.}
 We will as well use the notation $\leftB n,p\rightB $, $(n,p)\in (\N_0)^2, n\le p $, for the set of integers being between $n$ and $p$.
\\

From now on, we introduce for $k\in \N_0,\beta\in (0,1] $ and $m\in \{1,d,d\times r\} $ the H\"older spaces
\begin{eqnarray}\label{def_Holder}
&&{\mathcal C}^{k,\beta}(\R^d,\R^m)
:=\{  f \in {\mathcal C}^{k}(\R^d,\R^m):  \forall \alpha \in \N^d, |\alpha|\in \leftB 1,k\rightB, \sup_{x\in \R^d} |D^\alpha f(x)|<+\infty  , [f^{(k)}]_\beta<+\infty\},
\nonumber \\
&&{\mathcal C}^{k,\beta}_b(\R^d,\R^m)
:= 
 \{  f \in {\mathcal C}^{k,\beta}(\R^d,\R^m) :  \|f\|_\infty<+\infty  \}.
\nonumber \\
\end{eqnarray}
In the above definition,
for a bounded mapping $\zeta: \R^d \to \R^m$, $m \in \{1,d,d \times r\}$, we write $\| \zeta \|_\infty := \sup_{x \in \R^d} \| \zeta \zeta^*(x) \| $ with
 $\|\zeta(x)\| = \Tr \left ( \zeta \zeta^*(x) \right )^{1/2} $, where for $M\in \R^m\otimes \R^m$, $\Tr(M) $ stands for the trace of $M$.
Hence $\| \cdot \|$ is the Fr\"obenius norm 
\footnote{This notation allows to define similarly vector and matrix norms. In fact, $\R^d $ vectors can be regarded as line vectors. Then we define similarly for both cases the uniform norm $\|\cdot \|_\infty$}.

With these notations, ${\mathcal C}^{k,\beta}(\R^d,\R^m) $ stands for the subset of ${\mathcal C}^k(\R^d,\R^m) $ whose elements have bounded derivatives up to order $k$ and $\beta $-H\"older continuous $k^{\rm th}$ derivatives. For instance, the space of Lipschitz continuous functions from $\R^d $ to $\R^m $ is denoted by   ${\mathcal C}^{0,1}(\R^d,\R^m) $.\\

Eventually, for a given Borel function $f:\R^d\rightarrow E$, where $ E$ can be $\R,\ \R^d,\ \R^d\otimes \R^r, \R^d\otimes \R^d $, we set for $k\in \N_0 $:
\begin{equation*}
f_k :=f(X_k).
\end{equation*}
For $k\in \N_0$, we denote by $\F_k:={\boldsymbol{\sigma}}\big( (X_j)_{j\in \leftB 0,k\rightB}\big) $ the ${\boldsymbol{\sigma}} $-algebra generated by the $(X_j)_{j\in \leftB 0,k\rightB} $.\\

\subsection{Hypotheses}
\label{Hypotheses}
\begin{trivlist}
\item[\A{C1}] The first term of the random sequence $X_0$ is supposed to be sub-Gaussian, i.e. there is a threshold $\lambda_0 >0$ such that:
\begin{equation*}\label{cond_X0}
\forall \lambda < \lambda_0,\quad \E[\exp(\lambda |X_0|^2)] < +\infty.
\end{equation*}

\item[\A{GC}]
The innovations $(U_n)_{n\ge 1} $ form an i.i.d. sequence with law $\mu$, we also assume that $\E[U_1]=0$ and for all $(i,j,k)\in \left\{ 1,\cdots,r\right\}^3 $, $\E[U_1^i U_1^j]=\delta_{ij},\ \E[U_1^iU_1^jU_1^k]=0 $. Moreover, $(U_n)_{n\ge 1} $ and $X_0$ are independent. Eventually, $U_1$ satisfies the following standard Gaussian concentration property, i.e. for every $1-$Lipschitz continuous function $g:\R^r\rightarrow \R$ and every $\lambda>0 $:
\begin{equation*}
\E\big[\exp(\lambda g(U_1))\big] \leq \exp\big( \lambda \E[g(U_1)] + \frac{ \lambda^2}{2}\big).
\end{equation*}
In particular, Gaussian and symmetrized Bernoulli random variables (in short r.v.) satisfy this inequality.

Pay attention that a wider class of sub-Gaussian distributions could be considered. Namely, random variables for which there exists $\varpi>0 $ s.t. for all $\lambda>0 $:
\begin{equation}
\label{GEN_GC}
\E\big[\exp(\lambda g(U_1))\big] \leq \exp \big( \lambda \E[g(U_1)] + \frac{ \varpi \lambda^2}{4} \big).
\end{equation}
It is well know that this assumption yields that for all $r\ge 0$, $\P[|U_1|\ge r]\le 2\exp(-\frac{r^2}{\varpi}) $.
\\

\item[\A{C2}]
There is a positive constant $\kappa $ s.t., defining for all $x\in \R^d, \ \Sigma(x):=\sigma\sigma^*(x) $:
\begin{align*}\label{B}
\sup_{x\in \R^d} \Tr(\Sigma(x))=\sup_{x\in \R^d}\|\sigma(x)\|^2 \leq \kappa.
\end{align*}
\item[($\mathbf{\mathcal{L}_V}$)]
We consider the following Lyapunov like stability condition:
\\

There exists $V: \mathbb{R}^d \longrightarrow [v^*, +\infty[$ with $v^* >0$ s.t. 

\begin{enumerate}[label=\roman*)]

\item $V \in \mathcal{C}^2(\R^d, \R)$, $\| D^2 V \| _{\infty} < \infty$, and $\lim_{|x| \rightarrow \infty} V(x) = + \infty$.

\item There exists $C_V \in (0,+\infty)$ s.t. for all $x \in \R^d$:
\begin{equation*}\label{L}
|\nabla V(x)|^2 + |b(x)|^2 \leq C_V V(x).
\end{equation*}

\item Let $\mathcal{A}$ be the infinitesimal generator associated with the diffusion equation~\eqref{eq_diff}, defined for all $\varphi \in \mathcal C_0^2(\R^d,\R) $ and for all $x\in \R^d $ by:
\begin{equation}\label{A_phi}
\mathcal{A} \varphi(x)= b(x) \cdot \nabla \varphi(x) + \frac1{2} \Tr \big( \Sigma(x) D^2 \varphi(x)\big),
\end{equation}
where, for two vectors $v_1,v_2 \in \R^d $, the symbol $v_1 \cdot v_2$ stands for the canonical inner product of $v_1$ and $v_2 $. 
There exist $\alpha_V >0$, $\beta_V \in \R^+$ s.t. for all $x \in \R^d$,
\begin{equation*}
\mathcal{A} V(x) \leq -\alpha_V V(x) + \beta_V.
\end{equation*}
\end{enumerate}

\item[\A{U}] There is a unique invariant measure $\nu $ to equation~\eqref{eq_diff}. \\

For $\beta\in (0,1] $, we introduce:
\item[\A{T${}_\beta $}]  We choose a test function $\varphi$ for which
\begin{enumerate}
\item[i)] $\varphi $  smooth enough, i.e. $\varphi \in {\mathcal C}^{3,\beta}(\R^d,\R)$,

 \hspace{-1cm} We further assume that:
\item[
ii)] the mapping $x\mapsto \langle b(x),\nabla\varphi(x) \rangle  $ is Lipschitz continuous. 
\item[
iii)] there exists $ C_{V,\varphi}>0$ s.t. for all $x\in \R^d$ $ |\varphi(x)|\le C_{V,\varphi}(1+\sqrt{V(x)}). $
\end{enumerate}

 \vspace{0.25cm}
 \item[\A{S}]
 We assume that the sequence $(\gamma_k)_{k\ge 1} $ is \textit{small enough}, Namely, we suppose that for all
$ k \geq 1$: 
\begin{equation*}
\gamma_k \le \min \Big (\frac{1}{2\sqrt{C_V\bar c}}, \frac{\alpha_V}{2C_V \| D^2 V \|_{\infty}}\Big ).
\end{equation*}
The constraint in \A{S} means that the time steps have to be sufficiently \textit{small} w.r.t. the diffusion coefficients and the Lyapunov function.
\end{trivlist}

\begin{remark}
\label{THE_REM_COEFF}
The above condition (${\mathcal{L}_V}$) actually implies that the drift coefficient $b$ lies, out of a compact set, between two hyperplanes separated from 0. Also, the Lyapunov function is lower than the square norm.
In other words, there exist constants $K,\bar c>0$ such that for all $|x| \geq K$, 
\begin{eqnarray}\label{b_V}
|V(x)| \leq \bar c |x|^2, \ |b(x)| \leq \sqrt{C_V \bar c} |x|.
\end{eqnarray}
Observe that we have supposed \A{U} without imposing any non-degeneracy conditions. Existence of invariant measure follows from \A{${\mathcal L}_{{\mathbf V}} $}  (see~\cite{ethi:kurz:97}). 
For uniqueness, additional conditions need to be considered ((hypo)ellipticity~\cite{khasminskii2011stochastic},~\cite{pard:vere:01},~\cite{vill:09} or confluence~\cite{page:panl:12}).
\end{remark}


 \begin{remark}\label{Krylov_Priola} 
In \A{T${}_\beta $}, condition ii) \textcolor{black}{is direct if we consider Lyapunov function $V(x) \asymp 1+|x|^2$. Indeed, $\varphi$ is supposed, in condition \A{T${}_\beta $} i), to be Lipschitz continuous and so under a linear map. 
Hypothesis ii) is natural when there is a function $f \in {\mathcal C}^{1,\beta}( \R^d, \R)$ 
with $\nu(f)=0$ s.t.
\begin{equation}\label{Remark_ii_T}
\mathcal A \varphi = f.
\end{equation}
From the definition of ${\mathcal A} $ in \eqref{A_phi}, we rewrite:
\begin{equation}\label{eq_b_nabla_varphi}
\langle \nabla \varphi(x), b(x)\rangle =f(x)-\nu(f)-\frac 12 \Tr\Big(\Sigma(x) D_x^2 \varphi(x) \Big).
\end{equation}
Since the source $f$ is Lipschitz continuous, and $\sigma, D_x^2 \varphi$ are bounded and Lipschitz continuous, the left hande side of the equation \eqref{eq_b_nabla_varphi} is also Lispchitz continuous.}
%
\end{remark}

We say that assumption \A{A} holds whenever \A{C1}, \A{GC}, \A{C2}, \A{$\mathbf{{\mathcal L}_V} $}, \A{U}, \A{T${}_\beta $} for some $\beta\in (0,1] $  and \A{S} are fulfilled. Except when explicitly indicated, we assume throughout the paper that assumption \A{A} is in force.
\\

Assume the step sequence $(\gamma_k)_{k\ge 1} $ is chosen s.t. $\gamma_k\asymp k^{-\theta},\ \theta\in (0,1] $. In particular, this implies that, for any $\varepsilon \geq 0$, $\Gamma_n^{(\varepsilon)} \asymp n^{ 1-\varepsilon \theta}$ if $\varepsilon \theta < 1$, $\Gamma_n^{(\varepsilon)} \asymp \ln(n)$ if $\varepsilon \theta = 1$ and $\Gamma_n^{(\varepsilon)} \asymp1$ if $\varepsilon \theta > 1$. 

\subsection{On some Related Existing Result}\label{existing_thm}

In \cite{lamb:page:02}, Lamberton and Pagès, proved an asymptotic result with the  decreasing step scheme \eqref{scheme}. 
Precisely, they obtain the discrete counterpart of \eqref{Bhat} established in \cite{bhat:82}, emphasizing as well some discretization effects, leading to a bias in the limit law, when the time step  becomes too coarse. This last case is however the one leading to the highest convergence rates in the CLT. 
We recall here their main results, Theorem 10 of the above reference,  for the sake of completeness.
\begin{theo}\label{theo}[Asymptotic Limit Results in \cite{lamb:page:02}]
Assume \A{C2}, \A{$\mathbf{{\mathcal L}_V} $}, \A{U} hold. If $\E[U_1]=0,\ \E[U_1^{\otimes 3}]=0$, we get the following limit results where $\nu_n$ stands for the empirical measure defined in \eqref{measure}.
\begin{enumerate}
\item[(a)] \textit{Fast decreasing step.} If $ \theta \in (\frac 13, 1]$  and $\E[|U_1|^6] < +\infty$, then, for all function $\varphi \in \mathcal{C}^{2,1}(\R^d,\R)\cap \mathcal{C}^{3}(\R^d,\R)$, one has:
\begin{equation*}
\sqrt{\Gamma_n} \nu_n (\mathcal{A} \varphi) \underset{n \to \infty}{\overset{\mathcal{L}}{\longrightarrow}} \mathcal{N} \big(0, \int_{\R^d} | \sigma^* \nabla \varphi|^2 d\nu \big).
\end{equation*}
\item[(b)] \textit{Critical decreasing step.} If $\theta= \frac 13$ and if $\E[|U_1|^8] < +\infty$, then for all function $\varphi \in  \mathcal{C}^{3,1}(\R^d,\R) \cap \mathcal C^4(\R^d,\R)  $, one gets: 
\begin{equation*}
\sqrt{\Gamma_n} \nu_n (\mathcal{A} \varphi) \underset{n \to \infty}{\overset{\mathcal{L}}{\longrightarrow}} \mathcal{N}\big(\widetilde{\gamma} m, \int_{\R^d} | \sigma^*  \nabla \varphi|^2 d\nu \big),
\end{equation*}
where 
\begin{eqnarray*}
\ \widetilde \gamma &:=&\lim_{n \to + \infty} \frac{\Gamma_n^{(2)}}{\sqrt{\Gamma_n}}, \ m := -  \int_{\R^d}  \Big( \Tr \big ( \frac{1}{2} D^2 \varphi(x) b(x) ^{\otimes 2} \big )+ \Phi_4(x) \Big) \nu(dx), 
\\
\Phi_4(x) &:=& \int_{\R^r} \Tr \Big( \frac{1}{2}  D^3 \varphi(x) b(x)( \sigma(x) u ) ^{\otimes 2}  + \frac{1}{24} D^4 \varphi(x) ( \sigma(x) u)^{\otimes 4} \Big) \mu (du),
\end{eqnarray*}
recalling that $\mu$ denotes the law of the i.i.d.  innovations $(U_k)_{k\ge 1} $
\footnote{With our tensor notations, 
$D^2 \varphi(x) b(x) ^{\otimes 2}  \in (\R^d)^{ \otimes 2}$, $D^3 \varphi(x) b(x)( \sigma(x) u ) ^{\otimes 2} \in (\R^d)^{ \otimes 3}$,  and $D^4 \varphi(x) ( \sigma(x) u)^{\otimes 4} \in  (\R^d)^{ \otimes 4}.$ 
}.

\item[(c)] \textit{Slowly decreasing step.} If $\theta \in (0, \frac 13)$ and if $\E[|U_1|^8] < +\infty$, then for all globally Lipschitz function $\varphi \in\mathcal{C}^{3,1}(\R^d,\R) \cap \mathcal C^4(\R^d,\R)  $, one gets: 
\begin{equation*}
\frac{\Gamma_n}{\Gamma_n^{(2)}} \nu_n ({\mathcal A} \varphi ) \overset{\P}{\longrightarrow}  m.
\end{equation*}
\end{enumerate}
\end{theo}
It is possible to relax the boundedness condition on $\sigma $ in \A{C2}, considering $\lim_{|x|\rightarrow +\infty}\frac{|\sigma^*  \nabla \varphi(x)|^2}{V(x)}=0 $ (strictly sublinear diffusion) in case \textit{(a)} and $\sup_{x\in \R^d}\frac{|\sigma^*  \nabla \varphi(x)|^2}{V(x)}<+\infty $ (sublinear diffusion) in case \textit{(b)}. We refer to Theorems 9 and 10 in~\cite{lamb:page:02} for additional details.

\begin{remark}
First of all, observe that the normalization is the same as for \eqref{Bhat}. It is the square root of the considered running time, namely  $t$ for the diffusion and $\Gamma_n$ for the scheme. In other words, a CLT is still available for the discretization procedure.
However, by choosing a critical time step, i.e. for the fast convergence $\theta= \frac 13$, a bias is begot. It can be regarded as  a discretization effect.
Note that, for all $\theta \geq \frac 13$, any step leads to the same asymptotic variance, namely the \textit{carré du champ}, $ \int_{\R^d} | \sigma^*  \nabla \varphi|^2 d\nu $, like in \cite{bhat:82}.
However, for the slow decreasing step, $\theta < \frac 13$, the discretization effect is prominent and ``hides" the CLT.
\end{remark}

Let us also mention the work of Panloup \cite{panl:08:AAP}, where under similar assumptions for stochastic equation driven by a Lévy process, the convergence of the decreasing time step algorithm towards the invariant measure of the stochastic process is established (see also \cite{panl:08} for the CLT associated with square integrable Lévy innovations).
\\

In the current diffusive context, i.e. under \A{A}, some non-asymptotic results were successfully established in  \cite{hon:men:pag:16}. 
It was as well observed there that if we slacken the regularity of the test function $\varphi$, a new bias looms.
For $ \varphi \in \mathcal C^{3,\beta}(\R^d,\R)$ 
with $\beta \in (0,1)$, if $\theta= \frac{1}{2+\beta}$ then $\sqrt{\Gamma_n} \nu_n(\mathcal A \varphi)$ exhibits deviations similar to the ones of a biased normal law with a different bias than in Theorem \ref{theo}, c.f. Theorems 2, 3, 4 and 5 in \cite{hon:men:pag:16}. 
When $\beta=1 $, the two biases correspond.
We willingly shirk any discussion about bias appearance, which is discussed in the formerly mentioned article.
Our target is to refine Theorem 4 in \cite{hon:men:pag:16} that we recall:

\begin{theo}\label{THM_COBORD} [Non-asymptotic concentration inequalities in \cite{hon:men:pag:16}]
Assume \A{A} holds, 
if there is $\vartheta \in \mathcal C^{3,\beta}(\R^d,\R)$ satisfying \A{T${}_\beta$} and s.t.
 $$\mathcal A \vartheta=\|\sigma\|^2- \nu(\|\sigma\|^2).$$
For $\beta\in (0,1] $ and $ \theta \in (\frac{1}{2+\beta}, 1 ]$, there exist two explicit monotonic sequences $\tilde c_n\le 1\le \tilde C_n,\ n\ge 1$,   with $\lim_n \tilde C_n = \lim_n \tilde c_n =1$
such that for all $n\ge 1$ and $a>0$:
 \begin{eqnarray}\label{DEV_HMP}
&& \P\big[ |\sqrt{\Gamma_n}\nu_n( \mathcal{A} \varphi )| \geq a \big] 
   \leq 2\, \widetilde C_n \exp \big (\! - \frac{\widetilde c_n}{2\nu(\|\sigma\|^2)\|\nabla \varphi\|_\infty^2}\, \Phi_n(a)
\big), \nonumber \\
&&\Phi_n(a):=\Big[ \!\Big(\! a^2\big( 1-\frac{2}{1+\sqrt{1+4\,\bar c_n^3\,\frac{\Gamma_n}{a^2}}}\big)  \Big)\! \vee \! \Big(\!a^{\frac 43}\Gamma_n^{\frac{1}{3}} \bar c_n\big(1-\frac 23\bar c_n \big (\frac{\Gamma_n }{a^{2}} \big)^{\frac 13} \big)_+ \Big) \Big],
\end{eqnarray}
where $x_+=\max(x,0)$ and  $ \bar c_n:=\left(\frac{[\varphi]_1}{[\vartheta]_1} \right)^{2/3}\nu(\|\sigma\|^2)\|\sigma\|_\infty^{-2/3}\check c_n$ with $\check c_n $ being an explicit nonnegative sequence s.t. $\check c_n \downarrow_n 1 $.
\end{theo}

\begin{remark}\label{LA_RQ_SUR_LES_REGIMES}
Observe that two regimes compete in the above bound. 
From now on, we refer to \textbf{Gaussian deviations} when $\frac {a}{\sqrt{\Gamma_n}} \underset{n}{\rightarrow} 0$. In this case, asymptotically the right hand side of the inequality \eqref{DEV_HMP} is 
$2 \exp \big( - \frac{ a^2}{2\nu(\|\sigma\|^2)\|\nabla \varphi\|_\infty^2} \big )$. 
In other words, the empirical measure is sub-Gaussian with asymptotic variance equals to $\nu(\|\sigma\|^2) \| \nabla \varphi \|_{\infty}^2$ which is an upper-bound of the \textit{carré du champ}, $\nu(|\sigma^*  \nabla \varphi |^2)$ (asymptotic variance in the limit theorem). Thus, this is  not fully satisfactory.
Throughout the article, we refer to \textbf{super Gaussian deviations} when $\frac {a}{\sqrt{\Gamma_n}} \underset{n }{\rightarrow} + \infty$.
In this case, a subtle phenomenon appears: 
the right hand side gives a super Gaussian regime.
In particular, the term in the exponential of the r.h.s. of \eqref{DEV_HMP}  is bounded from above and below by $a^{4/3} \Gamma_n^{1/3}$. 
We anyhow emphasize that Theorem 2 in \cite{hon:men:pag:16} provides a non-asymptotic Gaussian concentration for all deviation regimes:
\begin{equation}\label{Theo_1_HMP}
\P\big[ |\sqrt{\Gamma_n}\nu_n( \mathcal{A} \varphi )| \geq a \big] 
\leq 2\, \widetilde C_n \exp \big (\! - \frac{\widetilde c_n a^2}{2 \|\sigma\|_{\infty}^2\|\nabla \varphi\|_\infty^2}
\big). 
\end{equation}
In particular, for \textit{super Gaussian deviations}, the deviation \eqref{Theo_1_HMP} is asymptotically better.
\textcolor{black}{However, it had already been observed in \cite{hon:men:pag:16} that the bound \eqref{DEV_HMP} turned out to be useful for numerical purposes as it led to bounds closer to the empirical realizations. We will derive in Theorem \ref{THM_CARRE_CHAMPS_super_Gaussian} of Section \ref{sec_optim} a deviation bound similar to \eqref{DEV_HMP} with an improved variance bound. Namely, we succeed to replace $ \nu(\|\sigma\|^2) \| \nabla \varphi \|_{\infty}^2$ by the \textit{carré du champ}. We then observe  in the numerical results of Section \ref{SEC_NUM} that the associated deviation bounds match rather precisely those of the empirical realizations.}

We will also employ the terminology of \textbf{intermediate Gaussian deviations} when $\frac {a}{\sqrt{\Gamma_n}} \underset{n}{\rightarrow} C>0$. For this regime, we keep a Gaussian regime with deteriorated constants.
\textcolor{black}{Again, we first deal with \textit{Gaussian deviations}, and we postpone the study of  \textit{super Gaussian deviations} to Section \ref{sec_optim}.} 
\end{remark}

\begin{remark}
Actually, in the proof of Theorem \ref{THM_COBORD}, we can only use a map $\vartheta$ satisfying assumption \A{T${}_\beta$} s.t. $\mathcal A \vartheta \geq |\sigma^*  \nabla \varphi |^2 - \nu( |\sigma^*  \nabla \varphi |^2)$.
However, this inequality is equivalent to the coboundary condition $\nu-a.s.$. In fact, we set the function $f:=\mathcal A \vartheta - |\sigma^*  \nabla \varphi |^2 + \nu( |\sigma^*  \nabla \varphi |^2) \geq 0$,
and 
\begin{equation}
\nu(f) = \nu(\mathcal A \vartheta) =0.
\end{equation}
As f is continuous and non negative, $f=0 \ \nu-a.s.$ Hence $\mathcal A \vartheta = |\sigma^*  \nabla \varphi |^2 - \nu( |\sigma^*  \nabla \varphi |^2)\ \nu-a.s.$
\end{remark}

\section{Main results}
\label{MR}
\setcounter{equation}{0}

Our main contribution consists in establishing a concentration inequality whose variance matches asymptotically the \textit{carré du champ}, see \eqref{Bhat} and Theorem \ref{theo} in what we called the regime of \textit{Gaussian deviations}. 
In the numerical part of \cite{hon:men:pag:16}, we see that changing the bound $\nu(\|\sigma\|^2) \| \nabla \varphi\|_{\infty}^2$ by the \textit{carré du champ}, leads to bounds much closer to the realizations. 
Here, we state a simple and ``sharp" inequality.

\begin{theo}[Sharp non-asymptotic deviation results]\label{THM_CARRE_CHAMPS}
Assume \A{A} is in force. Suppose that there exists $\vartheta \in \mathcal C^{3,\beta}(\R^d,\R)$ satisfying \A{T${}_\beta $} for some $\beta\in(0,1] $ s.t. 
\begin{equation}\label{Poisson}
\mathcal A \vartheta= |\sigma^*  \nabla \varphi |^2 -\nu(|\sigma^*  \nabla \varphi |^2).
\end{equation}
Then, for $ \theta \in (\frac{1}{2+\beta}, 1 ]$, there exist explicit non-negative sequences 
$( c_n)_{n\ge 1}$ and $( C_n)_{n\ge 1}$, respectively increasing and decreasing for $n$ large enough, with $\lim_n  C_n = \lim_n c_n =1$ 
s.t. for all $n\ge 1$, $a>0$ satisfying $\frac{a }{\sqrt{\Gamma_n}} \rightarrow 0$ (\textit{Gaussian deviations}), the following bound holds:
\begin{equation*}
\P\big[ |\sqrt{\Gamma_n}\nu_n( \mathcal{A} \varphi )| \geq a \big] 
\leq 2\,  C_n \exp \big(\! -  c_n\frac{a^2}{2 \nu(|\sigma^*  \nabla \varphi |^2)}
\big). 
\end{equation*}
\end{theo}

\begin{remark}
We obtain the optimal Gaussian bound with the \textit{carré du champ} as variance which corresponds to CLT. 
This is asymptotically the sharpest result that we can expect. 
This inequality is very important for confidence intervals, as in this context $a$ is supposed to be ``small", i.e. bounded.
Under suitable regularity assumptions on $f$, which guarantee that the function $\varphi $ solving $\mathcal A \varphi= f-\nu(f) $ satisfies \A{T${}_\beta $} for some $\beta\in (0,1] $, it readily follows from Theorem \ref{THM_CARRE_CHAMPS} that:
\begin{equation*}
\P \Big [\nu(f)\in \big[\nu_n(f)-\frac{a}{\sqrt{\Gamma_n}},\nu_n(f)+\frac{a }{\sqrt{\Gamma_n}} \big] \Big]\ge 1-2  C_n\exp \big(-  c_n\frac{a^2}{2\nu(|\sigma^*  \nabla \varphi |^2)}
\big).
\end{equation*}
The conditions on $f$ that lead to the required smoothness on $\varphi $ and $\vartheta $ are discussed in Section \ref{sec_Reg} (see in particular Theorem \ref{TH_BOOTSTRAP} and Corollary \ref{COROL_BOOTSTRAP_CC}). Briefly, it suffices to consider that, additionally to \A{C2} and \A{${\mathbf{{\mathcal L}_V} } $}, $\Sigma$ is also uniformly elliptic,  $b\in {\mathcal C}^{1,\beta}(\R^d,\R^d),\sigma \in  {\mathcal C}^{1,\beta}(\R^d,\R^d\otimes \R^d)$ and that the source $f\in {\mathcal C}^{1,\beta}(\R^d,\R^d)$. This last assumption on $f$ can be weakened to Lipschitz continuous (see Theorem \ref{THEO_CTR_LIP}) with some restriction on the steps.

\end{remark}


\subsection{User's guide to the proof}\label{USER_GUIDE}
Recall that, for a fixed given $n \in \N$ and $\varphi \in \mathcal C^{3,\beta}(\R^d,\R)$, we want to estimate the quantity $$\P[\sqrt{\Gamma_n} |\nu_n(\mathcal A \varphi)| \geq a ], \ \forall a >0, $$
where $\nu_n(\mathcal A \varphi )=\frac 1{\Gamma_n} \sum_{k=1}^n \gamma_k \mathcal A \varphi(X_{k-1})$. We focus below on the term $\P[\sqrt{\Gamma_n} \nu_n(\mathcal A \varphi) \geq a ] $. Indeed, the contribution $\P[\sqrt{\Gamma_n} \nu_n(\mathcal A \varphi) \leq -a] $ can be handled by symmetry.

The first step of the proof consists in writing $\big(\mathcal A \varphi(X_{k-1})\big)_{k\in \leftB 1,n \rightB }$ with a splitting method to isolate the terms depending on the current innovation $U_k$ for $\mathcal A \varphi(X_{k-1})$. This is done in Lemma \ref{decomp_nu} below.
Precisely, for all $k \in \leftB 1,n \rightB$ and $\varphi \in \mathcal C^{3,\beta}( \R^d, \R)$ we prove that:
\begin{eqnarray}\label{Taylor}
\varphi ( X_k) - \varphi ( X_{k-1})  &=& \gamma_k \mathcal{A} \varphi(X_{k-1}) + \textcolor{black}{\gamma_k \int_0^1 \langle \nabla  \varphi (X_{k-1} + t \gamma_k b_{k-1})-\nabla \varphi(X_{k-1}),b_{k-1}\rangle dt}
\nonumber \\
&& + \frac 12 \gamma_k\, \Tr\Big( \big(D^2 \varphi ( X_{k-1} + \gamma_k b_{k-1} ) - D^2\varphi( X_{k-1}) \big) \Sigma_{k-1}\Big) +  \psi_k(X_{k-1}, U_k),
\end{eqnarray}
where
\begin{eqnarray}
\label{DEF_PSI_K} 
 \psi_k(X_{k-1}, U_k) &=&\sqrt{\gamma_k} \sigma_{k-1} U_k \cdot \nabla \varphi( X_{k-1} +\gamma_k b_{k-1} ) \nonumber  \\
& &+\; \gamma_k \int_0 ^1 ( 1-t) \Tr\Big ( D^2\varphi ( X_{k-1} +\gamma_k b_{k-1} + t \sqrt{\gamma_k} \sigma_{k-1} U_k ) \sigma_{k-1}U_k\otimes U_k \sigma_{k-1}^* 
\nonumber \\
&&- D^2\varphi ( X_{k-1} +\gamma_k b_{k-1} ) \Sigma_{k-1} \Big )dt.
\end{eqnarray}

 Observe that the term $\psi_k(X_{k-1},U_k)$ in the r.h.s. of \eqref{Taylor} is the only term containing the current innovation $U_k$. 
Thus, the mapping $u \mapsto \psi_k(X_{k-1},u)$ is Lipschitz continuous, because $\varphi$ is.
 
  This property is crucial to proceed with a martingale increment technique. 
Indeed, introducing the compensated increment $\Delta_k(X_{k-1},U_k):= \psi_k(X_{k-1},U_k) - \E[ \psi_k(X_{k-1},U_k) | \F_{k-1}]$, assumption \A{GC} allows to derive: 
  \begin{equation}\label{CTR_EC}
  \forall \lambda >0, \ \E[\exp(-\lambda \Delta_k(X_{k-1},U_k))|\F_{k-1}]\le \exp \big (\frac{\lambda^2[\psi(X_{k-1},\cdot)]_1^2}{2} \big). 
  \end{equation}
  The corner stone of the proof is then to apply recursively this control to the martingale $M_m:= \sum_{k=1}^m \Delta_k(X_{k-1},U_k)$, $m \in \leftB 1,n \rightB$.
\\

To control the deviation, the first step is an exponential inequality which combined to \eqref{Taylor} yields:
\begin{eqnarray}\label{Exp_Markov}
\P[\sqrt{\Gamma_n} \nu_n(\mathcal A \varphi) \geq a ] &\leq& \exp \big (-\frac{a \lambda}{\sqrt \Gamma_n} \big ) \E\Big[ \exp(\lambda \nu_n({\mathcal A} \varphi))\Big]\notag\\
& \leq& \exp \big (-\frac{a \lambda}{\sqrt \Gamma_n} \big ) \E \Big [\exp \big ( -\frac{\lambda q M_n}{\Gamma_n} \big ) \Big ]^{1/q} \mathcal R_n,
\end{eqnarray}
 for $\lambda >0$, $q>1$ and $  \mathcal  R_n$ is a remainder (whose behaviour is investigated in Lemma \ref{rn_choice}).
The main contribution in the above equation is the one involving $M_n$ which can be analyzed thanks to \eqref{CTR_EC}.
Namely:
\begin{eqnarray}
 \E \Big [\exp \big (- \frac{ q \lambda}{\Gamma_n} M_n \big) \Big ] &=& 
\E \Big [\exp \big ( -\frac{ q \lambda}{\Gamma_n} M_{n-1}  \big) \E \Big [\exp \big ( -\frac{ q \lambda}{\Gamma_n} \Delta_n(X_{n-1},U_n) \big) \Big |\F_{n-1} \big ] \Big ] 
\nonumber\\
&\le &\E \Big [\exp \big (- \frac{ q \lambda}{\Gamma_n} M_{n-1}  \big) \exp \big (\frac{q^2\lambda^2[\psi(X_{n-1},\cdot)]_1^2}{2\Gamma_n^2}  \big) \Big ].\label{PREAL_LIP}
\end{eqnarray}
A first approach	in \cite{hon:men:pag:16}, in order to iterate the estimates involving the conditional expectations, consisted in bounding uniformly $[\psi(X_{n-1},\cdot)]_1 \le \sqrt{\gamma_n} \| \sigma \|_{\infty} \| \nabla \varphi \|_{\infty}$ (which is easily deduced from \eqref{Taylor}). Iterating the procedure led to the estimate 
\begin{equation}
\label{BD_GROSS}
\P[\sqrt{\Gamma_n} \nu_n(\mathcal A \varphi) \geq a ] \leq  \exp \big (-\frac{a \lambda}{\sqrt \Gamma_n} \big) \exp \big( \frac{ q \lambda^2}{2\Gamma_n}\|\sigma\|_{\infty}^2 \| \nabla \varphi \|_{\infty}^2 \big) \mathcal R_n.
\end{equation}
Optimizing over $\lambda$, letting as well $q\downarrow_n 1 $ in a suitable way, gives the deviation upper-bound $ C_n \exp ( - c_n \frac{a^2}{2\|\sigma\|_{\infty}^2 \| \nabla \varphi \|_{\infty}^2} )$, with $ C_n,  c_n>0$ respectively increasing and decreasing to $1$ with $n$ (see Theorem 2 of \cite{hon:men:pag:16} for details).
\\

To obtain the expected variance corresponding to the \textit{carr\'e du champ} $\nu( | \sigma^*\nabla \varphi |^2)$,
the key point is to control finely the Lipschitz modulus of $ \psi_k(X_{k-1}, \cdot)$ in \eqref{CTR_EC}, \eqref{PREAL_LIP}.

From \eqref{Taylor}, we get the following simple expression of the derivative $ \nabla_u \psi_k(X_{k-1}, u)|_{u=U_k}=  \sqrt{\gamma_k} \sigma_{k-1}^* \nabla \varphi(X_{k})$. Hence, there is a remainder term $\mathcal R(\gamma_k,X_{k-1},U_k)$ \textcolor{black}{and a constant $C_{\eqref{EXPR_GRAD}}=C_\eqref{EXPR_GRAD}(\A{A})>0$} s.t.
\begin{equation}
\label{EXPR_GRAD}
|\nabla_u \psi_k(X_{k-1}, u) |^2|_{u=U_k}= \gamma_k |\sigma_{k-1}^* \nabla \varphi_{k-1}|^2 \textcolor{black}{+C_\eqref{EXPR_GRAD}\gamma_k^2 \sqrt{V_{k-1}}}+ \mathcal R(\gamma_k,X_{k-1},U_k),
\end{equation}
for more details see \eqref{omega1_LIP} below.

In order to exhibit for each evaluation	of the conditional expectations in \eqref{CTR_EC}, the contribution $\nu( | \sigma^* \nabla \varphi |^2)$, we use the auxiliary Poisson problem:
\begin{equation}\label{Poisson_eq}
\mathcal A \vartheta = | \sigma^* \nabla \varphi |^2 - \nu(| \sigma^* \nabla \varphi |^2).
\end{equation}
 We then write for the main term to control in \eqref{Exp_Markov}, 
\begin{equation}\label{Mn_ineq_T1T2}
\E[\exp(- \frac{\lambda q M_n}{\Gamma_n}) ] \le {\mathscr T}_1 ^{\frac{1}{\rho}} \textcolor{black}{{\mathscr T}_2^{\frac {\rho-1}{\hat q \rho}}{\mathscr T}_3^{\frac {\rho-1}{\hat p\rho}}},
\end{equation}
 for $\rho>1$, \textcolor{black}{$\hat p, \hat q>1$ s.t. $\frac 1 {\hat p} + \frac{1}{\hat q}=1$} where:
\begin{eqnarray}\label{Def_T_S}
{\mathscr T}_1 &:=& \E\exp \big (- \rho \frac{q\lambda}{\Gamma_n}M_n-\frac{\rho^2 q^2 \lambda^2}{2\Gamma_n^2}\sum_{k=1}^n \gamma_k{\mathcal A}\vartheta(X_{k-1})- C_{\eqref{EXPR_GRAD}} \gamma_k^{2} \sqrt{V_{k-1}}  \big),
\nonumber \\
{\mathscr T}_2 &:=& \E\exp \big (
\frac{\lambda^2 q^2 \rho^2 \hat q }{2(\rho-1)\Gamma_n^2}\sum_{k=1}^n \gamma_k{\mathcal A}\vartheta(X_{k-1})\big),
\nonumber \\
{\mathscr T}_3 &:=& \E\exp \big (
\frac{\lambda^2 q^2 \rho^2 \hat p}{2(\rho-1)\Gamma_n^2}\sum_{k=1}^n C_{\eqref{EXPR_GRAD}}\gamma_k^{2}  \sqrt {V_{k-1}}\big).
\end{eqnarray}
Exploiting \eqref{Poisson_eq}, we can now rewrite
\begin{eqnarray}\label{T1_user}
{\mathscr T}_1 &=& \E\exp \bigg ( -\rho \frac{q\lambda}{\Gamma_n}M_n-\frac{\rho^2 q^2 \lambda^2}{2\Gamma_n^2}\sum_{k=1}^n \Big ( \gamma_k \big  [ |\sigma^*  \nabla \varphi (X_{k-1}) |^2-\nu(| \sigma^*  \nabla \varphi|^2)  \big] +C_{\eqref{EXPR_GRAD}} \gamma_k^{2} \sqrt{V_{k-1}}  \Big )  \bigg)
\nonumber \\
 &=& \exp \Big ( \frac{\rho^2 q^2 \lambda^2}{2\Gamma_n}\nu(| \sigma^*  \nabla \varphi|^2)   \Big)
\E\exp \Big (- \rho \frac{q\lambda}{\Gamma_n}M_n-\frac{\rho^2 q^2 \lambda^2}{2\Gamma_n^2} \sum_{k=1}^n \big ( \gamma_k |\sigma^*  \nabla \varphi (X_{k-1})|^2 +C_{\eqref{EXPR_GRAD}} \gamma_k^{2} \sqrt{V_{k-1}}  \big )   \Big).
\nonumber \\
\end{eqnarray}
The first term in the above r.h.s. yields the expected variance when we optimize over $\lambda$ for $q$ and $\rho$ going to $1$, which is the case in the regime of so called \textit{Gaussian deviations} in Theorem \ref{THM_CARRE_CHAMPS}. It improves the previous bound \eqref{BD_GROSS}.
Introduce now  for $m \in \leftB 1,n\rightB$,
\begin{equation}\label{Def_S}
S_m := \exp \Big ( - \frac{\rho q\lambda}{\Gamma_n}M_m-\frac{\rho^2 q^2 \lambda^2}{2\Gamma_n^2}\sum_{k=1}^m \big ( \gamma_k |\sigma^*  \nabla \varphi (X_{k-1})|^2 +C \gamma_k^{2} \sqrt{V_{k-1}}  \big )   \Big).
\end{equation}
Bringing to mind that $M_m= \sum_{k=1}^m \Delta_k(X_{k-1},U_k)$, where $\E[ \Delta_k(X_{k-1},U_k) | \F_{k-1} ] = 0$ and $[\Delta_k(X_{k-1}, \cdot)]_1 = [\psi_k(X_{k-1}, \cdot)]_1$, we get from \eqref{EXPR_GRAD}, that, up to the remainder term $(\mathcal R(\gamma_k,X_{k-1},U_k))_{k \in \leftB 1,n \rightB}$, $S_m$ can be viewed as a super martingale (see Lemma \ref{CTR_T1_N} for details).  
We actually rigorously show that, in the Gaussian regime (i.e. for $\frac a {\sqrt{\Gamma_n}} \to 0$), for $\theta \in (1/3,1)$ $$\E[S_n]^{\frac 1{\rho q}} \le  {\mathscr R}_n \underset{n \to + \infty}{\longrightarrow} 1.$$ For $\theta=1$, or for \textit{super Gaussian deviations}  (i.e. for $\frac a {\sqrt{\Gamma_n}} \to + \infty$, see Section \ref{sec_optim}) with $\theta \in (1/3,1)$ we get:  
\begin{equation*}
\E[S_n]^{\frac 1{\rho q}} \le  {\mathscr R}_n \exp \Big (  \big(\frac{\rho q \lambda^2}{\Gamma_n}+\frac{ \rho^3 q^3 \lambda^4}{(\rho-1)\Gamma_n^3}  \big) e_n \Big ) ,
\end{equation*}
 where $e_n>0$ decreases to $0$ with $n$ and $ {\mathscr R}_n>0$ is still going to $1$ with $n$.
The difficulty in the above control is that the optimized $\lambda$ also depends on $n$ and $\rho$ (see  \eqref{lambda} below).
\\

The second term $\mathscr T_2$ is estimated directly repeating the arguments of the proof of Theorem 2 in \cite{hon:men:pag:16} which are recalled above (see equations \eqref{Exp_Markov} to \eqref{BD_GROSS}). We apply the previous martingale increment technique that previously led to \eqref{BD_GROSS}. Denoting by $M_n^\vartheta$ the martingale associated with the $\big(\psi_k^\vartheta(X_{k-1},U_k)\big)_{k\in \leftB 1,n\rightB}$ deriving from the expansion of $\mathcal A \vartheta$ similarly to \eqref{Taylor}, we obtain:
\begin{equation}\label{T2_user}
 \mathscr T_2 \leq  \E \Big [\exp \big (- \frac{\lambda^2 q^2 \rho^2 \bar q M_n^{\vartheta}}{2(\rho-1)\Gamma_n^2} \big ) \Big ]^{1/\bar q} \mathcal R_n^{\vartheta} 
\leq \exp \Big ( \frac{\lambda^4 q^4 \rho^4 \bar q}{8(\rho-1)^2\Gamma_n^3} \|\sigma\|_{\infty}^2 \|\nabla \vartheta\|_{\infty}^2 \Big)   \mathcal R_n^{\vartheta},
\end{equation}
for $\bar q>1$ and where the superscript $\vartheta$ means that we only need to replace $\varphi$ by $\vartheta$ in the previous definitions. 
Like in \eqref{Exp_Markov}, $  \mathcal R_n^{\vartheta}$ is here a remainder.

\textcolor{black}{The third component $\mathscr T_3$ is first controlled by Jensen inequality (over the exponential function and the measure is $\frac{1}{\Gamma_n^{(3)}}\sum_{k=1}^n \gamma_k^{2} \delta_k$):
\begin{equation}\label{ineq_T3}
{\mathscr T}_3 \leq \frac{1}{\Gamma_n^{(2)}}\sum_{k=1}^n \gamma_k^{2} 
\E \exp \big (
\frac{\lambda^2 q^2 \rho^2 \hat p \Gamma_n^{(2)}}{2(\rho-1)\Gamma_n^2}C 
\sqrt {V_{k-1}}\big).
\end{equation}
For the control of this term (as well for remainders from Taylor expansion in \eqref{Taylor} and in Lemma \ref{decomp_nu}), we recall a useful result from \cite{hon:men:pag:16} (see Proposition 1 therein). 
Under \A{A}, there is a constant $c_V:=c_V(\A{A}) >0$ such that for all $\lambda\in[0,c_V] $, $\xi \in [0,1]$: 
\begin{equation}\label{expV_int}
I_V^\xi:=\sup_{n \geq 0 } \E[\exp(\lambda V_n^{\xi})] < + \infty.
\end{equation}
We  also refer to Lemaire (see~\cite{lemaire:hal-00004266}) for additional integrability results of the Lyapunov functions in a more general framework.
The identity \eqref{ineq_T3} is handled by Young inequality
\begin{equation*}
{\mathscr T}_3 \leq \frac{ \exp \big ( \frac{1}{2c_V} (\frac{\lambda^2 q^2 \rho^2 \hat p \Gamma_n^{(2)}}{2(\rho-1)\Gamma_n^2}C )^2 \big )}{\Gamma_n^{(2)}}\sum_{k=1}^n \gamma_k^{2}  
 \E\exp \big (c_V V_{k-1}\big)
 =  \frac{ \exp \big ( \frac{\lambda^4}{\Gamma_n^3} e_n \big )}{\Gamma_n^{(2)}}\sum_{k=1}^n \gamma_k^{2}  
 \E\exp \big (c_V V_{k-1}\big),
\end{equation*}
with $\hat{p} \to_n + \infty$ s.t. for fixed $\rho,q>1$, $e_n =  \frac{1}{2c_V} (\frac{ q^2 \rho^2 \hat p }{16 c_V^2(\rho-1)^2}\frac{\Gamma_n^{(2)}}{\Gamma_n^2}C)^2 \rightarrow_n 0$, note that for all $\theta \in (\frac{1}{3},1]$, $\frac{\Gamma_n^{(2)}}{\sqrt{\Gamma_n}} \to_n0$. 
We obtain then by \eqref{expV_int}:
\begin{equation}\label{T3_user}
{\mathscr T}_3^{\frac{\rho-1}{\hat p \rho}} \leq 
  \exp \big ( \frac{\lambda^4}{\Gamma_n^3} e_n \big )(I_V^1)^{\frac{\rho-1}{\hat p \rho}}= \mathscr R_n \exp \big ( \frac{\lambda^4}{\Gamma_n^3} e_n \big ) ,
\end{equation}
for $\hat p= \hat p(n) \to_n+ \infty$.
}
 
 Eventually, by \eqref{Exp_Markov}, \eqref{T1_user}, \eqref{T2_user} and \eqref{T3_user} with the different controls of $\E[S_n]$ (see Lemma \ref{CTR_T1_N}): 
\begin{equation}\label{Ineq_P}
\P[\sqrt{\Gamma_n} \nu_n(\mathcal A \varphi) \geq a ] \leq \exp \Big (- \frac{ a\lambda}{\sqrt{\Gamma_n}}+\frac{\lambda^2}{\Gamma_n}A_n(\rho)+\frac{\lambda^4}{\Gamma_n^3}B_n (\rho) \Big){\mathscr R}_n,
\end{equation}
with ${\mathscr R}_n \rightarrow 1$, $A_n(\rho):= \rho (\frac{q \nu(| \sigma^*  \nabla \vartheta |^2)}{2}+e_n), B_n:=\frac{\rho^3}{\rho-1} \frac{q^3 \hat q}{4} ( \frac{\bar q \|\sigma \|_{\infty}^2 \|\nabla \vartheta \|_{\infty}^2}{2}+ e_n)$ for $e_n >0$ decreasing to $0$ with $n$.

We perform an optimization over $\lambda$ with the Cardan method. However, the optimal choice of $\lambda$ depends on $\rho$. 
So an optimization can be done for $\rho$ too. In Lemma \ref{RHO} below, we choose $\rho$ for the regime of \textit{Gaussian deviations}  (i.e. $\frac{a}{\sqrt {\Gamma_n}} \underset n \to 0$) which yields:
\begin{equation*} 
\P\big[ |\sqrt{\Gamma_n}\nu_n( \mathcal{A} \varphi )| \geq a \big] 
\leq 2\, C_n \exp
\Big(\! - c_n\frac{a^{2} }{2\nu( |\sigma^* \nabla \varphi |^2)}
\Big),
\end{equation*}
 for $c_n,C_n>0$ respectively decreasing and increasing (for $n$ big enough) to $1$ with $n$.

The optimal choices of $\lambda$ and $\rho$ for the regime of  \textit{super Gaussian deviations} (i.e. $\frac{a}{\sqrt {\Gamma_n}} \underset n \to + \infty$) is eventually discussed in \textcolor{black}{ Section \ref{sec_optim}}. This leads to 
\begin{equation*}
\P\big[ |\sqrt{\Gamma_n}\nu_n( \mathcal{A} \varphi )| \geq a \big] 
\leq 2\, C_n \exp \Big (\! - c_n\frac{a^{4/3} \Gamma_n^{1/3}}{2 \|\sigma\|_{\infty}^{2/3} \| \nabla \vartheta\|_{\infty}^{2/3}}
\Big). 
\end{equation*}

\subsection{Technical lemmas and Proof of the Main Results}

We first give a decomposition lemma of $\nu_n({\mathcal A}\varphi) $ which is the starting point of our analysis. Its proof can be found in \cite{hon:men:pag:16} (see Lemma 1 therein). 
\begin{lemme}[Decomposition of the empirical measure]
\label{decomp_nu}
For all  $n \ge 1$, $k \in \leftB 1 , n\rightB$ and $\varphi \in \mathcal C^{2}( \R^d, \R)$, the identity \eqref{Taylor} holds and we have:
\begin{eqnarray}\label{Taylor_Sum}
\Gamma_n \nu_n(\mathcal{A} \varphi) &=& \varphi ( X_n) - \varphi ( X_0) - \Big[ \sum_{k=1}^n \textcolor{black}{\gamma_k \int_0^1 \langle \nabla  \varphi (X_{k-1} + t \gamma_k b_{k-1})-\nabla \varphi(X_{k-1}),b_{k-1}\rangle dt} 
\nonumber \\
&& 
+ \frac 12\sum_{k=1}^n \gamma_k\, \Tr\Big( \big(D^2 \varphi ( X_{k-1} + \gamma_k b_{k-1} ) - D^2\varphi( X_{k-1}) \big) \Sigma_{k-1}^2\Big) + \sum_{k=1}^n \psi_k(X_{k-1}, U_k) \Big],
\end{eqnarray}
where $ \psi_k(X_{k-1}, U_k)$ is defined in \eqref{DEF_PSI_K}.
\end{lemme}

\begin{remark}
In spite of the square terms in $U_k$ appearing in the r.h.s. of \eqref{DEF_PSI_K}, we have that, conditionally to $\F_{k-1} $, $ u\mapsto \psi_k ( X_{k-1}, u)$ is Lipschitz continuous.
Indeed, on the r.h.s. $U_k$ only appears in $\psi_k$ and on the l.h.s. we know that $\varphi$ is Lipschitz.
Hence, for all $(u,u')\in (\R^d)^2$:
\begin{equation}\label{Psi_Lipschitz}
|\psi_k ( X_{k-1}, u)-\psi_k ( X_{k-1}, u')|\le \sqrt{\gamma_k} \|\sigma_{k-1}\|\|\nabla \varphi\|_{\infty}|u-u'|.
\end{equation}
Our strategy consists in controlling how far the Lipschitz modulus in \eqref{Psi_Lipschitz} is from $| \sigma_{k-1}^* \nabla \varphi_{k-1} |$. 
The first step is to obtain an explicit derivative of $\psi_k(X_{k-1},\cdot)$, see \eqref{DERIV_PSI} below.
\end{remark}
For notational convenience we introduce, for a given $n\in \N^* $ the following quantities:
\begin{eqnarray}
\label{DECOUP}
R_n&:=&\varphi ( X_n) - \varphi ( X_0) - \sum_{k=1}^n \textcolor{black}{\gamma_k \int_0^1 \langle \nabla  \varphi (X_{k-1} + t \gamma_k b_{k-1})-\nabla \varphi(X_{k-1}),b_{k-1}\rangle dt}
\nonumber \\
&-& \frac 12 \sum_{k=1}^n \gamma_k \Tr\Big( \big(D^2 \varphi ( X_{k-1} + \gamma_k b_{k-1} ) - D^2\varphi( X_{k-1}) \big) \Sigma_{k-1}^2 \Big), 
\nonumber \\
\nonumber \\
M_n &:=&\sum_{k=1}^n  \Delta_k(X_{k-1},U_k), \ \widetilde R_n := R_n -  \sum_{k=1}^n  \E\big[\psi_k(X_{k-1}, U_k)|\,\F_{k-1}\big],
\end{eqnarray}
where for all $k \in \leftB 1, n \rightB$:
\begin{equation}\label{Delta_def}
\Delta_k(X_{k-1},U_k):= \psi_k(X_{k-1}, U_j) - \E\big[\psi_k(X_{k-1}, U_k)|\,\F_{k-1}\big].
\end{equation}
From these definitions, Lemma~\ref{decomp_nu} can be rewritten:
\begin{equation}\label{nu_nA_phi_psi_R}
\nu_n ({\mathcal A}\varphi)=\frac{1}{\Gamma_n}( \widetilde R_n-M_n),
\end{equation}
where $M_n$ is  a martingale.
The key idea of the proof is to control more precisely the Lipschitz modulus of $\psi_n(X_{n-1}, \cdot)$ than it was done in \cite{hon:men:pag:16}. 
From the definition in \eqref{Taylor}, let us write for all $k \in \leftB 1, n \rightB$:
\begin{equation*}
\psi_k(X_{k-1}, U_k) = \varphi_k- \varphi_{k-1} + R_{k-1,k},
\end{equation*}
where 
\begin{eqnarray*}
R_{k-1,k}&:=&- \gamma_k \mathcal A \varphi ( X_{k-1}) - \textcolor{black}{\gamma_k \int_0^1 \langle \nabla  \varphi (X_{k-1} + t \gamma_k b_{k-1})-\nabla \varphi(X_{k-1}),b_{k-1}\rangle dt}
\nonumber \\
&-& \frac 12 \gamma_k \Tr\Big( \big(D^2 \varphi ( X_{k-1} + \gamma_k b_{k-1} ) - D^2\varphi( X_{k-1}) \big) \Sigma_{k-1}^2 \Big) 
.
\end{eqnarray*}
Hence, by derivation
\begin{equation}\label{DERIV_PSI}
\nabla_u \psi_k(X_{k-1}, u)|_{u=U_k} = \sqrt{\gamma_k} \sigma_{k-1}^*  \nabla \varphi(X_{k}).
\end{equation}
We will establish that the value of $\nabla_u \psi_k (X_{k-1},u)|_{u=U_k}$ is not ``too far" from $ \sqrt{\gamma_k} \sigma_{k-1}^* \nabla \varphi(X_{k-1})$.
\subsubsection{Proof of Theorem \ref{THM_CARRE_CHAMPS} for bounded innovations}\label{U_bounded}
We first give the complete proof in this particular case. We will specify the additional required controls for possibly unbounded innovations in the next subsection.

A key tool in the derivation of our main results is the following lemma whose proof is postponed to Section \ref{SEC_TEC} for the sake of clarity.
\begin{lemme}\label{HMP_lemmas}[Remainders from Taylor decomposition]
Under \A{A}, for all $q \geq 1$ and $\lambda>0$, we have:
\begin{equation}\label{ineq_T1}
\P\big[  \sqrt{\Gamma_n}\nu_n({\mathcal A} \varphi) \geq a\big] 
\leq 
\exp\big( - \frac{a\lambda}{\sqrt{\Gamma_n}} \big)  \Big ( \E \exp\big( -\frac{q \lambda}{ \Gamma_n}M_n\big) \Big)^{\frac1q} \exp( \frac{ \lambda^2}{\Gamma_n} e_n) \mathscr R_n .
\end{equation}
\end{lemme}
We will now sharply control the Lipschitz constant of $\psi_{k}(X_{k-1},\cdot) $, or equivalently $\Delta_k(X_{k-1},\cdot) $, which appears iteratively to handle the martingale term in \eqref{ineq_T1}.

In case of bounded innovations, we see by assumption \A{T${}_\beta$} i) (smoothness of $\varphi$) that:
\begin{eqnarray}\label{U_bounded_Lip}
&&| \nabla_u  \Delta_k(X_{k-1},u) |_{u=U_k}  |= | \sqrt{\gamma_k} \sigma_{k-1}^*  \nabla \varphi(X_{k})|
\nonumber \\
&& \leq | \sqrt{\gamma_k} \sigma_{k-1}^*  \nabla \varphi(X_{k-1})| 
+ | \sqrt{\gamma_k} \sigma_{k-1}^*  \left [ \nabla \varphi(X_{k}) -  \nabla \varphi(X_{k-1} + \gamma_k b_{k-1})  \right]| 
\nonumber \\
&&+ | \sqrt{\gamma_k} \sigma_{k-1}^*  \left [  \nabla \varphi(X_{k-1} + \gamma_k b_{k-1}) -\nabla \varphi(X_{k-1})   \right ]| 
\nonumber \\
&&\leq  \sqrt{\gamma_k} |\sigma_{k-1}^*  \nabla \varphi(X_{k-1})| +  \gamma_k \|\sigma_{k-1}\|^2  \|D^2 \varphi \|_{\infty}    \|U_k \|_{\infty} 
\nonumber \\
&&+  
| \sqrt{\gamma_k} \sigma_{k-1}^*  \left [  \nabla \varphi(X_{k-1} + \gamma_k b_{k-1}) -\nabla \varphi(X_{k-1})   \right ]| .
%
\end{eqnarray}
\textcolor{black}{
Remark that we have both controls
\begin{eqnarray}\label{double_ineq_diff_nabla_b}
 \!\!\!|\left [  \nabla \varphi(X_{k-1} + \gamma_k b_{k-1}) -\nabla \varphi(X_{k-1})   \right ]| \!\!\!
&\leq &  \!\!\!
\gamma_k \|D^2\varphi\|_\infty |b_{k-1} |  \!\!\!
\overset{\text{\A{${\mathcal L}_{\mathbf  V} $}, ii)}}{\leq}  \!\!\! \gamma_k \sqrt{C_V} \|D^2\varphi\|_\infty \sqrt V_{k-1} ,
\nonumber \\ 
 \!\!\! |\left [  \nabla \varphi(X_{k-1} + \gamma_k b_{k-1}) -\nabla \varphi(X_{k-1})   \right ]|
  \!\!\! &\le&  \!\!\! (2\|\nabla \varphi\|_\infty)^{\frac 12} |\left [  \nabla \varphi(X_{k-1} + \gamma_k b_{k-1}) -\nabla \varphi(X_{k-1})   \right ]|^{\frac 12}
  \nonumber \\
\!\!\! &\overset{\text{\A{${\mathcal L}_{\mathbf  V} $}, ii)}}{\leq} & \!\!\!  
 (2\|\nabla \varphi\|_\infty)^{\frac 12} \gamma_k^1/2 C_V^{1/4} \|D^2\varphi\|_\infty^{1/2} V^{1/4}_{k-1}  ,
 \end{eqnarray}
  in order to keep integrable powers of the Lyapunov function. We therefore eventually get from \eqref{U_bounded_Lip} and inequalities in \eqref{double_ineq_diff_nabla_b}: 
\begin{eqnarray}\label{U_bounded_Lip}
&&| \nabla_u  \Delta_k(X_{k-1},u) |_{u=U_k}  |^2
\le
\gamma_k |\sigma_{k-1}^*  \nabla \varphi(X_{k-1})|^2
\nonumber \\ 
&&+2 \sqrt{\gamma_k} |\sigma_{k-1}^*  \nabla \varphi(X_{k-1})| \big (\gamma_k \|\sigma_{k-1}\|^2  \|D^2 \varphi \|_{\infty}    \|U_k \|_{\infty}  +  \gamma_k^{3/2} \|\sigma\|_\infty \sqrt{C_V} \|D^2\varphi\|_\infty \sqrt V_{k-1} \big )
\nonumber \\
&&+\big (\gamma_k \|\sigma_{k-1}\|^2  \|D^2 \varphi \|_{\infty}    \|U_k \|_{\infty}  +  \|\sigma\|_\infty (2\|\nabla \varphi\|_\infty)^{\frac 12} \gamma_k C_V^{1/4} \|D^2\varphi\|_\infty^{1/2} V^{1/4}_{k-1}\big )^2 
\nonumber \\&&
\leq 
\gamma_k |\sigma_{k-1}^*  \nabla \varphi(X_{k-1})|^2+ C_{1,\eqref{U_bounded_Lip}} \gamma_k^{3/2} \|U_k\|_\infty+ C_{2,\eqref{U_bounded_Lip}} \gamma_k^2 \|U_k\|_\infty^2 + C_{\eqref{EXPR_GRAD}} \gamma_k^{2}\sqrt{V}_{k-1},
\end{eqnarray}
with $C_{1,\eqref{U_bounded_Lip}}:=2\|\sigma\|_\infty^3 \|\nabla \varphi\|_\infty \|D^2 \varphi\|_\infty$, $C_{2,\eqref{U_bounded_Lip}}:=2\|\sigma\|_\infty^4 \|D^2 \varphi\|_\infty^2$, $C_{\eqref{EXPR_GRAD}}:=6\|\sigma\|_\infty^2 \|\nabla \varphi\|_\infty \|D^2 \varphi\|_\infty $.
The last inequality above is a consequence of convexity inequality (i.e. for all $(x,y) \in \R^2$, $(x+y)^2\leq 2x^2+2y^2$).
}

Recalling that we consider first $\|U_k\|_\infty\le C_\infty $, we then derive:
\begin{equation}\label{omega1_LIP}
 [\Delta_k(X_{k-1},\cdot)]_ 1^2 \leq  \gamma_k  | \sigma^* \nabla \varphi |^2 (X_{k-1}) + C  \gamma_k^{3/2} +C_{\eqref{EXPR_GRAD}} \gamma_k^2  \sqrt V_{k-1},
\end{equation}
where in the above identity $C= C_{1,\eqref{U_bounded_Lip}}C_\infty+ C_{2,\eqref{U_bounded_Lip}} \gamma_1^{1/2} C_\infty^2$.
Let us introduce for all $(m,n)\in  \N_0^2$, $m \leq n$ and  $(\rho, q) \in (1,+ \infty)^2$:
\begin{equation}\label{Def_T}
T_{m} := \exp \Big ( - \frac{\rho q\lambda}{\Gamma_n}\Delta_m(X_{m-1},U_m)-\frac{\rho^2 (q\lambda)^2}{2\Gamma_n^2} \gamma_m | \sigma^*  \nabla \varphi (X_{m-1}) |^2
\textcolor{black}{-\sum_{k=1}^m C_{\eqref{EXPR_GRAD}}\gamma_k^{2} V_{k-1}}
 \Big)
.
\end{equation}
From the definition of $S_m$ in \eqref{Def_S}, we write 
$
 S_m := \prod_{k=1}^m T_k .
$ 
The coefficients $(T_m)_{m \geq 1}$ can be viewed as multiplicative increments of 
$(\tilde S_m)_{m \geq 0}$.
Inequality \eqref{omega1_LIP} precisely allows to quantify the martingality default for $(\tilde S_m)_{m \geq 0}$.
These factors appear when we exploit the auxiliary Poisson problem \eqref{Poisson_eq} in the definition of $\mathscr T_1$ in \eqref{Def_T_S}.
\begin{equation*}
 {\mathscr T}_1 =\exp \big ( \frac{\rho^2 q^2 \lambda^2}{2\Gamma_n}\nu(| \sigma^*  \nabla \varphi|^2) ]  \big) \E \big [  S_{n-1} \E[ T_n | \F_{n-1} ]  \big] .
\end{equation*}
Thereby, from the upper-bound \eqref{omega1_LIP} of the Lipschitz modulus, we directly obtain from \eqref{Def_T} and \A{GC}  
\begin{eqnarray*}
\E[ T_n | \F_{n-1}]
&=&
\exp \big ( -\frac{\rho^2 (q\lambda)^2}{2\Gamma_n^2} \gamma_n | \sigma^*  \nabla \varphi (X_{n-1}) |^2  \big )
 \E \Big [\exp \big  (  \frac{\rho q\lambda}{\Gamma_n}\Delta_n(X_{n-1},U_n) -\sum_{k=1}^n C_{\eqref{EXPR_GRAD}}\gamma_k^{2} V_{k-1} \big) \Big | F_{n-1}\Big ]
  \nonumber \\
&\leq&  \exp \big( \frac{\rho^2 (q\lambda)^2}{2\Gamma_n^2}  C \gamma_n^{3/2} -\sum_{k=1}^{n-1} C_{\eqref{EXPR_GRAD}}\gamma_k^{2} V_{k-1} \big).
\end{eqnarray*}
Hence, iterating:
\begin{eqnarray*}
{\mathscr T}_1 
&\leq& \exp \big ( \frac{\rho^2 (q\lambda)^2}{2\Gamma_n}\nu(| \sigma^*  \nabla \varphi|^2)   \big) \E[S_{n-1}] \exp \big(  \frac{\rho^2 (q\lambda)^2}{2\Gamma_n^2}   \gamma_n^{3/2}  C  \big)
\\
&\leq& \exp \big ( \frac{\rho^2 (q\lambda)^2}{2\Gamma_n}\nu(| \sigma^*  \nabla \varphi|^2)   \big)
 \exp \big(  \frac{\rho^2 (q\lambda)^2}{2\Gamma_n}   \underbrace{\frac{\Gamma_n^{(3/2)}}{\Gamma_n}}_{=e_n}  C  \big) 
= \exp \big ( \frac{\rho^2 (q\lambda)^2}{2\Gamma_n}  ( \nu(| \sigma^*  \nabla \varphi|^2) + e_n )   \big),
\end{eqnarray*}
where $ e_n \rightarrow 0$. The controls for ${\mathscr T}_2$ are deduced from \eqref{T2_user}\textcolor{black}{, and ${\mathscr T}_3$ from \eqref{T3_user}}. We now gather the previous estimates into \eqref{Mn_ineq_T1T2} (we recall $\E[\exp(- \frac{\lambda q M_n}{\Gamma_n}) ] \le {\mathscr T}_1 ^{\frac{1}{\rho}} {\mathscr T}_2^{\frac {\rho-1}{\hat q \rho}} {\mathscr T}_2^{\frac {\rho-1}{\hat p \rho}}$) in the following lemma.
Note also that the term $\mathcal R_n^{\vartheta}$ appearing in \eqref{T2_user} is controlled similarly to remainders in Lemma \ref{HMP_lemmas}.

\begin{lemme}[Gaussian concentration term]\label{Mn}
With notations of \eqref{DECOUP}, under \A{A}, for a bounded $\rho>1$, we have:
\begin{equation*}
\E\exp\big(-\frac{\lambda q}{\Gamma_n}M_n\big) ^{\frac 1q}\leq \exp \big(\frac{\lambda^2}{\Gamma_n} 
\textcolor{black}{A_n}+\frac{\lambda^4}{\Gamma_n^3} 
\textcolor{black}{B_n} \big ) {\mathscr R}_n,
\end{equation*}
where \textcolor{black}{
\begin{equation}\label{def_ABbar}
A_n:= \rho \big( \frac{ q\nu(| \sigma^*  \nabla \varphi |^2)}{2}+e_n \big )
\quad\mbox{ and }\quad 
B_n:= \frac{\rho^3}{\rho-1} \frac{q^3\hat q}{4}\big( \frac{\bar q\|\sigma\|_\infty^2[\vartheta]_1^2}{2}+ e_n\big),
\end{equation}
}
for some $1<\bar q:=\bar q(n) \underset{n}{\rightarrow}1$, and with:
$
e_n\underset{n \to + \infty}{ \longrightarrow} 0, \
{\mathscr R}_n \underset{n \to + \infty}{ \longrightarrow} 1
$
uniformly in $\lambda$.
\end{lemme}

As a consequence of the previous Lemmas \ref{HMP_lemmas} and \ref{Mn}, we obtain \eqref{Ineq_P}, namely:
\begin{equation}\label{ineq_P}
\P \big( \sqrt{\Gamma_n} \nu_n ( \mathcal{A} \varphi ) \geq a \big) \leq C_n 
\exp \big( P(\lambda)  \big ),
\end{equation}
with $P(\lambda) :=  - \frac{ a\lambda }{\sqrt{\Gamma_n}}+\frac{\lambda^2}{\Gamma_n}A_n+\frac{\lambda^4}{\Gamma_n^3}B_n $,  
where $A_n = A_n(\rho) = \rho \widetilde A_n$ and $ B_n=B_n(\rho):=\frac{\rho^3}{\rho-1}\widetilde B_n$ with 
\begin{equation*}
\widetilde A_n = \frac{ q\nu(| \sigma^*  \nabla \varphi |^2)}{2}+e_n
\quad\mbox{ and }\quad 
\widetilde B_n =\frac{q^3 \hat q}{4}\big( \frac{\bar q\|\sigma \|_{\infty}^2 \| \nabla \vartheta\|_{\infty}^2}{2}+ e_n\big).
\end{equation*}
Next, like enunciated at the end of the \textit{User's guide to the proof}, 
 we optimize a fourth order polynomial by the Cardan method, see \eqref{lambda} below (and Section 4 in \cite{hon:men:pag:16}).
If $\lambda_n= \argmin_\lambda P(\lambda)$, then
$$P'(\lambda_n) =  - \frac{ a}{\sqrt{\Gamma_n}}+\frac{2 \lambda_n}{\Gamma_n}A_n+\frac{4 \lambda_n^3}{\Gamma_n^3}B_n = 0.$$
 The Cardan-Tartaglia formula yields only one positive real root. 
 Namely, setting
\[
\Phi_n(a,\rho)= 
\Big( \frac{a}{\sqrt{\Gamma_n}\tilde B_n }+\big(\frac{a^2}{\tilde B_n^2\Gamma_n}+ (\rho-1)\big(\frac{2 \tilde A_n}{3\tilde B_n}\big)^3\big)^{\frac 12}\Big)^{\frac 13}+\Big( \frac{a}{\sqrt{\Gamma_n}\tilde B_n }-\big(\frac{a^2}{\tilde B_n^2\Gamma_n}+ (\rho-1)\big(\frac{2 \tilde A_n}{3\tilde B_n}\big)^3\big)^{\frac 12}\Big)^{\frac 13},
\]
this conducts to take $\lambda=\lambda_n$ with:
\begin{equation}\label{lambda}
\lambda_n 
:=\frac{\Gamma_n}{2}\frac{(\rho-1)^{\frac 13}}{\rho}\Phi_n(a,\rho).
\end{equation}
Moreover, remark from the binomial Newton expansion that: 
\begin{equation*}
\Phi_n(a,\rho)^3 = \frac{2a}{\sqrt{\Gamma_n}\tilde B_n } -  \frac{2 ( \rho-1)^{1/3} \tilde A_n}{\tilde B_n} \Phi_n(a,\rho).
\end{equation*}
From \eqref{lambda} and the above expression,  
$ P(\lambda_n)=P_{\min}(a,\Gamma_n, \rho)=  \lambda_n \big (- \frac{ a }{\sqrt{\Gamma_n}}+\frac{\lambda_n}{\Gamma_n}A_n+\frac{\lambda_n^3}{\Gamma_n^3}B_n \big ) $, we \textcolor{black}{thus} obtain 
\begin{equation}\label{P_min_def}
P_{\min}(a,\Gamma_n, \rho) := -\frac{\sqrt{\Gamma_n} (\rho-1)^{1/3} \Phi_n(a,\rho)}{2^3 \rho} \big ( 3a- \sqrt{\Gamma_n}(\rho-1)^{1/3} \tilde A_n \Phi_n(a,\rho) \big ).
\end{equation}
Then, from \eqref{ineq_P}:
\begin{equation}
\label{ineq_PMIN}
\P \big( \sqrt{\Gamma_n} \nu_n ( \mathcal{A} \varphi ) \geq a \big) \leq C_n \exp \big( P_{\min}(a,\Gamma_n, \rho) \big ),
\end{equation}
\textcolor{black}{which} is exactly the same bound appearing in Remark 11 in \cite{hon:men:pag:16}, up to a modification of $\tilde A_n$, containing here the expected \textit{carré du champ}.

 \textcolor{black}{The optimization over $ \lambda$ leads to study how $\rho$ should asymptotically behave. The following lemma indicates that, when $a=o(\sqrt{\Gamma_n}) $, taking $\rho -1 \asymp \frac a{\sqrt{\Gamma_n}}$ yields a Gaussian concentration inequality in \eqref{ineq_P} with the optimal constant.}
\begin{lemme}[Choice of $\rho$ 
for the  Gaussian concentration regime]\label{RHO}
For 
$P_{\min}(a,\Gamma_n, \rho)$ as in \eqref{P_min_def},
there is $\rho:=\rho(n,a)>1$ s.t.
If $\frac{a}{\sqrt{\Gamma_n}} \underset{n}{\rightarrow} 0$, taking $\rho-1 \asymp \frac a {\sqrt{\Gamma_n}}$ 
\begin{equation*}
P_{\min}(a,\Gamma_n, \rho) \underset{\frac{a}{\sqrt{\Gamma_n}} \underset{n}{\rightarrow} 0}{ =} - \frac{a^2}{2 \nu(| \sigma^*  \nabla \varphi |^2)}(1+ o(1 )).
\end{equation*}
\end{lemme}
For the sake of clarity, the proof of Lemma \ref{RHO} is postponed to Section \ref{proof_lemma_RHO}.
From \eqref{ineq_PMIN} and Lemma \ref{RHO}, we conclude the proof of Theorems \ref{THM_CARRE_CHAMPS} and \ref{Exact_OPTIM} for bounded innovations.

\subsubsection{Proof of Theorems \ref{THM_CARRE_CHAMPS} 
for unbounded innovations}
Switching to unbounded innovations requires additional technicalities.
Our strategy consists in 
considering a truncation argument writing $\Delta_k(X_{k-1},U_k)=\Delta_k(X_{k-1},U_k) [ \mathds{1}_{|U_k| \leq \frac{r_{k,n}}{2}} + \mathds{1}_{|U_k| > \frac{r_{k,n}}{2}} ]$, 
to control the Lipschitz modulus of $\psi_{k-1}$ where $(r_{k,n})_{n \geq 1, k \leq n}$ is a suitable sequence specified in \eqref{rn} below.
In particular, $r_{k,n}:=r_{k,n}(\A{A},\lambda,\rho) $ where $\lambda>0,\rho>1 $ are as in the \textit{User's Guide to the Proof}. 

For our choice below, we will have that, for all $k \in \leftB 1,n \rightB$,  $r_{k,n}\uparrow_n + \infty$.
That choice for $r_{k,n}$ also yields that when $|U_k| \leq r_{k,n}$, our controls behave like for the bounded case.
But when $|U_k| > r_{k,n}$, we will handle this large deviation regime by assumption \A{GC}. We indeed know that for all $K>0$:
\begin{eqnarray}\label{ineq_dens_inno}
\mu(\{|x| > K \}) \leq 2 \exp( -\frac{ K^2}2).
\end{eqnarray}
Let us recall from the definition of $\Delta_k(X_{k-1},U_k)$ in  \eqref{DEF_PSI_K} and \eqref{DECOUP} that:
\begin{equation}
\label{def_DELTA_K}
\Delta_k(X_{k-1}, U_k) =\sqrt{\gamma_k} \sigma_{k-1} U_k \cdot \nabla \varphi( X_{k-1} +\gamma_k b_{k-1} ) + \Xi_k(X_{k-1},U_k),
\end{equation}
where for all $(k,u) \in [\![1,n]\!] \times \R^r$:
\begin{eqnarray}\label{def_Xi}
\Xi_k(X_{k-1},u) := && \gamma_k \int_0 ^1 ( 1-t) \Tr\Big ( D^2\varphi ( X_{k-1} +\gamma_k b_{k-1} + t \sqrt{\gamma_k} \sigma_{k-1} u ) \sigma_{k-1}u\otimes u \sigma_{k-1}^* 
\nonumber \\
&&- \E \big[ D^2\varphi ( X_{k-1} +\gamma_k b_{k-1} + t \sqrt{\gamma_k} \sigma_{k-1} U_k ) \sigma_{k-1}U_k\otimes U_k \sigma_{k-1}^* | \F_{k-1} \big ] \Big )dt.
\end{eqnarray}
For the terms  $(T_k)_{k \leq n}$ defined in \eqref{Def_T}, the lemma below controls the ``\textit{super martingality}" default of $S_n= \prod_{k=1}^n T_k$.
\begin{lemme}\label{CTR_T1_N}
For all $k \in [\![1,n]\!]$\textcolor{black}{
\begin{equation*}
\E[ T_k | \F_{k-1} ] \leq \aleph_{k,n} (\lambda, \gamma_k,  r_{k,n}) \exp \big (-\sum_{i=1}^{k-1} C_{\eqref{EXPR_GRAD}}\gamma_i^{2} V_{i-1} \big ), 
\end{equation*}
}with, \textcolor{black}{as in \eqref{U_bounded_Lip}, $C_{\eqref{EXPR_GRAD}}:=6\|\sigma\|_\infty^2 \|\nabla \varphi\|_\infty \|D^2 \varphi\|_\infty $,}
\begin{equation}\label{def_aleph}
\aleph_{k,n} (\lambda, \gamma_k,  r_{k,n}) 
:=
\big( 1+ 2\exp ( - C  r_{k,n}^2 ) \big) \exp \big (\frac{ \rho q \lambda}{ \Gamma_n} C \gamma_k^{1/2} \exp(-\frac{ r_{k,n}^2}{4}  ) \big) 
 \exp \big ( \frac{ \rho^2 q^2 \lambda^2 }{\Gamma_n^2} C \gamma_k^{3/2}  r_{k,n}^2 
\big ),
\end{equation}
and
\begin{equation}\label{rn}
r_{k,n}=r_{k,n}(\A{A},\lambda,\rho):=
\left\{
\begin{array}{ll}
&r_n= C 
 (1+\frac{\rho q \lambda}{\Gamma_n} ) \big (\frac{\Gamma_n}{\Gamma_n^{(3/2)}} \big)^{1/4}, \ for \ \theta \in (\frac 13, 1),
\\
& C 
 (1+\frac{\rho q \lambda}{\Gamma_n} ) \ln(n+1)^{1/4} \ln(k+1)^{1/2}
,  \ for \ \theta=1,
\end{array}
\right.
\end{equation}
for $C>0$ s.t. $r_{1,1} > \bar c\sqrt{\gamma_1} \|\sigma \|_{\infty} \|\nabla \varphi \|_{\infty}$ for $\bar c$ large enough (every $\bar c>8 $ works, see the proof of Lemma \ref{CTR_T1_N}, and equation \eqref{GC_Delta_barHn1}). 
This choice is briefly explained in Remark \ref{r_n_aleph} below.
\end{lemme}

\begin{remark}\label{r_n_aleph}
The specific form of the truncation and of the time steps chosen yields $$r_{k,n}
=\begin{cases} 
O \big( (1+\frac{\rho q \lambda}{\Gamma_n} ) n^{\frac{1-\theta}4} \big ),\ {\rm if}\ \theta \in (2/3,1),\\
O \big ((1+\frac{\rho q \lambda}{\Gamma_n} )  \ln(n)^{-1/4}n^{\frac{1}{12}} \big),\ {\rm if}\ \theta =2/3,\\
O\big ((1+\frac{\rho q \lambda}{\Gamma_n} ) n^{\frac{\theta}8} \big ),\ {\rm if}\ \theta \in (1/3,2/3).
\end{cases}
$$ Anyhow, we always have for each $k\leq n$, $r_{k,n} \underset{n }{\longrightarrow} + \infty$.
Identity \eqref{ineq_Sn} in the following lemma can give an intuition of our choice in \eqref{rn}.
This result ensures that the terms $\aleph_{k,n} (\lambda, \gamma_k,  r_{k,n}) $ can be viewed as remainders (observe indeed that some $e_n\underset{n}{\rightarrow} 0 $ appear in the exponential \eqref{ineq_Sn} below). 
From the optimization in $\lambda $ performed in the proof Lemma \ref{RHO} (see equation \eqref{RLNG}), the contribution $\frac{\rho \lambda}{\Gamma_n}$ appearing in \eqref{def_aleph} will be large in the regime of \textit{super Gaussian} deviations (see as well Remark 
\ref{REM_RESTE_LAMBDA_GAMMA_N}). The above choice of $r_{k,n} $ actually permits
to control the remainders in all the considered regimes. 

For $\theta\in (\frac 13,1) $, the choice in \eqref{rn} can seem natural in order to absorb the term \\ $\prod_{k=1}^n\exp \big (\frac{ \rho q \lambda}{ \Gamma_n} C \gamma_k^{1/2} \exp(-\frac{ r_{k,n}^2}{16}  )\big)=\exp\big(\frac{ \rho q \lambda}{ \Gamma_n} C \Gamma_n^{(1/2)} \exp(-\frac{ r_{n}^2}{16}  )\big)$ coming from the iteration of \eqref{def_aleph}. For $\theta=1 $, the choice is a bit different due to the associated logarithmic explosion rates (i.e. $\Gamma_n \asymp \ln(n) $).
\textcolor{black}{Actually, for the \textit{Gaussian deviations} ($\frac a{\sqrt{\Gamma_n}} \to 0$), we have $ \frac{\rho q \lambda}{\Gamma_n} \to_n 0$, see again \eqref{RLNG} and Remark \ref{REM_RESTE_LAMBDA_GAMMA_N} below, and the term $\frac{\rho q \lambda}{\Gamma_n}$ could be removed in \eqref{rn}.}
On the other hand, the contribution $\exp \big ( \frac{ \rho^2 q^2 \lambda^2 }{\Gamma_n^2} C \gamma_k^{3/2}  r_{k,n}^2\big) $ will eventually yield a negligible contribution in the polynomial appearing in Lemma \ref{Mn}. We refer to the proof of Lemma \ref{CTR_T1_N} for details.
\end{remark}
\begin{lemme}[Control of ``super martingality default" of $S_n$] \label{rn_choice}
There exist non negative sequences $( \mathscr R_n)_{n \geq 1}$, $(e_{n})_ {n\geq 1}$ s.t. $ \mathscr R_n \underset{n}{\longrightarrow} 1$, $e_{n}\underset{n}{\longrightarrow} 0$, and for all $n \geq 1$:
\begin{equation}\label{ineq_Sn}
\E[ S_n ] \leq\prod_{k=1}^n \aleph_{k,n}(\lambda, \gamma_k,r_{k,n}) = \mathscr R_n  \exp \Big (  \big(\frac{ \rho^2 q^2 \lambda^2 }{\Gamma_n}    +\frac{ \rho^4 q^4 \lambda^4 }{\Gamma_n^3}  \big)e_{n}  \Big ).
\end{equation}
\end{lemme}
Observe that a term in $\lambda^4 $ appears here for the control of $\E[S_n] $. This is specifically due to the unbounded contributions. Namely, the exponential term in \eqref{ineq_Sn} comes from $\exp(\frac{\rho^2q^2\lambda^2}{\Gamma_n^2}C\gamma_k^{3/2}r_{k,n}^2) $ in \eqref{def_aleph}, the definition of $r_{k,n} $ in \eqref{rn}, and using as well that $e_n \asymp \frac{\Gamma_n^{(3/2)}}{\Gamma_n} \underset{n \to \infty}\rightarrow 0$. The other terms in \eqref{def_aleph}, corresponding to sub-Gaussian tails, give the remainder ${\mathscr R_n} $.

Note now carefully that, reproducing the arguments of the bounded case and using as well Lemma \ref{rn_choice} to control $\E[S_n] $ yields that Lemma \ref{Mn} remains valid, up to a modification of the remainders $e_n $. The proof then follows similarly to the bounded case. To sum up, the specificity of the unbounded innovations was to precisely control the Lipschitz constants, considering a suitable truncation, as well as the ``martingality default" of $S_n $ appearing in ${\mathscr T}_1 $.

\section{Proofs of technical lemmas}
\setcounter{equation}{0}
\label{SEC_TEC}

\subsection{Remainders from the Taylor decomposition}

\begin{proof}[Proof of Lemma \ref{HMP_lemmas}]

From the notations in \eqref{DECOUP}, we recall \eqref{nu_nA_phi_psi_R}: 
$\nu_n ({\mathcal A}\varphi)=\frac{1}{\Gamma_n}(  \widetilde  R_n-M_n)$.
The idea is now to write for $a, \lambda>0 $:
\begin{eqnarray}\label{Tchebytchev_holder}
\P\big[ \sqrt{\Gamma_n}\nu_n({\mathcal A} \varphi)  
\geq a\big]&\le&
\exp \big(-\frac{a \lambda}{\sqrt{\Gamma_n}}\big)  \E \Big[\exp \big (\frac{\lambda }{\Gamma_n}(  \widetilde  R_n-M_n) \big) \Big]\notag\\
&\le & \exp\big(-\frac{a \lambda}{\sqrt{\Gamma_n}}\big)\E \Big[\exp\big(-\frac{q \lambda }{\Gamma_n}M_n\Big) \Big]^{1/q}\E \Big[\exp\big(\frac{p \lambda }{\Gamma_n}| \widetilde R_n| \big) \Big]^{1/p},\label{EXPLI_DEV}
\end{eqnarray}
for $ \frac 1p+\frac 1q=1,\ p,q>1$.
We rewrite the Taylor expansion with the same notations as in \cite{hon:men:pag:16}:
$\widetilde R_n=L_n-  ( D_{2,b,n}+D_{2,\Sigma,n}+\bar G_n  ) $ where:
\begin{eqnarray}
\label{DECOUP_lemme_HMP}
D_{2,b,n}&:=&\sum_{k=1}^n \textcolor{black}{\gamma_k \int_0^1 \langle \nabla  \varphi (X_{k-1} + t \gamma_k b_{k-1})-\nabla \varphi(X_{k-1}),b_{k-1}\rangle dt},\nonumber\\
D_{2,\Sigma,n}&:=&\frac 12 \sum_{k=1}^n \gamma_k \Tr\Big(  \big(D^2 \varphi ( X_{k-1} +  \gamma_k b_{k-1} ) - D^2\varphi( X_{k-1}) \big) \Sigma_{k-1}^2 \Big),\nonumber\\
\bar G_{n}&:=&\sum_{k=1}^n \E\,[\psi_k(X_{k-1}, U_k)|\F_{k-1}],\notag\\
L_n&:=&\varphi ( X_n) - \varphi ( X_0).
\end{eqnarray}
From \eqref{DECOUP_lemme_HMP},   \eqref{Tchebytchev_holder} and the Cauchy-Schwarz inequality, we get:
\begin{eqnarray}\label{Holder_lemmas_hmp}
\P\big[  \sqrt{\Gamma_n}\nu_n({\mathcal A} \varphi)  \geq a\big]
\leq  \exp\big( - \frac{a\lambda}{\sqrt{\Gamma_n}} \big)
 \Big( \E \exp\big( -\frac{q \lambda}{ \Gamma_n}M_n\big)\Big)^{\frac1q}
\times \Big(\E \exp \big( \frac{2p \lambda}{ \Gamma_n} \big| L_n\big| \big) \Big)^{\frac 1{2p}}
\nonumber\\
\times \Big(\E \exp\big( \frac{4p \lambda}{ \Gamma_n} \big| \bar G_n\big|\big) \Big)^{\frac 1{4p}}
\Big(\E\exp\big( \frac{8p \lambda}{ \Gamma_n}\big| D_{2,\Sigma,n}\big| \big) \Big)^{\frac1{8p}}
\Big(\E\exp\big( \frac{8p \lambda}{ \Gamma_n}\big| D_{2,b,n}\big| \big) \Big)^{\frac 1{8p}}. 
\label{DECOUP_TCHEB}
\end{eqnarray}
The term $L_n$ in \eqref{Holder_lemmas_hmp} is controlled in Lemma 4 in \cite{hon:men:pag:16} (for $j=2$ therein):
\begin{equation}\label{Ln}
\Big(\E \exp\big( 4p \lambda \frac{|L_n| }{ \Gamma_n}\big) \Big)^{\frac{1}{4p}} \leq 
(I_V^1)^{\frac{1}{4p}} \exp \big(\frac{3 p C_{V,\varphi}^2 \lambda^2}{ c_V \Gamma_n^2}+\frac{c_V}{p}\big)
= \mathscr R_n \exp\big(\frac{ \lambda_n^2}{  \Gamma_n} e_n\big),
\end{equation}
for $p=p_n \to_n + \infty$ s.t. $\frac{p}{\Gamma_n} \to_n 0$.

Thanks to Lemma 3 in \cite{hon:men:pag:16}, 
we obtain:
$$
\frac{|\bar G_n|}{\sqrt{\Gamma_n}}\le a_n:= \frac{ [\varphi^{(3)}]_{\beta}  \big\|\sigma\big\|_\infty^{3+\beta} \E\big[|U_1|^{3+\beta} \big] }{(1+\beta)(2+\beta)(3+\beta)}   \frac{\Gamma_n^{(\frac{3+\beta}2)}}{\sqrt{\Gamma_n}}, \,a.s.  \ .
$$
Moreover, $a_n\underset{n}{\ra} a_\infty=0 $ for $\theta\in(\frac{1}{2+\beta},1] $. Hence, for all $p>1$:
\begin{equation}
\label{CTR_BIAS_HOLDER}
\Big(\E\exp\big(\frac{4p\lambda}{\Gamma_n}|\bar G_n|\big)\Big)^{\frac1{4p}}\le \exp\big(\frac{ \lambda }{\sqrt{\Gamma_n}}a_n \big)\le \exp\big(\frac{ \lambda^2}{2\Gamma_np}+\frac{a_n^2p}{2} \big)= \mathscr R_n  \exp\big(\frac{ \lambda^2}{\Gamma_n} e_n \big).
\end{equation}
for $p=p_n \to_n + \infty$ s.t. $a_n^2 p\to_n 0$.

We handle the term  $D_{2,\Sigma,n}$ in \eqref{Holder_lemmas_hmp} by Lemma 5  in \cite{hon:men:pag:16}:
there exists  $C_{1}:=C_1(\A{A},\varphi)>0$ such that
\begin{equation}\label{ineq_rest_s2_phi}
\begin{split}
\Big( \E  \exp\big( \frac{4p \lambda_n}{ \Gamma_n}  \big|D_{2,\Sigma,n}\big|\big) \Big)^{\frac{1}{4p}} 
\leq \exp\big( C_{1}\frac{p \lambda_n^2  (\Gamma_n^{(2)})^2}{  \Gamma_n^2} \big) (I_V^1)^{\frac1{4p}}
= \mathscr R_n \exp \big(\frac{ \lambda_n^2}{  \Gamma_n} e_n \big),
\end{split}
\end{equation}
for $p=p_n \to_n + \infty$ s.t. $p\frac{(\Gamma_n^{(2)})^2}{\Gamma_n} \to_n 0$ (we recall that for all $\theta \in ( \frac 13 , 1]$, $\frac{\Gamma_n^{(2)}}{\sqrt{\Gamma_n}} \underset n \rightarrow 0$).

We deal with the term $D_{2,b,n}$ from \eqref{Holder_lemmas_hmp}. Because $x \mapsto \langle \nabla \varphi(x), b(x) \rangle$ is Lipschitz continuous,
 thanks to Lemma 5 in \cite{hon:men:pag:16}, we know that there exists  $C_{2}:=C_2(\A{A},\varphi)>0$ such that:
\begin{equation}\label{ineq_rest_b2_phi}
\begin{split}
\Big(\E  \exp\big( \frac{4p \lambda_n}{ \Gamma_n} \big|D_{2,b,n}\big|\big) \Big)^{\frac{1}{4p}} 
\leq \exp \big(C_{2} ( \frac{3p\lambda_n^2 (\Gamma_n^{(2)})^2}{2\Gamma_n^2}+ \frac{1}{2p} ) \big )(I_V^1)^{\frac1{4p}}
= \mathscr R_n \exp \big(\frac{ \lambda_n^2}{  \Gamma_n} e_n \big),
\end{split}
\end{equation}
also for $p=p_n \to_n + \infty$ s.t. $p\frac{(\Gamma_n^{(2)})^2}{\Gamma_n} \to_n 0$.

We gather \eqref{Ln}, \eqref{CTR_BIAS_HOLDER}, \eqref{ineq_rest_s2_phi} and \eqref{ineq_rest_b2_phi} into \eqref{Holder_lemmas_hmp}, which
allows us to control the remainder involving $\widetilde  R_n$ previously decomposed in \eqref{DECOUP_lemme_HMP}. 
There are non-negative sequences $(\mathscr R_n)_{n \geq 1}$, $(e_n)_{n \geq 1} $ s.t.
$\lim \limits_{n \to + \infty} \mathscr R_n =1$, $ \lim \limits_{n \to + \infty} e_n=0$ and:
\begin{equation}\label{DEV_PRELIM}
\P\big[  \sqrt{\Gamma_n}\nu_n({\mathcal A} \varphi) \geq a\big] 
\leq \exp\Big( - \frac{a\lambda}{\sqrt{\Gamma_n}} \Big) \left( \E \exp\Big( -\frac{q \lambda}{ \Gamma_n}M_n\Big)\right)^{\frac1q}
 \exp\left(\frac{ \lambda_n^2}{  \Gamma_n} e_n\right)  \mathscr R_n.
\end{equation} 
\end{proof}
\subsection{\textcolor{black}{Asymptotics in the parameter $\rho $}}\label{proof_lemma_RHO}
We first begin with the proof of Lemma \ref{RHO} which is purely analytical and rather independent of our probabilistic setting. We recall that we use it for both bounded and unbounded innovations.

\begin{proof}[Proof of Lemma \ref{RHO} ]

From the expression of $\Phi_n(a,\rho)$ in Theorem \ref{Exact_OPTIM}, remember that
\begin{equation*}
\Phi_n(a,\rho)= 
\Big( \frac{a}{\sqrt{\Gamma_n}\tilde B_n }+\big(\frac{a^2}{\tilde B_n^2\Gamma_n}+ (\rho-1)\big(\frac{2 \tilde A_n}{3\tilde B_n}\big)^3\big)^{\frac 12}\Big)^{\frac 13}+\Big( \frac{a}{\sqrt{\Gamma_n}\tilde B_n }-\big(\frac{a^2}{\tilde B_n^2\Gamma_n}+ (\rho-1)\big(\frac{2 \tilde A_n}{3\tilde B_n}\big)^3\big)^{\frac 12}\Big)^{\frac 13}.
\end{equation*}
Here, $\rho>1 $ is a \textit{free} parameter. Let us set:
\begin{equation}\label{CHOIX_RHO}
\rho-1 := \xi \frac{27}{8} \frac{\tilde B_n a^2}{\tilde A_n^3 \Gamma_n},
\end{equation}
for a parameter $\xi:=\xi(a,n)>0$ to optimize.
This choice yields 
\begin{equation*}
\label{REL_PHI_AL}
 \Phi_n(a,\rho)= \frac{a^{1/3}}{\tilde B_n^{1/3} \Gamma_n^{1/6}} \big ( (1+ \sqrt{1+ \xi})^{1/3} + (1 - \sqrt{1+\xi})^{1/3} \big).
 \end{equation*}
Hence, from the definition of $\lambda_n$ in \eqref{lambda}:
\begin{equation}
\label{RLNG}
\frac{\rho \lambda_n}{\Gamma_n}=\frac{1}{2}(\rho-1)^{1/3} \Phi_n(a,\rho) = \frac{3}{2^2} \frac{ a}{\tilde A_n \sqrt{\Gamma_n} } \xi^{1/3} \big ( (1+ \sqrt{1+ \xi})^{1/3} + (1 - \sqrt{1+\xi})^{1/3} \big).
\end{equation}
We point out that 
$
\xi \longmapsto  \xi^{1/3}\big( (1+ \sqrt{1+\xi})^{1/3}+(1-\sqrt{1+\xi})^{1/3} \big )
$ 
is a bounded function from $[0,+\infty)$ to $[0,\frac{2}{3})$. 
 In fact, 
\begin{eqnarray}\label{limite}
&&\xi^{1/3}\big( (1+ \sqrt{1+\xi})^{1/3}+(1-\sqrt{1+\xi})^{1/3} \big )
\nonumber \\
 &\underset{\xi \to + \infty} {=}& \xi^{1/3} (1+\xi)^{1/6}\big ( 1+ \frac{1}{3 \sqrt{1+\xi}} - ( 1 - \frac{1}{3\sqrt{1+\xi}}) + o(\frac{1}{\sqrt{1+\xi}}) \big )
\underset{\xi \to + \infty} {\sim}\frac{2}{3}.
\end{eqnarray}
From the definition of $P_{\min}(a,\rho,\Gamma_n)=P( \lambda_n) =\min_{x>0} P(x)$ in \eqref{P_min_def} and from the definition of  $P(\lambda)$ in \eqref{ineq_P},
\begin{eqnarray}
P \big(\lambda_n \big)&=& - \frac{(\rho-1)^{1/3}}{\rho} \frac{\sqrt{\Gamma_n} \Phi_n(a,\rho)}{8} ( 3a -  \sqrt{\Gamma_n}  (\rho-1)^{1/3} \tilde A_n \Phi_n(a,\rho))
\notag\\
&=&-\frac{(\rho-1)^{1/3}\Phi_n(a,\rho)}{2}\times \frac{\sqrt{\Gamma_n}}{4\rho}\times ( 3a -  \tilde A_n \sqrt{\Gamma_n}(\rho-1)^{1/3} \Phi_n(a,\rho)  )\notag\\
&\underset{\eqref{CHOIX_RHO}, \eqref{RLNG}}{=}& - \frac{3^2}{2^4} \frac{ a^2}{\tilde A_n ( \xi \frac{27 \tilde B_n a^2}{8 \tilde A_n^3 \Gamma_n}+1) } \xi^{1/3} \big ( (1+ \sqrt{1+ \xi})^{1/3} + (1 - \sqrt{1+\xi})^{1/3} \big) 
\notag\\
&& \times \Big ( 1 - \frac{ \xi^{1/3}}{2} \big( (1+ \sqrt{1+ \xi})^{1/3} + (1 - \sqrt{1+\xi})^{1/3} \big) \Big)
\notag\\
&=:& - \frac{3^2}{2^4 \tilde A_n} a^2 f_{\Psi}(\xi),\label{PREAL_OPTIM_POLY}
\end{eqnarray}
for 
\begin{equation}\label{function}
f_\Psi : \xi \in \R_+\longmapsto \frac{g(\xi)}{\Psi \xi+1},
\end{equation}
where 
\begin{equation*}
g : \xi \longmapsto \xi^{1/3}\big( (1+ \sqrt{1+\xi})^{1/3}+(1-\sqrt{1+\xi})^{1/3} \big )\Big  (1- \frac{\xi^{1/3}}{2} \big ((1+ \sqrt{1+\xi})^{1/3}+(1-\sqrt{1+\xi})^{1/3} \big ) \Big  ),
\end{equation*}
and
\begin{equation}\label{Psi_def}
\Psi:=\frac{27}{8} \frac{\tilde B_n a^2}{\tilde A_n^3 \Gamma_n} \Big (\overset{\eqref{CHOIX_RHO}} =\frac{\rho-1}{\xi} \Big ).
\end{equation}
We bring to mind that we consider ``Gaussian deviations'', namely $\frac{a}{\sqrt{\Gamma_n}} \underset{n}{\rightarrow} 0 $.

From \eqref{PREAL_OPTIM_POLY} and the asymptotic of $\tilde A_n $ defined in \eqref{eq:ABtilde}, i.e. $4 \tilde A_n \rightarrow_n 2\nu(|\sigma^*\nabla \varphi|^2) $, we want to choose $\xi:=\xi(a,n) $ s.t. $\Lambda(\xi):= \frac{3^2}{2^4 }  f_\Psi(\xi) \underset{\frac{a}{\sqrt{\Gamma_n}} \underset{n}{\rightarrow} 0}{\rightarrow}\frac 14  $.
This would indeed yield
$P(\lambda_n) \underset{\frac{a}{\sqrt{\Gamma_n}} \underset{n}{\rightarrow} 0} {\sim }- \frac{a^2}{4\tilde A_n} $. 
From the definition of $\Psi$ in \eqref{Psi_def}, we have $\Psi= \frac{27 \tilde B_n a^2}{8 \tilde A_n \Gamma_n} \underset {\frac a{\sqrt{\Gamma_n}} \to 0} {\longrightarrow} 0$. Observe then, that, taking $\xi $ going to infinity such that 
\begin{equation}\label{COND}
\xi\Psi\underset{\frac{a}{\sqrt{\Gamma_n}}\rightarrow 0}{\rightarrow} 0
\end{equation}
 yields $\Lambda(\xi)=\frac{3^2}{2^4 } \frac{g(\xi)}{\Psi\xi+1} \underset{ \frac a{\sqrt{\Gamma_n}} \to 0 
}{\rightarrow} \frac 14$, noting from \eqref{limite} and the above definition of $g$  that $g(\xi)\underset{\xi \rightarrow +\infty}{\rightarrow}\big(\frac 23\big)^2  $.
\end{proof}
\begin{remark}[Controls of the optimized parameters $\lambda$ and $\rho$ for \textit{Gaussian deviations}]\label{REM_RESTE_LAMBDA_GAMMA_N_asymp}
We give here some useful estimates to control the remainder terms in the truncation procedure associated with unbounded innovations, see proof of Lemma \ref{rn_choice}. 
They specify the behaviour of the quantity $\frac{\rho \lambda_n}{\Gamma_n} $.

In the regime of \textit{Gaussian deviations},
 we get from \eqref{RLNG} and \eqref{limite}:
\begin{equation}
\label{REG_1_RLG}
\frac{\rho \lambda_n}{\Gamma_n} \underset{ \frac{a}{\sqrt{\Gamma_n}} 
{\rightarrow} 0}{\longrightarrow} 0
\textcolor{black}{, \, \lambda_n \underset{ \frac{a}{\sqrt{\Gamma_n} }\to 0}  \asymp a \sqrt{\Gamma_n} 
.}
\end{equation}
\end{remark}




\subsection{Technical Lemmas for Unbounded Innovations}
We proceed with the proof of Lemmas  \ref{CTR_T1_N}  and \ref{rn_choice} which are specifically needed for unbounded innovations.

\begin{proof}[Proof of Lemma \ref{CTR_T1_N}]
We use a partition at the threshold $r_{k,n}$ on the variable $U_k$, to control finely the term $ \Delta_k(X_{k-1},U_k)$:
\begin{eqnarray}\label{Holder_Hn}
&&\E \Big [ \exp\big( - \frac{ \rho q \lambda}{ \Gamma_n} \Delta_k(X_{k-1},U_k)\big) | \F_{k-1}\Big ]
 = \E \Big [ \exp\big(-  \frac{ \rho q \lambda}{ \Gamma_n}  \Delta_{n}(X_{k-1},U_k) \big)\mathds{1}_{|U_k|\le r_{k,n} }| \F_{k-1}\Big ]
\nonumber \\
&&+ \E \Big [ \exp\big( -\frac{ \rho q \lambda}{ \Gamma_n} \Delta_k(X_{k-1},U_k) \big) \mathds{1}_{|U_k| > r_{k,n}} | \F_{k-1}\Big  ]=: T_{k,s}^M+T_{k,l}^M,
\end{eqnarray}
where $T_{k,s}^M$ and $T_{k,l}^M$  stand for the contributions in \textcolor{black}{$\E[T_k|\F_{k-1}] \exp(\sum_{i=1}^{k-1} C_{\eqref{EXPR_GRAD}}\gamma_i^{2} V_{i-1})$} associated with the martingale increment $\Delta(X_{k-1},U_k) $ for which the innovation is respectively \textit{small}  and \textit{large}.

Let us first write:
\begin{equation*}
T_{k,s}^M =\E \Big [ \exp\big( - \frac{ \rho q \lambda}{ \Gamma_n}  \Delta_{n}(X_{k-1},U_k) \big)\mathds{1}_{|U_k|\le r_{k,n} }| \F_{k-1}\Big ]=\int_{|u|\le r_{k,n}}\exp\big( - \frac{ \rho q \lambda}{ \Gamma_n}  \Delta_{n}(X_{k-1},U_k) \big) \mu(du).
\end{equation*}
Observe now that, similarly to the computations of Section \ref{U_bounded} for bounded innovations, see equation \eqref{omega1_LIP},  if $|u|\le r_{k,n} $, $u\in \R^r\mapsto  \Delta_{n}(X_{k-1},u) $ is s.t.
\begin{equation}
\label{LIP_UNBOUNDED}
[\Delta_{n}(X_{k-1},\cdot)]_{1,B(0,r_{k,n})}^2\le \gamma_k |\sigma_{k-1}^*\nabla \varphi(X_{k-1})|^2+C\gamma_k^{3/2}r_{k,n}^2+C_{\eqref{EXPR_GRAD}}\gamma_k^{2} V_{k-1},
\end{equation}
where  $[\Delta_{n}(X_{k-1},\cdot)]_{1,B(0,r_{k,n})}^2$ denotes the Lipschitz modulus of $\Delta_{n}(X_{k-1},\cdot)$ restricted on the ball  $B(0,r_{k,n})$ of $\R^r $ with radius $r_{k,n} $. 

Let us now extend $u \in B(0,r_{k,n}) \mapsto \Delta_{n}(X_{k-1},u) $ into a Lipschitz function on the whole set $ \R^r$ which globally verifies the bound of equation \eqref{LIP_UNBOUNDED}. The easiest way to do so is to consider:
$$u\in \R^r \mapsto \bar \Delta(X_{k-1},u)=\Delta\big(X_{k-1},\Pi_{\bar B(0,r_{k,n})}(u)\big),$$
where $\Pi_{\bar B(0,r_{k,n})}(\cdot) $ denotes the projection on $\bar B(0,r_{k,n}) $, namely for $u\in \bar B(0,r_{k,n})$, $\Pi_{\bar B(0,r_{k,n})}(u)=u $, for $u\not \in B(0,r_{k,n})$, $\Pi_{\bar B(0,r_{k,n})}(u)=\frac{u}{|u|}r_{k,n} $. It is readily seen that:
$$[\bar \Delta(X_{k-1},\cdot)]_1^2\le  \gamma_k |\sigma_{k-1}^*\nabla \varphi(X_{k-1})|^2+C\gamma_k^{3/2}r_{k,n}^2+C_{\eqref{EXPR_GRAD}}\gamma_k^{2} V_{k-1}.$$
Hence:
\begin{eqnarray}
T_{k,s}^M&=&\int_{|u|\le r_{k,n}}\exp\big( - \frac{ \rho q \lambda}{ \Gamma_n}  \bar \Delta_{n}(X_{k-1},U_k) \big) \mu(du)
\le \E\Big [ \exp\big( -\frac{ \rho q \lambda}{ \Gamma_n}  \bar \Delta_{n}(X_{k-1},U_k) \big)| \F_{k-1}\Big ]\notag\\
&\underset{\A{GC}}{\le}&\exp\big(\frac{ \rho^2 q^2 \lambda^2}{2 \Gamma_n^2} [\bar \Delta(X_{k-1,\cdot})]_1^2-\frac{ \rho q \lambda}{ \Gamma_n} \E[\bar \Delta(X_{k-1},U_k)]\big)\notag\\
&\le &\exp\big(\frac{ \rho^2 q^2 \lambda^2}{2 \Gamma_n^2} (\gamma_k |\sigma_{k-1}^*\nabla \varphi(X_{k-1})|^2+C\gamma_k^{3/2}r_{k,n}^2+C_{\eqref{EXPR_GRAD}}\gamma_k^{2} V_{k-1})-\frac{ \rho q \lambda}{ \Gamma_n} \E[\bar \Delta(X_{k-1},U_k)]\big).\label{PREAL_BD_TKSM}
\end{eqnarray}
Bearing in mind that $\E[\Delta(X_{k-1},U_k)]=0 $, and observe now:
\begin{eqnarray*}
\E[\bar \Delta(X_{k-1},U_k)]
&=& \E[\Delta(X_{k-1},U_k)\mathds{1}_{|U_k|<r_{k,n}}]
+ \E[\bar \Delta(X_{k-1},U_k)\mathds{1}_{|U_k|\geq r_{k,n}}]\\
&=& -\E[\Delta(X_{k-1},U_k)\mathds{1}_{|U_k|\ge r_{k,n}}]+\E[\Delta\big(X_{k-1},\Pi_{B(0,r_{k,n})}(U_k)\big)\mathds{1}_{|U_k|\ge r_{k,n}}]\\
&\le &\Big(\E[|\Delta(X_{k-1},U_k)|^2]^{1/2}+\E[|\Delta\big(X_{k-1},\Pi_{B(0,r_{k,n})}(U_k))|^2]^{1/2}\Big)\P[|U_k|\ge r_{k,n}]^{1/2}\\
&\le & C\gamma_k^{1/2}\exp(-\frac{r_{k,n}^2}{4}),
\end{eqnarray*}
exploiting \eqref{ineq_dens_inno}, \eqref{def_DELTA_K} and \eqref{def_Xi} for the last inequality. Plugging this bound into \eqref{PREAL_BD_TKSM} yields:
\begin{equation}
\label{CTR_TKS_M}
T_{k,s}^M\le \exp\big(\frac{ \rho^2 q^2 \lambda^2}{2 \Gamma_n^2} (\gamma_k |\sigma_{k-1}^*\nabla \varphi(X_{k-1})|^2+C\gamma_k^{3/2}r_{k,n}^2+ C_{\eqref{EXPR_GRAD}}\gamma_k^{2} V_{k-1})+C \gamma_k^{1/2} \frac{ \rho q \lambda}{ \Gamma_n}\exp(-\frac{r_{k,n}^2}{4})\big).
\end{equation}
The remaining term in \eqref{Holder_Hn}, involving also the large deviations of the innovation, can be controlled as follows:
\begin{eqnarray*}
T_{k,l}^M&=&\E \Big [ \exp\big(- \frac{\rho q \lambda}{ \Gamma_n} \Delta_k(X_{k-1},U_k) \big) \mathds{1}_{|U_k| > r_{k,n}} | \F_{k-1}\Big ]
\nonumber \\
&
\le&\E \Big [ \exp\big( -\frac{2 \rho q \lambda}{ \Gamma_n} \Delta_k(X_{k-1},U_k) \big)  | \F_{k-1}\Big ]^{1/2} \E \Big [ \mathds{1}_{|U_k| > r_{k,n}} \Big ]^{1/2} 
\nonumber \\
&
\leq& 2 \E \Big [ \exp\big( -\frac{2 \rho q \lambda}{ \Gamma_n} \Delta_k(X_{k-1},U_k) \big ) | \F_{k-1}\Big ]^{1/2} \exp(- \frac{ r_{k,n}^2}{4}  )
\nonumber \\
& \overset{\A{GC}}{\leq}& 2 \exp\big( \frac{ \rho^2 q^2 \lambda^2}{ \Gamma_n^2} \gamma_k \|\sigma \|_{\infty}^2 \| \nabla \varphi \|_{\infty}^2 - \frac{ r_{k,n}^2}{4}   \big ).
\end{eqnarray*}
Let us proceed from our definition of $r_{k,n}=r_{k,n}(\A{A},\lambda,\rho)$ in \eqref{rn} and write $r_{k,n}= (1+\frac{ \rho q \lambda}{ \Gamma_n} ) u_{k,n}$. One has that  for all $k \leq n$, $ u_{k,n} \underset{n}{\rightarrow} + \infty$. In other words, with the previous inequality, recalling $r_{k,n}^2\ge u_{k,n}^2(1+\frac{\rho^2q^2\lambda^2}{\Gamma_n^2}) $:
\begin{eqnarray}\label{GC_Delta_barHn1}
T_{k,l}^M
 &\leq& 2 \exp\big( \frac{  q^2 \rho^2 \lambda^2}{ \Gamma_n^2} (\gamma_k \|\sigma \|_{\infty}^2 \| \nabla \varphi \|_{\infty}^2 - \frac{u_{k,n}^2}{4}) - \frac{u_{k,n}^2}{4} \big ) 
\nonumber \\
&  \leq& 2 \exp\big( - C \frac{  q^2 \rho^2 \lambda^2}{ \Gamma_n^2}  \frac{u_{k,n}^2}{8}-\frac{u_{k,n}^2}{4} \big ) 
\leq 2 \exp \big( - C r_{k,n}^2 \big ),
\end{eqnarray}
since $u_{k,n}^2 > 16 \gamma_1  \|\sigma \|_{\infty}^2 \| \nabla \varphi \|_{\infty}^2 $ (\textcolor{black}{which explains our choice for the constant  $r_{k,n}$} in \eqref{rn}).
Plugging \eqref{CTR_TKS_M} and \eqref{GC_Delta_barHn1}  into \eqref{Holder_Hn} yields:
\begin{eqnarray*}
&&\E \Big [ \exp\big( -\frac{\rho q \lambda}{ \Gamma_n} \Delta_k(X_{k-1},U_k)\big) | \F_{k-1}\Big ]
\\ &\leq &
 \exp \Big ( 
\frac{ \rho^2 q^2 \lambda^2 }{ 2 \Gamma_n^2} \big( \gamma_k  | \sigma_{k-1}^* \nabla \varphi_{k-1} |^2 
+ C\gamma_k^{3/2}  r_{k,n}^2 +C_{\eqref{EXPR_GRAD}}\gamma_k^{2} V_{k-1}  
\big)
+
\frac{ \rho q \lambda}{ \Gamma_n} C \gamma_k^{1/2} \exp(-\frac{ r_{k,n}^2}{4}  )
 \Big )
\nonumber \\
&&+ 2 \exp \big( - C r_{k,n}^2 \big )
\nonumber \\
&\leq &
 \exp \big ( 
\frac{ \rho^2 q^2 \lambda^2 }{ 2 \Gamma_n^2} \big( \gamma_k  | \sigma_{k-1}^* \nabla \varphi_{k-1} |^2 
+ C\gamma_k^{3/2}  r_{k,n}^2 +C_{\eqref{EXPR_GRAD}}\gamma_k^{2} V_{k-1}  
\big)
+
\frac{ \rho q \lambda}{ \Gamma_n} C \gamma_k^{1/2} \exp(-\frac{ r_{k,n}^2}{4}  )
 \big )
\nonumber \\
&& \times \big ( 1+ 2 \exp ( - C r_{k,n}^2  ) \big )
\nonumber \\
&=:&\exp \big (\frac{ \rho^2 q^2 \lambda^2 }{2\Gamma_n^2} \gamma_k | \sigma_{k-1}^* \nabla \varphi_{k-1} |^2 +C_{\eqref{EXPR_GRAD}}\gamma_k^{2} V_{k-1}\big )\times \aleph_{k,n} (\lambda, \gamma_k,  r_{k,n}) ,
\end{eqnarray*}
where 
\begin{equation*}
\aleph_{k,n} (\lambda, \gamma_k,  r_{k,n}) 
:=
\Big( 1+ 2\exp \big( - C r_{k,n}^2 \big )\Big) \exp \big (\frac{ \rho q \lambda}{ \Gamma_n} C \gamma_k^{1/2} \exp(-\frac{ r_{k,n}^2}{4}  ) \big) 
 \exp \big ( \frac{ \rho^2 q^2 \lambda^2 }{\Gamma_n^2} C  \gamma_k^{3/2}  r_{k,n}^2
\big ).
\end{equation*}
We have thus isolated the ``significant" term $\exp \big (\frac{ \rho^2 q^2 \lambda^2 }{2\Gamma_n^2} \gamma_k | \sigma_{k-1}^* \nabla \varphi_{k-1} |^2 \big )$.
The result follows from the definition of $T_k$ in \eqref{Def_T}.
\end{proof}

\begin{proof}[Proof of Lemma \ref{rn_choice}]
We recall for convenience the definition of $r_{k,n}$ in \eqref{rn}:
\begin{equation*}
r_{k,n}=
\left\{
\begin{array}{ll}
&r_n= C (1+\frac{\rho q \lambda}{\Gamma_n} )  (\frac{\Gamma_n}{\Gamma_n^{(3/2)}} )^{1/4}, \ for \ \theta \in (\frac 13, 1),
\\
& C (  (1+\frac{\rho q \lambda}{\Gamma_n} ) \ln(n+1)^{1/4} \ln(k+1)^{1/2}
,  \ for \ \theta=1.
\end{array}
\right. 
\end{equation*}
From the above definition of $\aleph_{k,n} (\lambda, \gamma_k,  r_{k,n}) $, we introduce two remainders:
\begin{eqnarray}\label{prod_Aleph}
\mathcal R_n^1 &:=& \prod_{k=1}^n \Big( 1+ 2\exp \big( - C r_{k,n}^2 \big )\Big)
 \times \prod_{k=1}^n \exp \big (\frac{  \rho q \lambda}{ \Gamma_n} C \gamma_k^{1/2} \exp(-\frac{ r_{k,n}^2}{4}  ) \big)
\nonumber \\
&=:&\mathcal R_n^{11} \times \mathcal R_n^{12},
\nonumber \\
\mathcal R_n^2 &:=& 
 \exp \big ( \frac{ \rho^2 q^2 \lambda^2 }{\Gamma_n^2} C   \sum_{k=1}^n \gamma_k^{3/2} r_{k,n}^2
\big ),
\end{eqnarray}
that naturally appear when we iterate Lemma \ref{CTR_T1_N}. Precisely:
\begin{equation}\label{Sn}
\E[S_n] \leq 
\E[S_{n-1}] \aleph_{n,n} (\lambda, \gamma_n,  r_{n,n}) 
\leq \prod_{k=1}^n \aleph_{k,n} (\lambda, \gamma_k,  r_{k,n}) = \mathcal R_n^1 \times \mathcal R_n^2. 
\end{equation}

$ \bullet$ For $\theta \in (1/3,1)$:

We have chosen $r_{k,n}=r_n$ in \eqref{rn}, so, $\mathcal R_n^1 = \Big( 1+ 2\exp \big( - C r_n^2 \big )\Big)^n \exp \big (\frac{  \rho q \lambda}{ \Gamma_n} C \Gamma_n^{(1/2)} \exp(-\frac{r_{n}^2}{4}) \big)=\mathcal R_n^{11}\mathcal R_n^{12}$, and:
\begin{eqnarray*}
\mathcal R_n^{11} &=&\Big( 1+ 2\exp \big( - C r_{n}^2 \big ) \Big)^n 
\leq \exp \Big ( n \ln \big ( 1+2 \exp  ( - C   ( \frac{\Gamma_n}{\Gamma_n^{(3/2)}} )^{1/2}  ) \big )  \Big )
\\
&&\leq
\exp \big ( 2 n  \exp  ( - C  ( \frac{\Gamma_n}{\Gamma_n^{(3/2)}} )^{1/2}  )  \big ) 
\underset{n \to + \infty}{\longrightarrow} 1,
\end{eqnarray*}
as, for $\theta \in (1/3,1)$, $ \left( \frac{\Gamma_n}{\Gamma_n^{(3/2)}} \right)^{1/2}\ge \eta(n)$ where
\begin{equation}\label{eta_def}
\eta(n):=C\big(n^{(1-\theta)/2}\mathds{1}_{\theta\in (2/3,1)}+  (n^{1/3} \ln(n)^{-1}  )^{\frac 12} \mathds{1}_{\theta=2/3}+n^{\theta/4}\mathds{1}_{\theta\in(1/3,2/3)}\big),
\end{equation}
see also Remark \ref{r_n_aleph}. 
For the remaining of the proof, we will thoroughly exploit this kind of arguments. Precisely, we recall that:
\begin{equation}\label{GR_COMP}
\forall \zeta\in \R_+,\ \exists C_\zeta\ge 1,\  s.t.\  \forall 0\le \beta\le \zeta, \forall x\in \R_+,\ x^\beta\exp(-x^2)\le C_\zeta\exp(-C_\zeta^{-1}x^2).
\end{equation}
Thus, for the remaining term ${\mathcal R}_{n}^{12} $ in ${\mathcal R}_{n}^{1} $, exploiting \eqref{GR_COMP}\textcolor{black}{, up to a modification of the constant $C>0$ from line to line}:
\begin{eqnarray*}
{\mathcal R}_{n}^{12} &\le&\exp \big (\frac{  \rho q \lambda}{ \Gamma_n} C \Gamma_n^{(1/2)} \exp(-\frac{r_{n}^2}{4}) \big) \\
&\leq& \exp \Big (\frac{  \rho q \lambda}{ \Gamma_n} C \Gamma_n^{(1/2)}
 \exp \big (- \frac{\rho^2 q^2 \lambda^2}{\Gamma_n^2} \frac C{4} (\frac{\Gamma_n}{\Gamma_n^{(3/2)}} )^{1/2} \big) 
 \exp \big (-\frac C{4} (\frac{\Gamma_n}{\Gamma_n^{(3/2)}} )^{1/2} \big) 
\Big) 
\\
&\leq &\exp \Big ( C \Gamma_n^{(1/2)} (\frac{\Gamma_n^{(3/2)}}{\Gamma_n})^{1/4}
 \exp \big (-  \frac{\rho^2 q^2 \lambda^2}{\Gamma_n^2} \frac C{8} (\frac{\Gamma_n}{\Gamma_n^{(3/2)}} )^{1/2} \big) 
 \exp  (-\frac C{4} (\frac{\Gamma_n}{\Gamma_n^{(3/2)}} )^{1/2} ) 
\Big) 
\\
&\leq &\exp \big ( C \Gamma_n^{(1/2)}
 \exp  (-\frac C{8} (\frac{\Gamma_n}{\Gamma_n^{(3/2)}} )^{1/2} ) 
\big) 
\le \exp\big(Cn^{1-\theta/2}e^{-\eta(n)}\big)\overset{\eqref{eta_def},\eqref{GR_COMP}}{\underset{n \to + \infty}{\longrightarrow}} 1.
\end{eqnarray*}
Hence $\mathcal R_n^1=\mathcal R_n^{11}\mathcal R_n^{12} \underset{n}{\rightarrow} 1$.

Introducing 
\begin{equation*}
e_{n}^{\theta <1} :=C  \sqrt{ \frac{\Gamma_n^{(3/2)}} {\Gamma_n}}  \underset{n \to + \infty}{\longrightarrow} 0,
\end{equation*}
we get from the definition of $r_{k,n} $ in \eqref{rn} and \eqref{prod_Aleph}
\begin{equation*}
\mathcal R_n^2 \leq 
 \exp \Big ( \big(\frac{ \rho^2 q^2 \lambda^2 }{\Gamma_n}    +\frac{ \rho^4 q^4 \lambda^4 }{\Gamma_n^3} \big)e_{n}^{\theta <1}   \Big ).
\end{equation*}
Recalling, from \eqref{Sn}, that $\E[S_n]\le \mathcal R_n^1\mathcal R_n^2 $, \textcolor{black}{we thus get from the above computations}:
$$\E[S_n]\le  {\mathscr R}_n \exp \big ( (\frac{ \rho^2 q^2 \lambda^2 }{\Gamma_n}    +\frac{ \rho^4 q^4 \lambda^4 }{\Gamma_n^3} )e_{n}   \big ),$$
which gives the result for $\theta\in  (\frac 13,1) $.

$ \bullet$ For $\theta =1$, using the previous notations:
\begin{eqnarray*}
0 &\le& \ln(\mathcal R_n^{11})\leq 
 \ln \Big (\prod_{k=1}^n \big( 1+ 2\exp ( - C r_{k,n}^2  )\big)  \Big)
=   \sum_{k=1}^n \ln \Big( 1+ 2\exp \big( - C r_{k,n}^2 \big )\Big) 
\nonumber \\
&\leq& C \sum_{k=1}^n \exp \big( - C r_{k,n}^2 \big ) 
\leq C \sum_{k=1}^n \exp \big ( - C \ln(n+1)^{1/2} \ln(k+1) \big)  
\nonumber \\
&=& C \sum_{k=1}^n (k+1)^{ - C \ln(n+1)^{1/2}   } 
\leq C 2^{ - C \ln(n+1)^{1/2}   }  + C \int_2^n x^{ - C \ln(n+1)^{1/2}   }  dx
\nonumber \\
&=& C 2^{ - C \ln(n+1)^{1/2} }  + C \frac{2^{ 1 - C \ln(n+1)^{1/2} }- n^{ 1- C \ln(n+1)^{1/2} }}{ C \ln(n+1)^{1/2} -1}
\underset{n \to + \infty}{\longrightarrow} 0.
\end{eqnarray*}
We deduce that: 
\begin{equation*}
{\mathcal R}_n^{11}=\prod_{k=1}^n \Big( 1+ 2\exp \big( - C r_{k,n}^2 \big )\Big) \underset{n \to + \infty}{\longrightarrow} 1.
\end{equation*}

Like for the case $\theta \in (1/3,1)$, 
we get the control:
\begin{eqnarray*}
{\mathcal R}_n^{12}&\le & \prod_{k=1}^n \exp \big (\frac{  \rho q \lambda}{ \Gamma_n} C \gamma_k^{1/2} \exp(-\frac{r_{k,n}^2}{4}) \big) 
\nonumber \\
&\leq& \prod_{k=1}^n \exp \Big (\frac{  \rho q \lambda}{ \Gamma_n} C \gamma_k^{1/2}
 \exp  (-  \frac C{4}   [1+(\frac{\rho q \lambda}{\Gamma_n}  )^2  ] \ln(n+1)^{1/2} \ln(k+1)
 \big) 
\Big) 
\\
&\underset{\eqref{GR_COMP}}{\leq}& \prod_{k=1}^n \exp \Big (C \gamma_k^{1/2}
 \exp \big (-  \frac C{4}  [1+C (\frac{\rho q \lambda}{\Gamma_n} )^2 ] \ln(n+1)^{1/2} \ln(k+1)
\big) 
\Big) \\
&\leq &\prod_{k=1}^n \exp \big (C \frac 1{k^{1/2}} (k+1)^{-  C \ln(n+1)^{1/2} }
\big) 
 \underset{n \to + \infty}{\longrightarrow} 1,
\end{eqnarray*}
\textcolor{black}{in fact,} we have already established for the control of $ {\mathcal R}_n^{11}$ that $\sum_{k=1}^n (k+1)^{ - C \ln(n+1)^{1/2}   }   \underset{n}{\longrightarrow} 0$.
Thus, we proved that $\mathcal R_n^1=\mathcal R_n^{11}\mathcal R_n^{12} \underset{n}{\rightarrow} 1$. 

Let us now turn to the other contribution in \eqref{prod_Aleph}:
\begin{eqnarray*}
\mathcal R_n^2&=&
 \exp \big ( \frac{ \rho^2 q^2 \lambda^2 }{\Gamma_n^2} C   \sum_{k=1}^n \gamma_k^{3/2} r_{k,n}^2
\big )
\\
&\leq&
 \exp \Big  ( 2C (\frac{ \rho^2 q^2 \lambda^2 }{\Gamma_n^2} + \frac{ \rho^4 q^4 \lambda^4 }{\Gamma_n^4}   ) \ln(n+1)^{1/2} \sum_{k=1}^n \gamma_k^{3/2} \ln(k+1)  
\Big )
\\
&\leq&
 \exp \big( 2C (\frac{ \rho^2 q^2 \lambda^2 }{\Gamma_n}+\frac{ \rho^4 q^4 \lambda^4 }{\Gamma_n^3})   \frac{ \sum_{k=1}^n \gamma_k^{3/2} \ln(k+1)}{\sqrt{\Gamma_n}} 
\big )
\\
&=&\exp \big (  (\frac{ \rho^2 q^2 \lambda^2 }{\Gamma_n} +  \frac{ \rho^4 q^2 \lambda^4 }{\Gamma_n^{3}}) e_{n}^{\theta =1} \big ),
\end{eqnarray*}
with $e_{1,n}^{\theta =1}:= 2C    \frac{1}{\sqrt{\Gamma_n}}  \sum_{k=1}^n \frac{ \ln(k+1)}{k^{3/2}}  \underset{n \to + \infty}{\longrightarrow} 0$,
\textcolor{black}{and recalling that $\sqrt{\Gamma_n} \asymp \ln(n+1)^{1/2}$ for the last inequality.}

To sum up, 
 for all $\theta \in (\frac 13,1]$, \textcolor{black}{thanks to inequality} \eqref{Sn}, we know that there exist non negative sequences $(\mathscr R_n)_{n \geq 1}$, $(e_{n})_ {n\geq 1}$,  s.t. $\mathscr R_n \underset{n}{\longrightarrow} 1$, $e_n \underset{n}{\longrightarrow} 0$, and
\begin{equation*}
\E[ S_n ] \leq \mathscr R_n  \exp \Big ( \big(\frac{ \rho^2 q^2 \lambda^2 }{\Gamma_n}    +\frac{ \rho^4 q^4 \lambda_n^4 }{\Gamma_n^3}\big) e_{n}  \Big ).
\end{equation*}
\end{proof}

\section{Regularity Results and Consequences}\label{sec_Reg}
\setcounter{equation}{0}

This section is devoted to some regularity results for the Poisson problem
\begin{equation}\label{Poisson_GEN}
\mathcal A \varphi= f -\nu(f).
\end{equation}

 In particular, we state  below some Schauder like controls, which are, because of our methodology that requires pointwise controls of the derivatives, more adapted than the standard Sobolev estimates (see e.g. Pardoux and Veretennikov \cite{pard:vere:01}). 

There are two kinds of assumptions that guarantee the solution $\varphi $ of \eqref{Poisson_GEN} enjoys the required smoothness of Theorems \ref{THM_CARRE_CHAMPS}.
\begin{trivlist}
\item[(a)] If $b$, $\sigma $ and $f$ are \textit{smooth}, under suitable confluence like conditions stated  below in \A{D${}_\alpha^p$} \textcolor{black}{with the condition on $D\sigma$: $\|D\sigma\|_\infty^2\le \frac{2\alpha}{2(1+\beta) -p} $ for given $(\alpha,p) \in (0,+\infty) \times [1,2)$}, the probabilistic representation $\varphi(x)=-\int_{0}^{+\infty} \E\big[\big(f(Y_t^{0,x})-\nu(f)\big)\big]dt $ of the solution of \eqref{Poisson_GEN} can be differentiated using iterated tangent flows to establish that $\varphi\in {\mathcal C}^{3,\beta}(\R^d,\R) $, for some $\beta\in (0,1] $. It suffices for that to have $f\in C^{3,\beta}(\R^d,\R) $ and $b\in  C^{3,\beta}(\R^d,\R^d), \sigma\in C_b^{3,\beta}(\R^d,\R^d\otimes \R^r)$. We refer to Section \textcolor{black}{2.2 and 5.1} of \cite{hon:men:pag:16} for additional details. Importantly, under such assumptions,  specific non-degeneracy conditions are not needed.

\item[(b)] If we do not assume such an \textit{a priori} smoothness on $f,b,\sigma $, we need to make an \textit{extra} non-degeneracy assumption which will allow some regularity gain (elliptic bootstrap)  and a \textcolor{black}{stronger} confluence like condition \A{D${}_\alpha^p$} \textcolor{black}{with $\|D\sigma\|_\infty^2\le \frac{2\alpha}{2(3+\beta) -p} $ for given $(\alpha,p) \in (0,+\infty) \times [1,2)$}. Precisely, assuming that the bounded diffusion coefficient $\Sigma $ is also uniformly elliptic, we exploit the results of Krylov and Priola  \cite{kryl:prio:10} (Theorems 2.4-2.6) to derive that, up to an additional technical condition on $\Sigma $ when $d>1$, we actually have $\varphi\in {\mathcal C}^{3,\beta}(\R^d,\R) $ for some $\beta\in (0,1)$, as soon as $f\in C^{1,\beta}(\R^d,\R) $ and $b\in  C^{1,\beta}(\R^d,\R^d), \sigma\in C_b^{1,\beta}(\R^d,\R^d\otimes \R^d)$. 
\end{trivlist}

We point out that the second set of assumptions is very important in order to go towards the \textit{usual} setting of functional/transport inequalities which typically involves Lipschitz continuous test functions. This is for instance the case for the controls of the Wasserstein distances between the law of $Y_t$ in \eqref{eq_diff} and the invariant measure $\nu$ (see e.g. \cite{bakr:gent:ledo:14}). We are thus able, in the non-degenerate framework (b), to consider directly sources $f\in {\mathcal C}^{1,\beta}(\R^d, \R) $. Observe that, when $\beta\rightarrow 0$, we are \textit{almost} Lipschitz. Actually, we can, in this setting, handle functions $f\in {\mathcal C}^{0,1}(\R^d, \R) $
 up to a spatial regularization which leads to a constraint on the steps (see Theorem \ref{THEO_CTR_LIP} below).
 
 \subsection{Assumptions and Regularity Results} 
The regularity results  stated here can be found in Section 5 of \cite{hon:men:pag:16}. Let us now recall the useful assumptions needed.

%
\begin{trivlist}
\item[\A{UE}] \label{LABEL_UE}Uniform ellipticity. We assume that $r\ge d $ 
in \eqref{eq_diff} and that there is $\underline{\sigma}>0$ such that
\begin{equation*}
 \forall \xi \in \R^d,\ \langle \sigma\sigma ^*(x)\xi,\xi\rangle \ge \underline{\sigma} |\xi|^2.
\end{equation*}

For 
$\beta \in (0,1) $, we introduce the following condition.  
\item[\A{R${}_{1,\beta}$}] Regularity Condition. 
From equation \eqref{eq_diff}, we suppose $b \in {\mathcal C}^{1,\beta}(\R^d,\R^d),\sigma \in {\mathcal C}_b^{1,\beta}(\R^d,\R^d)$. 
\medskip
\item[\A{D${}_\alpha^p$}]Confluence Conditions. We assume that there exists $\alpha >0 $ and $ p\in [1,2)$ such that  for all $x\in \R^d $, $\xi\in \R^d $
\begin{equation}
\Big \langle \frac {Db(x)+Db(x)^*} 2\xi,\xi \Big\rangle 
+\frac 12 \sum_{j=1}^r \Big( (p-2)\frac{|\langle  D\sigma_{\cdot j}(x) \xi, \xi\rangle|^2}{|\xi|^2}+|D\sigma_{\cdot j} \xi|^2
\Big)\le -  \alpha |\xi|^2, 
\end{equation}
where $Db$ stands here  for the Jacobian of $b$,
$\sigma_{\cdot j} $ stands for the $j^{{\rm th}} $ column of the diffusion matrix $\sigma $ and $D\sigma_{\cdot j} $ for its  Jacobian matrix.
\end{trivlist}
There are others assumptions than \A{D${}_\alpha^p$} which yield, in the non-degenerate setting, 
\textcolor{black}{gradient control.}
 This is the case for the so-called Bakry and \'Emery curvature criterion (\cite{bakr:emer:85, bakr:gent:ledo:14}) which is however pretty hard to check for general multidimensional diffusion coefficients. \textcolor{black}{However for H\"older control of the gradient, this critetion seems to be not adapted, see \cite{hon:men:pag:16v2} Section 2.2.2 for more details.}
\\

We eventually introduce, as in \cite{hon:men:pag:16}, a technical condition on the diffusion coefficient $\sigma$. It allows to prove that each partial derivative $\partial_{x_i} \varphi $ of the solution of \eqref{Poisson_GEN} satisfies an autonomous scalar Poisson problem.  We suppose:
\\

\A{$\Sigma$} \ for every  $(i,j)\in \leftB 1,d\rightB^2 $ and $x=(x_1, \hdots, x_d) \in \R^d$, $\Sigma_{i,j}(x)=\Sigma_{i,j}(x_{i\wedge j},\cdots,x_d) $.\\

We say that assumption \A{P${}_\beta$} is satisfied if \A{UE},  \A{D${}_\alpha^p$} with \textcolor{black}{$\|D\sigma\|_\infty^2\le \frac{2\alpha}{2(1+\beta) -p} $}, \A{R${}_{1,\beta}$}, and \A{$\Sigma$}  are in force. From Section 5.3 of \cite{hon:men:pag:16} we have the following result.
\begin{theo}[Elliptic Bootstrap in a non-degenerate setting]\label{TH_BOOTSTRAP}
Assume \A{P${}_\beta $} holds  for some $\beta\in (0,1) $ and that $f\in {\mathcal C}^{1,\beta}(\R^d,\R) $. Then, there is a unique  $\varphi\in {\mathcal C}^{3,\beta}(\R^d,\R) $ solving \eqref{Poisson_GEN}.
\end{theo}
Note as well from Remark \ref{Krylov_Priola}, that there exists $C>0,\ |D^2\varphi(x)|\le C(1+|x|)^{-1} $. In other words, the solution $\varphi $ of \eqref{Poisson_GEN} satisfies \A{T${}_\beta $}.
\begin{remark}[On Schauder estimates for $\beta=1 $]
We insist on the fact that $\beta\in (0,1) $ in the above theorem. Indeed, it is well known that the H\"older exponent $\beta $ cannot go to $1$ in the Schauder estimates.  Note that, for the particular case  
$f\in  {\mathcal C}^{1,1}(\R^d,\R)$, we also have $f\in  {\mathcal C}^{1,\beta}(\R^d,\R) $ for all $\beta\in (0,1) $. This means that, for such $f$, the elliptic bootstrap works up to an arbitrarily small correction. 
\end{remark}

From Theorem \ref{TH_BOOTSTRAP} we readily have:
\begin{corol}[Smoothness for the Poisson problem with \textit{Carr\'e du champ} source]
\label{COROL_BOOTSTRAP_CC}
Assume \A{P${}_\beta $} holds  for some $\beta\in (0,1)$ and that $f\in {\mathcal C}^{1,\beta}(\R^d,\R) $. Then, there is a unique  $ \vartheta\in {\mathcal C}^{3,\beta}(\R^d,\R) $ solving 
$${\mathcal A} \vartheta=|\sigma^*\nabla \varphi|^2- \nu(|\sigma^*\nabla \varphi|^2)$$
and satisfying \A{T${}_\beta $}.
\end{corol}
Indeed, it suffices to observe from Theorem \ref{TH_BOOTSTRAP} and the assumption \A{R${}_{1,\beta}$} in \A{P${}_\beta $} that $\tilde f:=|\sigma^*\nabla \varphi|^2 \in {\mathcal C}^{1,\beta}(\R^d,\R)$ and to apply again Theorem \ref{TH_BOOTSTRAP} for this source.\\


\subsection{Concentration bounds for a Lipschitz source in a non-degenerate setting}


As indicated in the introduction of the Section, we aim at controlling deviations for Lipschitz sources. In the current Lipschitz framework we aim to address, we need a slightly different set of assumptions. Namely, we will assume that \A{C1}, \A{GC}, \A{C2}, \A{${\mathbf {\mathcal L_V}}$}, \A{U}, \A{S} and \A{P${}_\beta$} 
are in force and we will say that \A{L${}_\beta $} holds. Under this new assumption, we have the following result.
\begin{theo}[Non-asymptotic concentration bounds for Lipschitz continuous source] \label{THEO_CTR_LIP}
Assume that \A{L${}_\beta$} 
is in force. Let $f$ be a Lipschitz continuous function. 
For a time step sequence $(\gamma_k)_{k\ge 1} $ of the form $\gamma_k\asymp k^{-\theta} $, $\theta\in (1/2,1] $, we have that, there exist two explicit monotonic sequences  
$c_n\le 1\le C_n,\ n\ge 1$,   with $\lim_n C_n = \lim_n c_n =1$
such that for all $n \geq 1$ and for every $a >0$:  
 \begin{eqnarray}
\P\big[ |\sqrt{\Gamma_n}\big(\nu_n( f)-\nu(f)\big) | \geq a \big] \leq 2 C_n \exp \big( - c_n \frac{a^2}{2\nu(|\sigma^*\nabla \varphi|^2)} 
\big), \label{BD_LIP}
\end{eqnarray}
where $\varphi \in \mathcal C^{0,1}(\R^d,\R)\cap W_{2,loc}^{2}(\R^d,\R)$ is a weak solution of the Poisson equation $\mathcal A \varphi= f -\nu (f )$.

\end{theo}
\begin{proof}[Sketch of the proof]
To prove the above result, the starting point consists in regularizing the source $f$ by mollification. Namely, we consider $f_\delta=f\star \eta_\delta $, where $\star$ denotes the usual convolution, for a suitable mollifier $\eta_\delta(\cdot):=\frac{1}{\delta^{d}}\eta(\frac{\cdot}{\delta}),\ \delta>0 $, where $\eta $ is a compactly supported non-negative function s.t. $\int_{\R^d}\eta(x)dx=1 $. We then write $\nu_n(f)-\nu(f)=\nu_n(f_\delta)-\nu(f_\delta)+(\nu_n-\nu)(f-f_\delta)=:\nu_n\big(f_\delta-\nu(f_\delta)\big)+{R_{n,\delta}} $. We aim at letting $\delta $ go to 0 so that ${R_{n,\delta}}  $ can be viewed as a remainder. On the other hand, we will apply the same strategy as in the proof of Theorem \ref{THM_CARRE_CHAMPS} to analyze the deviations of $\nu_n\big(f_\delta-\nu(f_\delta)\big)=\nu_n({\mathcal A}\varphi_\delta) $. Precisely, reproducing the arguments of Section 5.4 of \cite{hon:men:pag:16} to equilibrate the explosions of the derivatives of $\varphi_\delta $ in the proof of Theorem \ref{THM_CARRE_CHAMPS} yields that there exists two explicit monotonic sequences  
$\tilde c_n\le 1\le C_n,\ n\ge 1$,   with $\lim_n C_n = \lim_n \tilde c_n =1$ s.t.
\begin{equation}\label{PREAL_LIP}
\P\big[ |\sqrt{\Gamma_n}\big(\nu_n( f)-\nu(f)\big) | \geq a \big] \leq 2 C_n \exp \big( - \tilde c_n \frac{a^2}{2\nu(|\sigma^*\nabla \varphi_\delta|^2)}).
\end{equation}
From the previous Schauder estimates, we know that $\varphi_\delta\in {\mathcal C}^{3,\beta}(\R^d,\R) $ for all $\delta>0 $ with explosive ${\mathcal C}^{3,\beta} $ norm in $\delta $ but with bounded gradient. Recall indeed that, for all $\beta\in (0,1)$,
\begin{equation}
\label{EXPLO_SCHAUDER_CONTROLS_AND_GD_BOUNDED}
\|f_\delta\|_{{\mathcal C}^{1,\beta}}\le C\delta^ {-\beta},\ |\nabla \varphi_\delta|\le \frac{[f_\delta]_1}{\alpha}=\frac{[f]_1}{\alpha}.  
\end{equation}
 We again refer to Lemma 6 and Section 5.4 in \cite{hon:men:pag:16} for details. On the other hand, it is well known, see e.g. \cite{pard:vere:01}, that $\varphi_\delta(x)=-\int_0^{+\infty}\E[f_\delta(Y_t^{0,x})-\nu(f_\delta)] dt $. From their Proposition 1, we have in our case that, denoting by $\nu_{Y_t^{0,x}} $ the law of $Y_t^{0,x} $ we have the following control for the total variation between $\nu_{Y_t^{0,x}} $ and $\nu $. There exists constants $(C,c):=(C,c)($\A{L${}_\beta$}$)$ s.t.
\begin{equation}\label{TV_CONTROL}
\|\nu_{Y_t^{0,x}}-\nu \|_{T.V.}\le C\exp(c|x|)\exp(-\alpha t).
\end{equation}
Introducing now $\bar f_\delta=f_\delta-f-\nu\big(f_\delta-f \big) $, we rewrite that, for all $x\in \R^d $:
\begin{eqnarray*}
\big(\varphi_\delta-\varphi\big)(x)&=&-\int_{0}^{+\infty} \E[\bar f_\delta(Y_t^{0,x})] dt,\\
|\big(\varphi_\delta-\varphi\big)(x)|&\le& \int_0^{+\infty} \Big(\int_{\R^d} |\bar f_\delta(y)|^2 \big( \nu_{Y_t^{0,x}}+\nu\big) (dy)\Big)^{1/2} \|\nu_{Y_t^{0,x}}-\nu\|_{T.V.}^{1/2} dt\\
&\underset{\eqref{TV_CONTROL}}{\le} & C^{1/2}\exp(\frac c 2|x|)\|\bar f_\delta \|_\infty \int_0^{+\infty} \exp\Big(-\frac \alpha 2 t \Big) dt.
\end{eqnarray*}
Recalling that $f$ is Lipschitz and that $(f-f_\delta)(x)=\int_{\R^d} \big(f(x-y)-f(x)\big)\eta_\delta(y) dy $ , we actually have $\|\bar f_\delta \|_\infty \le C[f]_1\delta $, which establishes the  pointwise convergence
 $\varphi_\delta(x)\underset{\delta\rightarrow 0}{\longrightarrow }-\int_0^{+\infty}\E[f(Y_t^{0,x})-\nu(f)] dt=:\varphi(x)  $ which is the only weak solution of \eqref{Poisson_GEN} in $W_{p,loc}^2(\R^d,\R),p>1 $ (see Theorem 1 in  \cite{pard:vere:01}). 

\textcolor{black}{
Let us now prove that 
\begin{equation}
\label{CV_CARRE_CHAMP}
\lim_{\delta\to 0} \nu(|\sigma^* \nabla \varphi_\delta|^2) =\nu(|\sigma^* \nabla \varphi|^2).
\end{equation}
For all $\varepsilon>0$, there is a compact $K:=K(\varepsilon)$ such that, denoting by $K^c:=\R^d\backslash K $, we have: 
\begin{equation*}
\int_{\R^d}  \big | \nabla \varphi_\delta- \nabla \varphi \big |^2(x) \mathds 1_{K^c}(x) \nu(dx)
\leq
 4 \|\nabla \varphi \|_\infty^2
\int_{\R^d}   \mathds 1_{K^c}(x) \nu(dx)
\leq  \frac \varepsilon 2.
\end{equation*}
Also, since $\varphi\in W_{2,loc}^2(\R^d,\R)$, we write:
\begin{eqnarray*}
&&\int_{\R^d}  \big | \nabla \varphi_\delta- \nabla \varphi \big |^2(x) \mathds 1_{K}(x) \nu(dx)
\leq
\int_{\R^d}  \Big | \int_{\R^d} \big(  \nabla \varphi(x-z)- \nabla \varphi (x) \big )\rho_\delta(z) dz \Big |^2 \mathds 1_{K}(x) \nu(dx)
\nonumber \\
&& =
\int_{\R^d}  \Big | \int_{\R^d}  \big(\int_0^1 D^2 \varphi (x-\lambda z)  z  d\lambda\big) \rho_\delta(z) dz \Big |^2 \mathds 1_{K}(x) \nu(dx)
\nonumber \\
&& \leq \int_0^1 d\lambda \int_{\R^d} \Big | \int_{\R^d}   D^2 \varphi (x-\lambda z) z \rho_\delta(z) dz\Big|^2\mathds 1_{K}(x) \nu(dx)\\
&& \leq 
\int_0^1 d\lambda \Big (\int_{\R^d}   dz |z|^2  \rho_\delta(z) \int_{\R^d}   |D^2 \varphi (x-\lambda z)|^2      \mathds 1_{K}(x) \nu(dx) \Big ) 
\leq C \delta^2 \|\varphi \|_{W_{2}^2(\bar K,\R)}^2< \frac \varepsilon 2,
\end{eqnarray*}
for $\delta$ small enough,
 using the Cauchy-Schwarz inequality for the penultimate control and denoting by $\bar K $ a compact set such that for all $z\in B(0,C\delta) \supset {\rm supp}({\eta_\delta})$, $x\in K$, $ x-z\in \bar K$. This in particular gives \eqref{CV_CARRE_CHAMP}.
}
Hence, setting $$c_n:= \tilde c_n \frac{\nu(|\sigma^*\nabla \varphi|^2)}{ \nu(|\sigma^*\nabla \varphi_\delta|^2)} \to_n 1,$$
and recalling from Section 5.4. in \cite{hon:men:pag:16} that $\delta:=\delta(n)\to_n 0 $\footnote{which was anyhow constrained to go to 0  sufficiently \textit{slowly}  in order to balance the explosions in the derivatives coming from the Schauder estimates, see \eqref{EXPLO_SCHAUDER_CONTROLS_AND_GD_BOUNDED}. It is specifically this feature that led to the condition $\gamma_n \asymp n^{-\theta}, \theta\in (1/2,1]$.},
we derive that \eqref{BD_LIP} follows from \eqref{PREAL_LIP} up to a modification of $\tilde c_n $. 

Furthermore, let us point out that, the result can alternatively be stated replacing the \textit{carr\'e du champ} in \eqref{BD_LIP} by the variance of the Lipschitz source under the invariant law. In fact, we can write by the dominated convergence theorem:
\begin{eqnarray}\label{dominate_delta}
&&\lim_{\delta \to 0} \nu(|\sigma^*\nabla \varphi_\delta|^2)
=\nu(|\sigma^*\nabla \varphi|^2).
\end{eqnarray}
\end{proof}

\begin{remark}
The new threshold $\theta> \frac 12$ comes from the specific Lipschitz regularity of the test function $f$. Intuitively, this threshold naturally appears when we consider $\beta\rightarrow 0 $ in the previous condition $\theta\in (\frac{1}{2+\beta},1] $ induced by the regularity of $\varphi \in {\mathcal C}^{3,\beta}(\R^d,\R) $ which holds, under \A{P${}_\beta $}, when $f\in {\mathcal C}^{1,\beta}(\R^d,\R) $. We underline anyhow that, for $\beta=0$, the Schauder estimates do not directly apply.
\end{remark}

\section{Optimisation over $\rho$ under \textit{Gaussian} and \textit{super Gaussian deviations}}\label{sec_optim}

\textcolor{black}{
In Lemma \ref{RHO}, we performed an asymptotic estimation of the upper-bound for \textit{Gaussian deviations}.
However, from a numerical point of view, it appears to be more significant to optimize over $\rho$ in whole generality, i.e. not only when $\frac a {\sqrt{\Gamma_n}} \to 0$. 
This procedure conducts to deviations bounds that are much closer to the realizations.
}

\textcolor{black}{
In particular, in \textit{super Gaussian deviations} framework (i.e. $\frac a{\sqrt{\Gamma_n} } \to 0$), we provide here a ``weaker" concentration inequality than the Gaussian one, which precisely comes from the optimization over $\rho$ for this regime.
 This loss of concentration, with the terminology of Remark \ref{LA_RQ_SUR_LES_REGIMES}, is intrinsic to our method, as it will be shown in the proofs of Theorem \ref{THM_CARRE_CHAMPS_super_Gaussian} and Lemma \ref{RHO_optim} below, see also Remark \ref{remark_bootstrap}}.

\begin{theo}[Deviations in the \textit{super Gaussian} regime]\label{THM_CARRE_CHAMPS_super_Gaussian}
Assume \A{A} is in force. If there exists $\vartheta \in \mathcal C^{3,\beta}(\R^d,\R)$, $\beta\in (0,1] $ satisfying \A{T${}_\beta$} s.t. 
\begin{equation}\label{Poisson}
\mathcal A \vartheta= |\sigma^*  \nabla \varphi |^2 -\nu(|\sigma^*  \nabla \varphi |^2),
\end{equation}
then, for $ \theta \in (\frac{1}{2+\beta}, 1 ]$
, there exist explicit non-negative sequences 
$( c_n)_{n\ge 1}$ and $( C_n)_{n\ge 1}$, respectively increasing and decreasing for $n$ large enough, with $\lim_n  C_n = \lim_n c_n =1$ 
s.t. for all $n\ge 1$, $a>0$, the following bounds hold.
When $\frac{a }{\sqrt{\Gamma_n}} \rightarrow + \infty$ (\textit{Super Gaussian deviations}):
\begin{equation*}
\P\big[ |\sqrt{\Gamma_n}\nu_n( \mathcal{A} \varphi )| \geq a \big] 
\leq 2\,  C_n \exp \big(\! -  c_n\frac{a^{4/3} \Gamma_n^{1/3}}{2 \| \sigma \|_{\infty}^{2/3} [\vartheta]_1^{2/3} }
\big). 
\end{equation*}
\end{theo}
\begin{remark}
Observe from Corollary \ref{COROL_BOOTSTRAP_CC}, that, the function $\vartheta $ enjoys the required smoothness as soon as assumption \A{P${}_\beta$}  introduced in Section \ref{sec_Reg} holds.
\end{remark}

\begin{remark}
For \textit{super Gaussian deviations}, we obtain a sharper bound than in Theorem  \ref{THM_COBORD}. 
Nonetheless, asymptotically, this regime is less sharp than Theorem 2 in \cite{hon:men:pag:16} which provides a Gaussian bound with deteriorated constants (see also  the \textit{User's guide to the proof} in Section \ref{USER_GUIDE} below).
Even if, from a numerical point of view, the deviation bounds in Theorem \ref{Exact_OPTIM} below yield sharper controls with respect to simulated empirical measures \textcolor{black}{(see Figure \ref{nappe_H16} in the numerical Section below).}

For ``intermediate Gaussian deviations", i.e. for $ \frac{a}{\sqrt{\Gamma_n}} \underset{n}{\rightarrow} C>0$, the constants in the Gaussian bound deteriorate. 
So, it seems reasonable to see this situation like for the first regime $ \frac{a}{\sqrt{\Gamma_n}} \underset{n}{\rightarrow} 0$, namely where there are constants $ C_\infty>1,  c_\infty <1$ such that $ \lim_n  C_n=  C_\infty$, $ \lim_n  c_n=  c_\infty$ and
\begin{equation*}
\P\big[ |\sqrt{\Gamma_n}\nu_n( \mathcal{A} \varphi )| \geq a \big] 
\leq 2\,  C_n \exp \big(\! -  c_n\frac{a^2}{2 \nu(|\sigma^*  \nabla \varphi |^2)}
\big). 
\end{equation*}
Observe that for such regimes, there is an equivalence, up to multiplicative constants, between the bounds in Theorem \ref{THM_CARRE_CHAMPS}
and in Theorem \ref{THM_CARRE_CHAMPS_super_Gaussian}.
\end{remark}

The idea of the proof of Theorem \ref{THM_CARRE_CHAMPS_super_Gaussian} follows the same lines as for Theorem \ref{THM_CARRE_CHAMPS}, except for the optimization over $ \rho$ which is more fussy, see Lemma \ref{RHO_optim} below.

We recall that the analysis \textcolor{black}{ in the proof of Theorem \ref{THM_CARRE_CHAMPS} leaving open a possible optimization over the parameter $\rho$ which we now perform}
The next lemma indicates that this optimization implies a Gaussian regime for $\frac {a}{\sqrt{\Gamma_n}} \underset{n}{\rightarrow} 0$ and a super Gaussian one for $\frac {a}{\sqrt{\Gamma_n}} \underset{n}{\rightarrow} + \infty$. 
\begin{lemme}[Choice of $\rho$ 
for the  concentration regime]\label{RHO_optim}
For 
$P_{\min}(a,\Gamma_n, \rho)$ as in \eqref{P_min_def},
\begin{enumerate}
\item[(a)] If $\frac{a}{\sqrt{\Gamma_n}} \underset{n}{\rightarrow} 0$, taking 
$\rho:=\rho(a,n)$ s.t.
$\rho-1= \frac{1}{2} \frac{\tilde B_n^{1/2} a}{\tilde A_n^{3/2} \sqrt{\Gamma_n}}(1+o(1))$
\begin{equation*}
P_{\min}(a,\Gamma_n, \rho) \underset{\frac{a}{\sqrt{\Gamma_n}} \underset{n}{\rightarrow} 0}{ =} - \frac{a^2}{2 \nu(| \sigma^*  \nabla \varphi |^2)}(1+ o(1 )).
\end{equation*}
\item[(b)] If $\frac{a}{\sqrt{\Gamma_n}} \underset{n}{\rightarrow} + \infty$, taking 
$\rho:=\rho(a,n)$ s.t.
$\rho-1= \frac{1}{2}+o(1)$
\begin{equation*}
P_{\min}(a,\Gamma_n, \rho) \underset{\frac{a}{\sqrt{\Gamma_n}} \rightarrow + \infty} {=} -\frac{a^{4/3} \Gamma_n^{1/3}}{2 \| \sigma \|_{\infty}^{2/3} [\vartheta]_1^{2/3} } (1+o(1)).
\end{equation*}
\end{enumerate}
\end{lemme}

\begin{remark}
For \textit{Gaussian deviations}, $\rho-1= \frac{1}{2} \frac{\tilde B_n^{1/2} a}{\tilde A_n^{3/2} \sqrt{\Gamma_n}}(1+o(1)) \asymp \frac{a}{ \sqrt{\Gamma_n}}$ which corresponds to our choice in Lemma \ref{RHO}. We then retrieve the Gaussian regime.
In the \textit{super Gaussian deviations} framework, the optimization over $\rho$ leads to consider $\rho-1= \frac{1}{2}+o(1)$ which yields the loss in the concentration inequality.
\end{remark}

Let us first continue with the proof of Lemma \ref{RHO_optim} which is purely analytical and rather independent of our probabilistic setting. 

\begin{proof}[Proof of Lemma \ref{RHO_optim} ]
We keep the notations of Lemma \ref{RHO}, that we bring to mind.
\begin{eqnarray*}
\Phi_n(a,\rho)&=& 
\Big( \frac{a}{\sqrt{\Gamma_n}\tilde B_n }+\big(\frac{a^2}{\tilde B_n^2\Gamma_n}+ (\rho-1)\big(\frac{2 \tilde A_n}{3\tilde B_n}\big)^3\big)^{\frac 12}\Big)^{\frac 13}+\Big( \frac{a}{\sqrt{\Gamma_n}\tilde B_n }-\big(\frac{a^2}{\tilde B_n^2\Gamma_n}+ (\rho-1)\big(\frac{2 \tilde A_n}{3\tilde B_n}\big)^3\big)^{\frac 12}\Big)^{\frac 13}
\nonumber \\
&=&  \frac{a^{1/3}}{\tilde B_n^{1/3} \Gamma_n^{1/6}} \big ( (1+ \sqrt{1+ \xi})^{1/3} + (1 - \sqrt{1+\xi})^{1/3} \big),
\end{eqnarray*}
for
\begin{equation}\label{CHOIX_RHO_bis}
\rho-1 := \xi \frac{27}{8} \frac{\tilde B_n a^2}{\tilde A_n^3 \Gamma_n},
\end{equation}
where $\xi:=\xi(a,n)>0$ is a parameter that we are going to optimize.
Furthermore:
\begin{equation}
P (\lambda_n )= - \frac{3^2}{2^4 \tilde A_n} a^2 f_{\Psi}(\xi),\label{PREAL_OPTIM_POLY}
\end{equation}
for 
\begin{equation}\label{function}
f_\Psi : \xi \in \R_+\longmapsto \frac{g(\xi)}{\Psi \xi+1},
\end{equation}
where 
\begin{equation*}
g : \xi \longmapsto \xi^{1/3}\big( (1+ \sqrt{1+\xi})^{1/3}+(1-\sqrt{1+\xi})^{1/3} \big )
\Big  (1- \frac{\xi^{1/3}}{2} \big ((1+ \sqrt{1+\xi})^{1/3}+(1-\sqrt{1+\xi})^{1/3} \big ) \Big  ),
\end{equation*}
and
\begin{equation}\label{Psi_def}
\Psi:=\frac{27}{8} \frac{\tilde B_n a^2}{\tilde A_n^3 \Gamma_n} \big (\overset{\eqref{CHOIX_RHO}} =\frac{\rho-1}{\xi} \big ).
\end{equation}
$\bullet$ Let us first focus on case \textit{(a)}, ``Gaussian deviations'' ($\frac{a}{\sqrt{\Gamma_n}} \underset{n}{\rightarrow} 0 $).
We, now anyhow, want to maximize $\Lambda $ in $\xi $  to obtain the best possible concentration bound. 
Let ${\mathscr A}:=\{ \xi \in [0,+\infty]: f_\Psi(\xi)=\|f_\Psi\|_{\infty}\} $ be the set of points where $f_\Psi$ reaches its maximum. 
Observe that for a fixed $\Psi $, $f_{\Psi}(\xi)\underset{\xi\to \infty}{\rightarrow}0$. Thus, $+\infty\not\in {\mathscr A} $.
Let now $\xi_* $ be an arbitrary point in ${\mathscr A} $.  
From the smoothness of $f_\Psi $, the optimality condition writes:
\begin{equation}\label{sup_trick}
f_\Psi'(\xi_*)=\frac{g'(\xi_*)}{(\Psi \xi_*+1)} - \frac{\Psi g(\xi_*)}{(\Psi \xi_*+1)^2} = 0 \ \Leftrightarrow
\frac{g'(\xi_*)}{(\Psi \xi_*+1)}  = \frac{\Psi f_\Psi (\xi_*)}{(\Psi \xi_*+1)}
 \ \Leftrightarrow 
f_\Psi(\xi_*)=\frac{g'(\xi_*)}{\Psi}.
\end{equation}
Recall now that we want to maximize over the $\xi $ s.t. $\xi \to +\infty,\ \xi\Psi \underset{\frac a{\sqrt \Gamma_n}\to 0}{\rightarrow} 0$. 
\textcolor{black}{Indeed, from the proof of Lemma \ref{RHO}, we saw that for such a choice, we obtain the expected Gaussian concentration, namely 
$P(\lambda_n) \underset{\frac{a}{\sqrt{\Gamma_n}} \underset{n}{\rightarrow} 0} {\sim }- \frac{a^2}{4\tilde A_n} $. }

From the computations of Lemma \ref{asymptotic_Dg} in Appendix \ref{asymptotic_analysis}, we have:
\begin{equation}\label{sup_equi}
g'(\xi) \underset{\xi\rightarrow + \infty}{=} 
 \frac{8}{3^5 \xi^2} + o(\frac {1}{\xi^2}) .
\end{equation}
 So, 
 \begin{equation}\label{LIMITE_CHOICE}
  \Lambda(\xi_*) \underset{\xi_* \to + \infty}=\frac{1}{2\cdot 3^3 \xi_*^{2} \Psi}(1+ o(1))\underset{\frac a{\sqrt \Gamma_n}\to 0}{\rightarrow}\frac{1}{4}  ,
  \end{equation}
 where $o(1)$ denotes here a quantity going to 0 as $\xi_*\rightarrow +\infty $ and $\xi_*\Psi\underset{\frac a{\sqrt \Gamma_n}\to 0}{\rightarrow} 0 $. 
 Inspired by the identity \eqref{LIMITE_CHOICE}, and motivated the numerical simulations (see Section \ref{SEC_NUM}),  we set
\begin{equation}\label{c_rho}
\bar \xi_* := \frac{2^{1/2}}{3^{3/2} \sqrt{\Psi}}= \frac{2^{2} \tilde A_n^{3/2} \sqrt{\Gamma_n}}{3^{9} \tilde B_n^{1/2} a},
\end{equation}
which indeed satisfies \eqref{COND}, so that 
$$\min_{\rho>1}P(\lambda_n)\le -\frac{a^2}{\tilde A_n}\Lambda(\bar \xi_*)(1+o(1))=\frac{a^2}{4\tilde A_n}(1+o(1))\underset{n \rightarrow + \infty}{\sim} -\frac{ a^2}{2 \nu(|\sigma^*  \nabla \varphi |^2)}.$$
Observe as well that this choice yields:
\begin{equation}
\rho-1=  \frac{2^{1/2} \Psi}{3^{3/2} \sqrt{\Psi}} =  \frac{1}{2} \frac{\tilde B_n^{1/2} a}{\tilde A_n^{3/2} \sqrt{\Gamma_n}}.
\end{equation}
$\bullet$ Now we will study the the case \textit{(b)} ``super Gaussian deviations'' for which $\frac{a}{\sqrt{\Gamma_n}} \underset{n}{\rightarrow} + \infty$.

Note that we cannot expect a Gaussian regime in this case. 
In fact, for $\Psi$ going to infinity (see definition \eqref{Psi_def}), by \eqref{sup_trick},  to get a Gaussian regime at a maximizer $\xi_* $ of $\Lambda $, we have from \eqref{PREAL_OPTIM_POLY} that $f_\Psi(\xi_*)=\frac{g'(\xi_*)}\Psi $ has to remain separated from 0 when $\frac{a}{\sqrt{\Gamma_n}} \to +\infty $. 
Since, in this case $\Psi \underset n \rightarrow + \infty$, this imposes to consider points $\xi \rightarrow 0$ in order to exploit the asymptotic behaviour (see again Lemma \ref{asymptotic_Dg} for more details): 
\begin{equation}\label{GP_PETIT}
g'(\xi) \underset{\xi \to 0}{=}
\frac{2^{1/3}}{3\xi^{\frac 23 }}\big(1+
 o(1)\big)
.
\end{equation}

 In other words, from \eqref{GP_PETIT}, we expect that there is a constant $K_*>0$ s.t.
\begin{equation*}
f_\Psi(\xi_*) \underset{\xi^* \to 0}{\sim} \frac{2^{1/3}}{3\xi_*^{2/3}\Psi} \ge  K_*.
\end{equation*}
So $\xi_* \le  \frac{C}{\Psi^{3/2}}\underset{\Psi\to +\infty}{\rightarrow}0$. Now, $|f_\Psi(\xi_*)|=|\frac{g(\xi_*)}{\Psi\xi+1}|\le |g(\xi_*)|\underset{\xi_*\to 0}{\rightarrow }0 $. 
This means that it is impossible to stay in a Gaussian regime. 

We now still look at the optimal $\xi_* \rightarrow 0$ which allows to stay at ``the biggest possible regime".
Thenceforth, we will estimate $\xi_*$ directly from the map $f_\Psi$ defined in \eqref{function}:
\begin{equation}\label{A_FAIRE}
f_\Psi(\xi) \underset{\xi \rightarrow 0}{=} \frac{2^{1/3}\xi^{1/3}}{ \Psi \xi +1} (1+o(1)) =: f_{\Psi,0}(\xi)(1+o(1)).
\end{equation}
It can be directly checked that $\argmax_{\xi \in \R_+} f_{\Psi,0}(\xi) = \frac{1}{2\Psi}$. We therefore get:
\begin{equation}\label{optim_inf_THEONE}
\xi_* \underset{\Psi \to + \infty}{=} \frac{1}{2\Psi} + o(\frac 1 \Psi)
\underset{\frac a{\sqrt{\Gamma_n}} \to + \infty}{=} \frac{4 \tilde A_n^3 \Gamma_n}{27 \tilde B_n a^2} + o(\frac {\Gamma_n}{a^2}) \to 0.
\end{equation}
From \eqref{PREAL_OPTIM_POLY} and \eqref{A_FAIRE}, we get:
\begin{eqnarray*}
\min_{\rho>1}P(\lambda_n) &\underset{\frac{a}{\sqrt{\Gamma_n}} \underset{n}{\rightarrow} + \infty} {\le } &
-\frac{3^2}{2^4 \tilde A_n} a^2 \frac{2^{1/3}\xi_*^{1/3}}{\Psi\xi_*+1}(1+o(1))
\underset{\eqref{optim_inf_THEONE}}{=}
- \frac{3^2}{2^4} \frac{a^2}{\tilde A_n ( \frac 12 +1) }\Psi^{-1/3}(1+o(1)) \\
&\underset{\eqref{COND}} \le & -\frac{3^2}{2^4} \frac{a^2}{\tilde A_n ( \frac 12 +1) } ( \frac{2^3 \tilde A_n^{3} \Gamma_n}{3^3 \tilde B_n a^2})^{1/3}(1+o(1))
\\
 &\underset{\frac{a}{\sqrt{\Gamma_n}} \underset{n}{\rightarrow} + \infty} {=} &
 -\frac{a^{4/3} \Gamma_n^{1/3}}{2^2 \tilde B_n ^{1/3}}(1 +o(1))
 \underset{\frac{a}{\sqrt{\Gamma_n}} \underset{n}{\rightarrow} + \infty} {\overset{\eqref{def_ABbar}}=} -\frac{a^{4/3} \Gamma_n^{1/3}}{2\| \sigma \|_{\infty}^{2/3} [ \vartheta ]_1^{2/3}}(1+o(1)).
\end{eqnarray*}
From equations \eqref{CHOIX_RHO} and \eqref{Psi_def}, the choice \eqref{optim_inf_THEONE} yields: 
\begin{equation}\label{optim_inf}
\rho-1= \xi \Psi=\frac 12 +o(1).
\end{equation}
\end{proof}
\begin{remark}[Controls of the optimized parameters $\lambda$ and $\rho$ for \textit{super Gaussian deviations}]\label{REM_RESTE_LAMBDA_GAMMA_N}
We give here some useful estimates to control the remainder terms in the truncation procedure associated with unbounded innovations, see proof of Lemma \ref{rn_choice}. 
They specify the behaviour of the quantity $\frac{\rho \lambda_n}{\Gamma_n} $.

In the regime of \textit{super Gaussian deviations}, identities \eqref{RLNG}, \eqref{Psi_def} and the choice \eqref{optim_inf_THEONE} (i.e. $\xi_* = \frac 1 {2 \Psi}$) yields:
\begin{eqnarray*}
&&\frac{\rho^2 \lambda_n^2}{\Gamma_n^2} \asymp
C \Psi
\underbrace{ \xi_*^{2/3} \big ( (1+ \sqrt{1+ \xi_*})^{1/3} + (1 - \sqrt{1+\xi_*})^{1/3} \big)^2}_{\asymp \frac{2^{1/3}}{\Psi^{2/3}}}
\underset{\frac{a}{\sqrt{\Gamma_n}} \rightarrow + \infty} {\rightarrow} +\infty.
\end{eqnarray*}

\end{remark}

\begin{remark}\label{RK_interet_optim_rho_Gauss_superGauss}
Theorems \ref{THM_CARRE_CHAMPS} and \ref{THM_CARRE_CHAMPS_super_Gaussian} are actually a consequence of the more general following result which has a real  importance for numerical  applications. 
Indeed, for a given $n \in \N$, we have to control the non-asymptotic error in Theorems \ref{THM_CARRE_CHAMPS} and \ref{THM_CARRE_CHAMPS_super_Gaussian} . Furthermore, in Section \ref{SEC_NUM} (see Remark  \ref{super_Gaussian_numerical}) we will see that for $\theta \in (\frac 13,1]$ ,  for ``reasonable" $n$ (e.g. $n=5 \cdot 10^4$ in the following Section \ref{SEC_NUM}) and $a$ ($ \approx 1$) we are  already ``out" of the \textit{Gaussian deviations} regime, namely $\frac a{\sqrt{\Gamma_n}} \asymp1$. 
This illustration justifies the interest of optimizing over $\rho$ for  \textit{Gaussian deviations}  and  \textit{super Gaussian deviations}.
\end{remark}

\begin{theo}\label{Exact_OPTIM}
Let the assumptions of Theorem \ref{THM_CARRE_CHAMPS} be in force.
For $ \theta \in (\frac{1}{3}, 1 ]$
, there exist explicit non-negative sequences 
$( c_n)_{n\ge 1}$ and $( C_n)_{n\ge 1}$, respectively increasing and decreasing for $n$ large enough, with $\lim_n  C_n = \lim_n  c_n =1$ 
s.t. for all $n\ge 1$ for all $a>0$,
\begin{equation*}
\P\big[ |\sqrt{\Gamma_n}\nu_n( \mathcal{A} \varphi )| \geq a \big] 
\leq 2\,  C_n \exp \big(\!  c_nP_{\min}(a,\Gamma_n, \rho )
\big),
\end{equation*}
where  $\rho>1 $ and 
\begin{equation*}
P_{\min}\big(a, \Gamma_n,\rho \big)= - \frac{(\rho-1)^{1/3}}{\rho} \frac{\sqrt{\Gamma_n} \Phi_n(a,\rho)}{8} ( 3a -  \sqrt{\Gamma_n}  (\rho-1)^{1/3} \tilde A_n \Phi_n(a,\rho)),
\end{equation*}
with 
\begin{equation}
\quad\Phi_n(a,\rho):=
\Big( \frac{a}{\sqrt{\Gamma_n}\tilde B_n }+\big(\frac{a^2}{\tilde B_n^2\Gamma_n}+ (\rho-1)\big(\frac{2 \tilde A_n}{3\tilde B_n}\big)^3\big)^{\frac 12}\Big)^{\frac 13}+\Big( \frac{a}{\sqrt{\Gamma_n}\tilde B_n }-\big(\frac{a^2}{\tilde B_n^2\Gamma_n}+ (\rho-1)\big(\frac{2 \tilde A_n}{3\tilde B_n}\big)^3\big)^{\frac 12}\Big)^{\frac 13}
\end{equation}
\begin{equation}\label{eq:ABtilde}
\widetilde A_n:= \frac{ q\nu(| \sigma^*  \nabla \varphi |^2)}{2}+e_n
\quad\mbox{ and }\quad 
\widetilde B_n = \frac{q^3 \hat q}{4}\Big( \frac{\bar q\|\sigma \|_{\infty}^2 \| \nabla \vartheta\|_{\infty}^2}{2}
+ e_n\Big),
\end{equation}
where $ \ q:=q(n)>1,\ \bar q:=\bar q(n)>1$\textcolor{black}{, $\hat{q}:=\hat{q}(n)$} with 
\textcolor{black}{$q,\bar q\hat{q} \underset n\to 1$} and $e_n$ is an explicit sequence going to $0$.

The results of Theorem \ref{THM_CARRE_CHAMPS} explicitly follow taking (see for more details the proof of Lemma \ref{RHO_optim}):
\begin{eqnarray*}
\rho-1&=& \frac{1}{2} \frac{\tilde B_n^{1/2} a}{\tilde A_n^{3/2} \sqrt{\Gamma_n}}
, \ \text{for} \ \frac{a}{\sqrt{\Gamma_n}} \underset{n}{\rightarrow} 0,
\\
\rho-1&=& \frac{1}{2 }, \ \text{for} \ \frac{a}{\sqrt{\Gamma_n}} \rightarrow+ \infty.
\end{eqnarray*}
\end{theo}

\textcolor{black}{
\begin{remark}\label{remark_bootstrap}
It is natural to wonder if it is possible to get a sharper variance for \textit{super Gaussian deviations}.
Thereby it would be  tempting to bootstrap Lemma \ref{Mn}.
Such an iteration would lead to optimize polynomials of higher degrees.
Recall from Lemma \ref{RHO}, that we already have to handle a polynomial of order $4$ in the current setting.
This illustrates that for very large deviations the highest term dominates and deteriorates the concentration.
This is intrinsic to our approach.
Such a phenomenon would even more pregnant when iterating the procedure the polynomial of higher degree yields, for a certain very large deviation regime, concentration bounds that become closer and closer to the exponential.
However, bootstrapping might allow \textcolor{black}{to improve the constants in the successive deteriorated concentration regimes.
As indicated in  Remark \ref{RK_interet_optim_rho_Gauss_superGauss}, this could be useful for some numerical purposes.}
\end{remark}
}

\section{Numerical Results}\label{SEC_NUM}
\setcounter{equation}{0}
\subsection{Degenerate diffusion}

Here, we have chosen to highlight the possible absence of non-degene\-ra\-cy assumption for our results. 
\textcolor{black}{To oversimplify, simulations are done with $r=d=1$, $X_0$ and $U_1$ follow the standard normal distribution. }
Naturally, for a better convergence speed, we take $\theta \approx \frac 13$, precisely $\theta=\frac13 + \frac1{1000}$.

For this first example, we choose $\varphi$ (solution of the Poisson equation).
We take $\varphi=\sigma =\cos$ and for all $x \in \R$, $b(x)= -\frac x2$.
By this pick, we compute numerically $\nu(|\sigma^* \nabla \varphi|^2) \approx 0.1515$ and $\nu(|\sigma|^2) \approx 0.4171$ \textcolor{black}{ (that we provide here for comparison with the previous results in \cite{hon:men:pag:16})}, with the same parameters ($\theta= \frac 13 + \frac 1{1000}$, $n=5 \cdot 10^4$ and $MC=10^4$).
\\

Heed, for a non trivial test function $\varphi$, with our method, we cannot choose functions $b$ and $\sigma \neq 0$ canceling at the same point (0 here). 
Otherwise, the Poisson equation associated with the \textit{carré du champ} source, ${\mathcal A} \vartheta=|\sigma^*\nabla \varphi|^2- \nu(|\sigma^*\nabla \varphi|^2)$, would imply that $-\nu(|\sigma^* \nabla \varphi |^2)=0$, then $\nabla \varphi=0, \ \nu$
almost surely.
\\

Let us now check that the Confluence Conditions \A{D${}_\alpha^p$} are satisfied.
For $ p\in [1,2)$, we have for all $x\in \R^d $, $\xi\in \R^d $
\begin{eqnarray}
&&\Big \langle \frac {Db(x)+Db(x)^*} 2\xi,\xi \Big\rangle 
+\frac 12 \sum_{j=1}^r \Big( (p-2)\frac{|\langle  D\sigma_{\cdot j}(x) \xi, \xi\rangle|^2}{|\xi|^2}+|D\sigma_{\cdot j} \xi|^2
\Big)
\nonumber \\
&&=
-\frac 12\xi^2 
+\frac 12 \sin^2(x)
\xi^2  ( p-1 ).
\end{eqnarray}
So, for $p=\frac 32$, we directly obtain:
\begin{eqnarray}
&&\Big \langle \frac {Db(x)+Db(x)^*} 2\xi,\xi \Big\rangle 
+\frac 12 \sum_{j=1}^r \Big( (\frac 32-2)\frac{|\langle  D\sigma_{\cdot j}(x) \xi, \xi\rangle|^2}{|\xi|^2}+|D\sigma_{\cdot j} \xi|^2
\Big)
\nonumber \\
&&= -\frac 12 \xi^2
+\frac 14 \sin^2(x)
\xi^2
\leq  -\frac 14 \xi^2
=:-\alpha \xi^2.
\end{eqnarray}
\textcolor{black}{Note that, we have chosen a diffusion coefficient $\sigma$ which degenerates on $\{k \pi, k \in \Z \}$.
However, thanks to the smoothness of the diffusion parameters, we can still here apply Lemma 6 in \cite{hon:men:pag:16} which gives us a pointwise gradient bound of the solution of the Poisson problem in the current degenerate context.
}
In other words:
\begin{equation*}
[\vartheta ]_1 \leq \frac {[|\sigma^* \nabla \varphi|^2]_1}\alpha= 4[\cos^2 \sin^2]_1= 4 \sup_{x \in \R}(\cos(2x) \sin(2x))= 2.
\end{equation*}
Hence, this inequality leads us to approximate $ [\vartheta]_1 $ by $ 2$.
\textcolor{black}{Pay attention that the control of the Lipschitz constant $[\vartheta]_1$ is important for the \textit{super Gaussian deviations}.
Like illustrated in Remarks \ref{RK_interet_optim_rho_Gauss_superGauss} and \ref{super_Gaussian_numerical}, this regime appears ``sooner" than we might expect. 
}
\\

From Theorem~\ref{THM_CARRE_CHAMPS}, the function $a \mapsto g_{n}(a):=\ln(\P[ \sqrt{\Gamma_n} | \nu_n(\mathcal{A}\varphi) | \geq a ])$ is s.t. for $a>0$:
$$g_{n}(a)\le -c_n \frac{a^2}{2 \nu(|\sigma^*\nabla \varphi|^2)}+\ln(2C_n),$$ 
where $( c_n)_{n\ge 1}$ and $( C_n)_{n\ge 1}$ are sequences respectively increasing and decreasing for $n$ large enough, with $\lim_n  C_n = \lim_n c_n =1$.
\\

For  Figure \ref{nappe_H16}, the simulations have been performed for $n=5 \cdot 10^4$ and the probability estimated by Monte Carlo simulation for $MC=10^4$ realizations of the random variable $\sqrt{\Gamma_n} | \nu_n(\mathcal{A}\varphi) | $. 
The corresponding $95\% $ confidence intervals have size at most of order $0.0016$. 
We introduce the functions:
$$S (a):=-\frac{a^2}{2 \nu(|\sigma ^*\nabla \varphi|^2)}, \ S_{\sup} (a):=-\frac{a^2}{2 \|  \sigma \|_{\infty}^2  \|  \nabla \varphi \|_{\infty}^2}. $$
Like in Theorem \ref{Exact_OPTIM}, we take
$$P_{\min}(a,\Gamma_n,\rho)=- \frac{(\rho-1)^{1/3}}{\rho} \frac{\sqrt{\Gamma_n} \Phi_n(a,\rho)}{8} ( 3a - (\rho-1)^{1/3}  \tilde A_n \Phi_n(a,\rho) \sqrt{\Gamma_n} ),$$ where $\tilde A_n, \tilde B_n$ and $\Phi_n(a,\rho) $ are defined in \eqref{eq:ABtilde}.
Through our numerical results, we take $e_n=0$.

We set also: 
 \begin{equation*}
\rho_0 := 1+ \frac{1}{2} \frac{\bar B_n^{1/2} a}{\bar A_n^{3/2} \sqrt{\Gamma_n}} , \ 
\rho_ \infty := \frac{3}{2}.
\end{equation*}
We recall here that $\rho$ and $\rho_\infty$ respectivly correspond to the optimal values of $\rho$ in the \textit{Gaussian deviations} and \textit{super Gaussian deviations} (see Lemma \ref{RHO_optim}).
Eventually, we introduce:
 \begin{eqnarray*}
 P_{n,0,\infty}(a) &:=& \min \Big (  P_{\min}(a,\Gamma_n, \rho_0), P_{\min}(a,\Gamma_n, \rho_\infty) \Big ),
 \\
 P_n(a)&:=& \min_{\rho>1} P_{\min}(a,\Gamma_n,\rho).
\end{eqnarray*}
\textcolor{black}{Note that, the function $P_{n,0, \infty}$ takes into account the multi-regime competition.}
We have estimated $P_n(a)$ by a mesh method for $\rho \in (1,2)$ and for a grid with $5 \cdot 10^5$ steps.
\\

From the above notations, we add the subscript $\sigma$ to mean that we change $\nu(|\sigma^* \nabla \varphi|^2)$ into $\|\nabla \varphi\|_{\infty}^2 \nu(\|\sigma\|^2)$, i.e.
$$S_{\sigma}(a) =- \frac{a^2}{2\|\nabla \varphi\|_{\infty}^2 \nu(\|\sigma\|^2)}, \
 P_{n,\sigma}(a):= \min_{\rho>1} P_{\min,\sigma}(a,\Gamma_n, \rho),$$
and we have changed $\tilde A_n$ into 
\begin{equation*}
\tilde A_{\sigma,n}:= \frac{\|\nabla \varphi\|_{\infty}^2 \nu(\| \sigma\|^2)}{2}.
\end{equation*}
\textcolor{black}{The quantities with subscript $\sigma$ are those associated with the results in \cite{hon:men:pag:16}, recalled in the previous Theorem \ref{THM_COBORD}, where the variance is less sharp than the constants appearing in Theorems \ref{THM_CARRE_CHAMPS}, \ref{THM_CARRE_CHAMPS_super_Gaussian} and \ref{Exact_OPTIM}.
}
Thus, we can compare our main results with Remark 10 of \cite{hon:men:pag:16} which is a weakened form of Theorem \ref{Exact_OPTIM} where the \textit{carré du champ} is changed into $\|\nabla \varphi\|_{\infty}^2 \nu(\|\sigma\|^2)$ like in Theorem \ref{THM_COBORD}.

\begin{figure}[!h]
\centering
\begin{minipage}[b]{.46\linewidth}
\centerline{\includegraphics[scale=1.25]{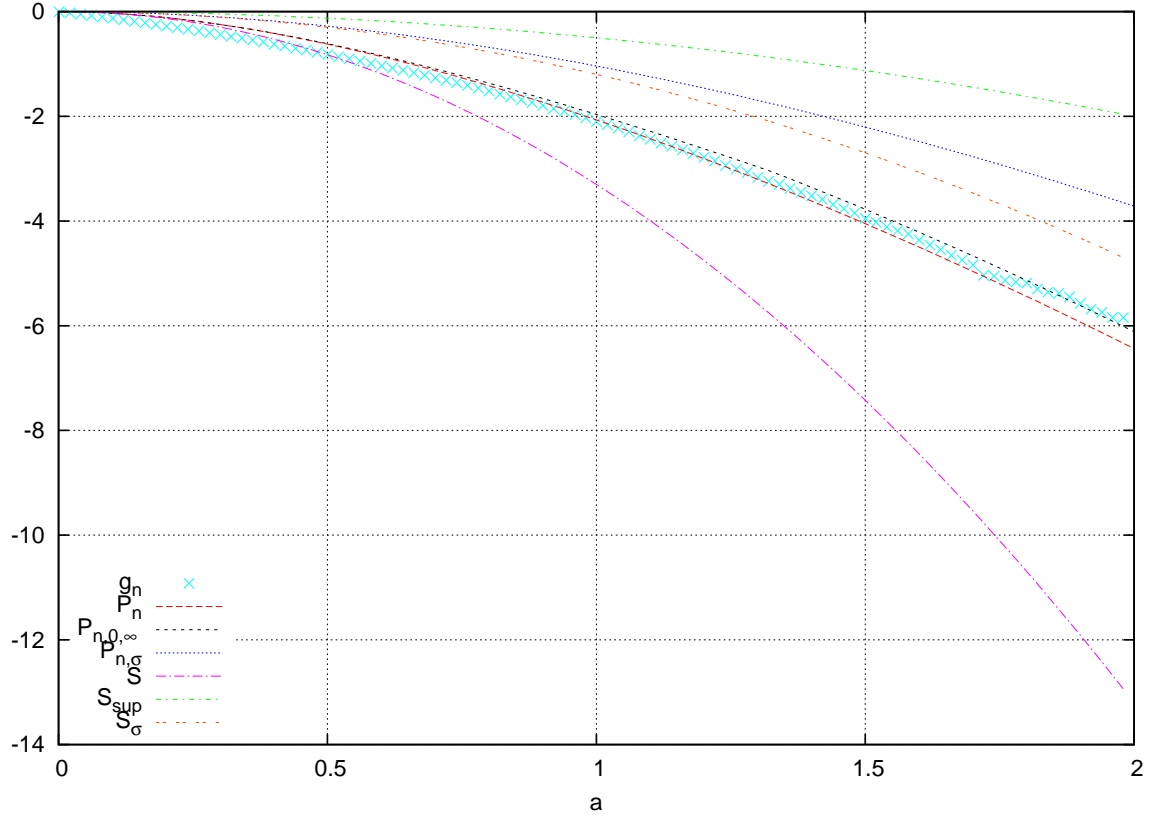}}
\caption{Plot of $a \mapsto g_{n}(a)$ with $\varphi(x)\!=\!  \sigma(x) = \cos(x)$.}
\label{nappe_H16}
\end{minipage}
\end{figure}
\vspace{1cm}

Figure \ref{nappe_H16} reveals that the asymptotic curve $S$ is much less sharp with respect to the realizations $g_n$ than our main estimations $P_n$ and $P_{n,0,\infty}$. In fact, these latter are very close to the realization $g_n$.
This claim enhances the significance of controlling finely, non-asymptotically, the deviation of the empirical measure.

In this plot, we can see that our pick of $\rho$ for $P_{n,0,\infty}$, set in Lemma \ref{RHO_optim}, is very close to the numerical optimization of $P_n$ over $\rho$.
Nevertheless, observe that for $a>0.5$, $P_{n,0,\infty}(a)$ and $P_n(a)$ slightly differ. It means that progressively the regime goes from \textit{Gaussian deviations} (i.e. $\frac a{\sqrt{\Gamma_n}} \to 0$) to \textit{intermediate Gaussian deviations} (i.e. $\frac a{\sqrt{\Gamma_n}} = O(1)$). 
Hence, the importance of optimizing globally the function $\rho \mapsto P_{\min}(a,\Gamma_n,\rho)$ (appearing in \eqref{ineq_PMIN}) in all regimes.

\textcolor{black}{\begin{remark}\label{super_Gaussian_numerical}
Remark that for the graphic \ref{nappe_H16}, we chose $n=5 \cdot 10^4$, but for $\theta \approx \frac 13$, $ \sqrt{\Gamma_n} \approx 37$ and for $\theta \approx \frac 1{2+0.5}$, $ \sqrt{\Gamma_n} \approx 26$.
In other words, for $a \approx 1$ we have \textit{intermediate Gaussian deviations} as emphasized by the graphic.
Hence the importance of the study of both regimes, \textit{Gaussian deviations} ($\frac a{\sqrt{\Gamma_n}} \to 0$) and \textit{super Gaussian deviations} ($\frac a{\sqrt{\Gamma_n}} \to +\infty$).
\end{remark}}

\vspace{5cm}

\appendix 

\mysection{Computation of asymptotic analysis}\label{asymptotic_analysis}

In this section, we perform asymptotic analysis for the map $g'$ defined in \eqref{function} in proof of Lemma \ref{RHO}. We recall that for all $\xi \in \R$:

$g(\xi)= \xi^{1/3}\Big( (1+ \sqrt{1+\xi})^{1/3}+(1-\sqrt{1+\xi})^{1/3} \Big )\Big (1- \frac{\xi^{1/3}}{2} \big ( (1+ \sqrt{1+\xi})^{1/3}+(1-\sqrt{1+\xi})^{1/3} \big ) \Big ).$

\begin{lemme}\label{asymptotic_Dg}
\begin{equation*}
g'(\xi) \underset{\xi \to 0}{=}
\frac{2^{1/3}}{3\xi^{\frac 23 }}\big(1
+ o(1)\big)
, \ \ g'(\xi) \underset{\xi\rightarrow + \infty}{=} 
 \frac{8}{3^5 \xi^2} + o(\frac {1}{\xi^2}) .
\end{equation*}
\end{lemme}

\begin{proof}
Denote $h(\xi) := \xi^{1/3} \left ( (1+ \sqrt{1+ \xi})^{1/3} + (1- \sqrt{1+ \xi})^{1/3} \right)$, so $g(\xi)= h(\xi) (1- \frac{h(\xi)}{2})$.
Differentiating, we get:
\begin{eqnarray*}
h'(\xi)&=&{{(1-\sqrt{1+\xi})^{{{1}\over{3}}}+(1+\sqrt{1+\xi}
)^{{{1}\over{3}}}}\over{3\,\xi^{{{2}\over{3}}}}}
\nonumber \\
&&+\xi^{{{1}\over{3
}}}\,({{1}\over{6\,\sqrt{1+\xi}\,(1+\sqrt{1+\xi})^{{{2
}\over{3}}}}}-{{1}\over{6\,\sqrt{1+\xi}\,(1-\sqrt{1+\xi})^{{{
2}\over{3}}}}})
\\
&=& {{\left(1-\sqrt{1+\xi}\right)^{{{1}\over{3}}}+\left(1+\sqrt{1+\xi}
\right)^{{{1}\over{3}}}}\over{3\,\xi^{{{2}\over{3}}}}}
+ \frac{1}{6 \sqrt{1+\xi}} \frac{ (1- \sqrt{1+\xi})^{2/3} - (1+ \sqrt{1+\xi})^{2/3} }{\xi^{2/3}}.
\end{eqnarray*}

\textit{(a)} For $\xi \to 0$,
\begin{equation*}
h(\xi) \underset{\xi \to 0}{=}
2^{1/3} \xi^{1/3}+ o(\xi^{1/3}),
\end{equation*}
and 
\begin{equation*}
h'(\xi) \underset{\xi \to 0}{=}
\frac{2^{1/3}}{3\xi^{\frac 23 }}
+\xi^{{{1}\over{3 }}}\,( \frac{1}{6 \times 2^{2/3}}
- \frac{2^{2/3}}{\xi^{2/3}}
)+ o(\frac{1}{\xi^{2/3}}) = \frac{2^{1/3}}{3\xi^{\frac 23 }}
+ o(\frac{1}{\xi^{2/3}}),
\end{equation*}
which yields that
\begin{equation*}
g'(\xi)=
h'(\xi) \big (1-h(\xi) \big)
\underset{\xi \to 0}{=}
\frac{2^{1/3}}{3\xi^{\frac 23 }} - 
+ o(\frac{1}{\xi^{2/3}}).
\end{equation*}

\textit{(b)} For $\xi \to + \infty$,

In order to estimate $g'$ we need to do a Taylor expansion up to the third order:
\begin{eqnarray*}
h(\xi) &\underset{\xi \to + \infty}{=}& \xi^{1/3}(1+\xi)^{1/6} \big  ( (1+\frac{1}{\sqrt{1+\xi}})^{1/3} - (1-\frac{1}{\sqrt{1+\xi}})^{1/3} \big)
\\
&=& \xi^{1/3} (1+\xi)^{1/6} \Big (1+ \frac{1}{3\sqrt{1+\xi}} - \frac{1}{3^2(1+ \xi)} + \frac{2 \times 5}{3^3 3! (1+\xi)^{3/2}}
\\
&&-(1- \frac{1}{3\sqrt{1+\xi}} - \frac{1}{3^2(1+ \xi)} - \frac{2 \times 5}{3^3 3! (1+\xi)^{3/2}} ) + o(\frac{1}{\xi^ {3/2}}) \Big )
\\ 
&=& \xi^{1/2} \big(1+\frac{1}{6\xi}+ o(\frac{1}{\xi}) \big ) \big ( \frac{2}{3\sqrt{1+\xi}} + \frac{10}{3^4(1+\xi)^{3/2}} + o(\frac{1}{\xi^{3/2}}) \big )
\\ 
&=& \xi^{1/2} \big(1+\frac{1}{6\xi}+ o(\frac{1}{\xi}) \big ) \big ( \frac{2}{3\sqrt{\xi}} - \frac{1}{3\xi^{3/2}} + \frac{10}{3^4(1+\xi)^{3/2}} + o(\frac{1}{\xi^{3/2}}) \big )
\\ 
&=& \xi^{1/2} \big ( \frac{2}{3\sqrt{\xi}} + \frac{1}{3^2\xi^{3/2}} - \frac{1}{3\xi^{3/2}} + \frac{10}{3^4(1+\xi)^{3/2}} + o(\frac{1}{\xi^{3/2}}) \big )
\\ 
&=& \frac 23 - \frac {8}{81 \xi}+ o(\frac 1 \xi).
\end{eqnarray*}
Differentiating the above expression, we get:
\begin{equation*}
h'(\xi) \underset{\xi \to + \infty}{=} \frac{8}{81 \xi^2} + o(\frac{1}{\xi^2}),
\end{equation*}
which yields that
\begin{eqnarray*}
g'(\xi)&=&h'(\xi) \big (1-h(\xi)\big)
\underset{\xi \to 0}{=}
\big ( \frac{8}{81 \xi^2} + o(\frac{1}{\xi^2})\big) \big( 1- \frac23 + \frac {8}{81 \xi} + o(\frac 1 \xi)\big )
\\
&=& \frac{16}{3^5\xi^2} -\frac{8}{3^5 \xi^2} + o(\frac {1}{\xi^2}) = \frac{8}{3^5 \xi^2} + o(\frac {1}{\xi^2}) .
\end{eqnarray*}
\end{proof}

\section*{Acknowledgments}
The author would like to warmly express his gratitude towards Stéphane MENOZZI 
 for his advice and his support which were determinant for this work.

\bibliographystyle{alpha}
\bibliography{bibli}

\end{document}